\documentclass[3p,times]{elsarticle}

\usepackage{amssymb}
\usepackage{amsfonts}
\usepackage{amsmath}
\usepackage{amsthm}
\usepackage{graphicx}
\usepackage{mathrsfs}
\usepackage{dsfont}
\usepackage{amscd}
\usepackage{multirow}
\normalfont
\usepackage[T1]{fontenc}
\usepackage{calligra}
\usepackage{verbatim}
\usepackage[usenames,dvipsnames]{color}
\usepackage[colorlinks=true,linkcolor=Blue,citecolor=Violet]{hyperref}

\newcommand{\N}{{\mathds{N}}}
\newcommand{\Z}{{\mathds{Z}}}
\newcommand{\Q}{{\mathds{Q}}}
\newcommand{\R}{{\mathds{R}}}
\newcommand{\C}{{\mathds{C}}}
\newcommand{\T}{{\mathds{T}}}

\newcommand{\D}{{\mathfrak{D}}}
\newcommand{\A}{{\mathfrak{A}}}
\newcommand{\B}{{\mathfrak{B}}}

\newcommand{\bigslant}[2]{{\raisebox{.2em}{$#1$}\left/\raisebox{-.2em}{$#2$}\right.}}
\newcommand{\Nbar}{\overline{\N}_\ast}
\newcommand{\Lip}{{\mathsf{L}}}
\newcommand{\Hilbert}{{\mathscr{H}}}

\newcommand{\dist}{{\mathsf{dist}}}

\newcommand{\propinquity}{{\mathsf{\Lambda}}}
\newcommand{\dpropinquity}[1]{{\mathsf{\Lambda}^\ast_{#1}}}

\newcommand{\Kantorovich}[1]{{\mathsf{mk}_{#1}}}

\newcommand{\Haus}[1]{{\mathsf{Haus}_{#1}}}

\newcommand{\StateSpace}{{\mathscr{S}}}

\newcommand{\mongekant}{{Mon\-ge-Kan\-to\-ro\-vich metric}}
\newcommand{\qms}{quantum locally compact metric space}

\newcommand{\Lqcms}{{\JLL} quantum compact metric space}
\newcommand{\qcms}{quantum compact metric space}

\newcommand{\unit}{1}

\newcommand{\sa}[1]{{\mathfrak{sa}\left({#1}\right)}}

\newcommand{\Adm}{{\mathrm{Adm}}}

\newcommand{\JLL}{Lei\-bniz}

\newcommand{\dom}[1]{{\operatorname*{dom}({#1})}}
\newcommand{\diam}[2]{{\mathrm{diam}\left({#1},{#2}\right)}}

\newcommand{\tunnelset}[3]{{\text{\calligra Tunnels}\,\left[\left({#1}\right)\stackrel{#3}{\longrightarrow}\left({#2}\right)\right]}}
\newcommand{\journeyset}[3]{{\text{\calligra Journeys}\left[\left({#2}\right)\stackrel{#1}{\longrightarrow}\left({#3}\right)\right]}}
\newcommand{\reversejourneyset}[3]{{\text{\calligra Hom}\left[\left({#2}\right)\stackrel{#1}{\longrightarrow}\left({#3}\right)\right]}}

\newcommand{\bridgelength}[2]{{\lambda\left({#1}\middle|{#2}\right)}}
\newcommand{\treklength}[1]{{\lambda\left({#1}\right)}}
\newcommand{\trekset}[2]{{\text{\calligra Treks}\left(\left({#1}\right)\longrightarrow\left({#2}\right)\right)}}

\newcommand{\bridgenorm}[2]{{\mathsf{bn}_{ {#1}  }\left({#2}\right)}}

\newcommand{\lambdaitineraries}[4]{{\text{\calligra Itineraries}\left( {#2}\stackrel{{#1} }{\longrightarrow}{#3}\middle|{#4}  \right)}}
\newcommand{\Jordan}[2]{{{#1}\circ{#2}}} 
\newcommand{\Lie}[2]{{\left\{{#1},{#2}\right\}}} 

\newcommand{\targetsettunnel}[3]{{\mathfrak{t}_{#1}\left({#2}\middle\vert{#3}\right)}}
\newcommand{\targetsetjourney}[3]{{\mathfrak{T}_{#1}\left({#2}\middle\vert{#3}\right)}}
\newcommand{\LQCMS}{{\mathcal{L}^\ast}}

\newcommand{\tunneldepth}[2]{{\delta\left(#1\right)}}
\newcommand{\tunnelreach}[2]{{\rho\left(#1\right)}}
\newcommand{\tunnellength}[2]{{\lambda\left(#1\right)}}
\newcommand{\journeylength}[1]{{\lambda\left(#1\right)}}
\newcommand{\prox}{{\mathrm{prox}}}
\newcommand{\co}[1]{{\overline{\mathrm{co}}\left(#1\right)}}
\newcommand{\alg}[1]{{\mathfrak{#1}}}

\theoremstyle{plain}
\newtheorem{theorem}{Theorem}[section]

\newtheorem{corollary}[theorem]{Corollary}

\newtheorem{lemma}[theorem]{Lemma}
\newtheorem{proposition}[theorem]{Proposition}

\newtheorem{theorem-definition}[theorem]{Theorem-Definition}

\theoremstyle{definition}
\newtheorem{definition}[theorem]{Definition}

\newtheorem{notation}[theorem]{Notation}

\newtheorem{convention}[theorem]{Convention}
\newtheorem{hypothesis}[theorem]{Hypthesis}

\theoremstyle{remark}

\newtheorem{example}[theorem]{Example}

\newtheorem{remark}[theorem]{Remark}

\renewcommand{\geq}{\geqslant}
\renewcommand{\leq}{\leqslant}

\numberwithin{equation}{section}

\begin{document}

\title{The Dual Gromov-Hausdorff Propinquity}
\author{Fr\'{e}d\'{e}ric Latr\'{e}moli\`{e}re}
\ead{frederic@math.du.edu}
\ead[url]{http://www.math.du.edu/\symbol{126}frederic}
\address{Department of Mathematics \\ University of Denver \\ Denver CO 80208}

\date{\today}

\begin{abstract}
Motivated by the quest for an analogue of the Gromov-Hausdorff distance in noncommutative geometry which is well-behaved with respect to C*-algebraic structures, we propose a complete metric on the class of {\Lqcms s}, named the dual Gromov-Hausdorff propinquity. This metric resolves several important issues raised by recent research in noncommutative metric geometry: it makes *-isomorphism a necessary condition for distance zero, it is well-adapted to Leibniz seminorms, and --- very importantly --- is complete, unlike the quantum propinquity which we introduced earlier. Thus our new metric provides a natural tool for noncommutative metric geometry, designed to allow for the generalizations of techniques from metric geometry to C*-algebra theory.

\bigskip

{\noindent \textsc{R{\'e}sum{\'e}:} Motiv{\'e}s par la qu{\^e}te d'une m{\'e}trique analogue {\`a} la distance de Gromov-Hausdorff pour la g{\'e}om{\'e}trie noncommutative et adapt{\'e}e aux C*-alg{\`e}bres, nous proposons une distance compl{\`e}te sur la classe des espaces m{\'e}triques compacts quantiques de Leibniz. Cette nouvelle distance, que nous appelons la proximit{\'e} duale de Gromov-Hausdorff, r{\'e}sout plusieurs probl{\`e}mes importants que la recherche courante en g{\'e}om{\'e}trie m{\'e}trique noncommutative a r{\'e}v{\'e}l{\'es}. En particulier, il est n{\'e}cessaire pour les C*-alg{\`e}bres d'{\^e}tre isomorphes pour avoir distance z{\'e}ro, et tous les espaces quantiques compacts impliqu{\'e}s dans le calcul de la proximit{\'e} duale sont de type Leibniz. En outre, notre distance est compl{\`e}te. Notre proximit{\'e} duale de Gromov-Hausdorff est donc un nouvel outil naturel pour le d{\'e}veloppement de la g{\'e}om{\'e}trie m{\'e}trique noncommutative.}
\end{abstract}
\maketitle



\section{Introduction}

Noncommutative metric geometry proposes to study certain classes of noncommutative algebras as generalizations of algebras of Lipschitz functions over metric spaces. Inspired by Connes' pioneering work on noncommutative metric geometry \cite{Connes89,Connes} and, in particular, the construction of a metric on the state space of C*-algebras endowed with a spectral triple, Rieffel introduced in \cite{Rieffel98a,Rieffel99} the notion of a compact quantum metric space, and then defined the quantum Gromov-Hausdorff distance \cite{Rieffel00}, a fascinating generalization of the Gromov-Hausdorff distance \cite{Gromov81} to noncommutative geometry. Various examples of compact quantum metric spaces \cite{Rieffel02,Ozawa05,Li09b,Li05} and convergence results for the quantum Gromov-Hausdorff distance have since been established \cite{Rieffel00,Li05,Latremoliere05,Rieffel01,Li09b,Rieffel09,Rieffel10,Rieffel10c}, often motivated by the desire to provide a formal framework to certain finite dimensional approximations of C*-algebras found in the mathematical physics literature (e.g. \cite{Connes97,Seiberg99,tHooft02,Zumino98}). Furthermore, the introduction of noncommutative metric information on C*-algebras, encoded in special semi-norms called Lip-norms, offers the possibility to extend techniques from metric geometry \cite{Gromov} to the realm of noncommutative geometry, opening a new avenue for the study of quantum spaces and their applications to mathematical physics.

To implement the extension of metric geometry to noncommutative geometry, we however require an analogue of the Gromov-Hausdorff distance which is well-behaved with respect to the underlying C*-algebraic structure, rather than with only the order structure on the self-adjoint part of C*-algebras, as with the quantum Gromov-Hausdorff distance. We propose such a metric in this paper, the dual Gromov-Hausdorff propinquity, which addresses several difficulties encountered during recent developments in noncommutative metric geometry \cite{Rieffel09,Rieffel10,Rieffel10c,Rieffel11,Rieffel12}, where the study of the behavior of such C*-algebraic structures as projective modules under metric convergence is undertaken. Indeed, our metric only involves {\Lqcms s}, i.e. quantum compact metric spaces described as C*-algebras endowed with Leibniz Lip-norms, and makes *-isomorphism of the underlying C*-algebras a necessary condition for distance zero. Moreover, our dual Gromov-Hausdorff propinquity is complete, which is an essential property of the Gromov-Hausdorff distance, and which differentiates our new metric from our earlier quantum Gromov-Hausdorff propinquity \cite{Latremoliere13}. Our dual Gromov-Hausdorff propinquity dominates Rieffel's quantum Gromov-Hausdorff distance, and is dominated by the Gromov-Hausdorff distance when restricted to classical compact metric spaces. It thus offers a natural framework for a noncommutative theory of quantum metric spaces.

The model for a quantum compact metric space is derived from the following construction. Let $(X,\mathsf{m})$ be a compact metric space. For any function $f: X\rightarrow\C$, we define the Lipschitz constant of $f$ as:
\begin{equation}\label{Lip-eq}
\mathsf{Lip}(f) = \sup \left\{\frac{|f(x)-f(y)|}{\mathsf{m}(x,y)} : x,y\in X\text{ and }x\not=y \right\}\text{,}
\end{equation}
which may be infinite. A function $f :X\rightarrow\C$ with a finite Lipschitz constant is called a Lipschitz function over $X$. The space of $\C$-valued Lipschitz functions is norm dense in the C*-algebra $C(X)$ of $\C$-valued continuous functions over $X$. Moreover, $\mathsf{Lip}$ is a seminorm on the space of $\C$-valued Lipschitz functions. However, we shall only work with the restriction of $\mathsf{Lip}$ to real-valued Lipschitz functions, which form a dense subset of the self-adjoint part $\sa{C(X)}$ of $C(X)$, because real-valued Lipschitz functions enjoy an extension property given by McShane's theorem \cite{McShane34} which will prove important in our work, as we shall see in a few paragraphs. Thus, unless explicitly stated otherwise, all the Lipschitz seminorms in this introduction are assumed to be restricted to real-valued Lipschitz functions.

A fundamental observation, due to Kantorovich, is that the dual seminorm of $\mathsf{Lip}$ induces a metric $\Kantorovich{\mathsf{Lip}}$ on the state space $\StateSpace(C(X))$ of $C(X)$, i.e. the space of Radon probability measures on $X$, and the topology for this metric is given by the weak* topology restricted to $\StateSpace(C(X))$. Moreover, the restriction of $\Kantorovich{\mathsf{Lip}}$ to the space $X$ identified with the subset of Dirac probability measures in $\StateSpace(C(X))$ is given by $\mathsf{m}$: thus the Lipschitz seminorm encodes all the metric information given by the metric $\mathsf{m}$ at the level of the C*-algebra $C(X)$. 

The distance $\Kantorovich{\mathsf{Lip}}$ is known as the Monge-Kantorovich metric, and was introduced by Kantorovich in \cite{Kantorovich40} as part of his research on the transportation problem introduced by Monge in 1781. The original formulation of the distance $\Kantorovich{\mathsf{Lip}}$ between two probability measures $\mu,\nu$ on $X$ involved minimizing $\displaystyle\int_{X\times X}\mathsf{m}\,d\gamma$ over all Radon probability measures $\gamma$ whose marginals are given by $\mu$ and $\nu$.  The duality relationship between the {\mongekant} and the Lipschitz seminorm was first made explicit by Kantorovich and Rubinstein in \cite{Kantorovich58}. This metric has since acquired several other names. Most notably, Dobrushin named this metric the Wasserstein metric in \cite{Dobrushin70}, after Wasserstein rediscovered it in the context of probability theory in \cite{Wasserstein69}. The work of Dobrushin in \cite{Dobrushin70}, incidentally, plays a fundamental role in our own work on generalizing the theory of quantum compact metric spaces to quantum locally compact spaces in \cite{Latremoliere12b}.

Rieffel proposed \cite{Rieffel98a,Rieffel99} to generalize the above model to the noncommutative world, as well as to extend the construction of metrics on the state spaces of C*-algebras from spectral triples proposed by Connes \cite{Connes89,Connes}. Thus, Rieffel defines a compact quantum metric space $(\A,\Lip)$ as a pair of an order-unit space $\A$ endowed with a densely defined seminorm $\Lip$ which vanishes exactly on the scalar multiples of the order-unit and such that the topology of the associated {\mongekant} $\Kantorovich{\Lip}$, defined between any two states $\varphi,\psi$ of $\A$ by:
\begin{equation*}
\Kantorovich{\Lip}(\varphi,\psi) = \sup\left\{ |\varphi(a)-\psi(a)| : a\in\A\text{ and }\Lip(a)\leq 1 \right\}
\end{equation*}
is the weak* topology restricted to the state space of $\A$. Note that this original definition does not involve any C*-algebraic structure. The seminorm $\Lip$ of a quantum compact metric space $(\A,\Lip)$ is called a Lip-norm \cite{Rieffel99}.

Rieffel then constructed the quantum Gromov-Hausdorff distance in analogy with Gromov's distance \cite{Gromov81,Gromov}, with the hope to initiate a new approach to problems in noncommutative geometry and mathematical physics based upon generalizations of ideas from metric geometry, where the Gromov-Hausdorff distance is an important tool. Given two compact metric spaces $(X,\mathsf{d}_X)$ and $(Y,\mathsf{d}_Y)$, the Gromov-Hausdorff distance may be defined as the infimum of the Hausdorff distance between $\iota_X(X)$ and $\iota_Y(Y)$ for any two isometric functions $\iota_X:X\hookrightarrow Z$ and $\iota_Y:Y\hookrightarrow Z$ into a compact metric space $(Z,\mathsf{d}_Z)$, taken over all such isometric embeddings. This distance may also be obtained by restricting one's attention to admissible metrics on the disjoint union $X\coprod Y$, i.e. metrics for which the canonical embeddings $X\hookrightarrow X\coprod Y$ and $Y\hookrightarrow X\coprod Y$ are isometries, and defining our distance between $(X,\mathsf{d}_X)$ and $(Y,\mathsf{d}_Y)$ as the infimum over all such admissible metrics of the Hausdorff distance between $X$ and $Y$ in $X\coprod Y$. 

In this context, we note that if $\mathsf{d}$ is an admissible metric on $X\coprod Y$, then any real-valued Lipschitz function on $(X,\mathsf{d}_X)$ can be extended to a real-valued function on $X\coprod Y$ with the same Lipschitz constant for $\mathsf{d}$ by McShane's theorem \cite{McShane34}. Thus, the Lipschitz seminorm associated with $\mathsf{d}_X$ is easily seen to be the quotient of the Lipschitz seminorm associated with $\mathsf{d}$. Conversely, if the Lipschitz seminorms for $\mathsf{d}_X$ and $\mathsf{d}_Y$ are the respective quotients of the Lipschitz seminorm for some metric $\mathsf{d}$ on $X\coprod Y$ along the canonical surjections from $X\coprod Y$ onto $X$ and $Y$, then $\mathsf{d}$ is easily seen to be admissible. Thus, admissibility can be rephrased in terms of Lipschitz seminorms, in the spirit of Rieffel's approach to compact quantum metric spaces. However, this observation rests squarely on our choice to only work with real-valued Lipschitz functions. As seen for instance in \cite{Rieffel06}, a complex-valued function on $X$ with a finite Lipschitz constant $L$ can be extended to a Lipschitz function over $X\coprod Y$ for any admissible metric $\mathsf{d}$, but the infimum of the set of Lipschitz constants of all such extensions will typically be strictly greater than $L$ (see \cite[Corollary 4.1]{Rieffel06}, where it is shown that the best one may hope in general is $\frac{4L}{\pi}$). This is the very reason for working with seminorms on dense subspaces of the self-adjoint part of C*-algebras, rather than seminorms on dense subsets of C*-algebras.

The construction of the quantum Gro\-mov-Haus\-dorff distance proceeds naturally by duality as follows: given two compact quantum metric spaces $(\A,\Lip_\A)$ and $(\B,\Lip_\B)$, we consider the set $\Adm(\Lip_\A,\Lip_\B)$ of all Lip-norms on $\A\oplus\B$ whose quotients to $\A$ and $\B$ are, respectively, $\Lip_\A$ and $\Lip_\B$. For any such admissible Lip-norm $\Lip$, one may consider the Haus\-dorff distance between $\StateSpace(\A)$ and $\StateSpace(\B)$ identified with their isometric copies in $\StateSpace(\A\oplus\B)$, where the state spaces $\StateSpace(\A)$, $\StateSpace(\B)$ and $\StateSpace(\A\oplus\B)$ are equipped with the metrics dual to, respectively, $\Lip_\A$, $\Lip_\B$ and $\Lip$. The infimum of these Haus\-dorff distances over all possible choices of $\Lip \in \Adm(\Lip_\A,\Lip_\B)$ is the quantum Gro\-mov-Haus\-dorff distance between $(\A,\Lip_\A)$ and $(\B,\Lip_\B)$.

As the quantum Gromov-Hausdorff distance is defined on a class of order-unit spaces, rather than C*-algebras, it is informally unaware of the multiplicative structure --- thus, two non-isomorphic C*-algebras may be at distance zero for Rieffel's distance, as long as their self-adjoint parts are isomorphic as order-unit spaces and the quantum metric structures are isomorphic in a natural sense \cite{Rieffel00}. In other words, when working with C*-algebras, distance zero for the quantum Gromov-Hausdorff distance gives rise to a Jordan isomorphism between self-adjoint parts, rather than a full *-isomorphism \cite{Alfsen01}. 

This weakened form of the desired coincidence property sparked a lot of research toward the goal of strengthening Rieffel's distance to make various stronger notions of isomorphisms necessary for distance zero \cite{kerr02,li03,li06,kerr09,Wu06b,Latremoliere13}. All of these approaches involve replacing the {\mongekant} on the state space with various alternative structures: in \cite{kerr02}, matricial-valued completely positive unital maps are used in lieu of states, in \cite{li03,li06,kerr09,Latremoliere13} the state space is replaced by objects built from the noncommutative analogues of the unit Lipschitz ball, such as the Lipschitz ball itself \cite{li06,kerr09}, or the restriction of the graph of the multiplication to the Lipschitz ball \cite{li03}, while in \cite{Wu06b}, the focus is on operator spaces and employs an enriched structure compared to compact quantum metric spaces \cite{Wu05,Wu06b}, adapted to noncommutative Banach spaces. In general, the methods employed to strengthen the quantum Gromov-Hausdorff distance relied on changing the object used to compute the distance rather than adding requirements to the Lip-norms themselves. Contrary to these earlier constructions, our new dual Gromov-Hausdorff propinquity relies solely on the natural metric on the state spaces of {\Lqcms s}, as does the quantum Gromov-Hausdorff distance, but add a natural requirement on Lip-norms.

Indeed, Lipschitz seminorms, as defined by Equation (\ref{Lip-eq}), posses a property which did not appear in the original work of Rieffel, and which connects the metric structure and the C*-algebraic structure: they satisfy the Leibniz identity
\begin{equation}\label{Leibniz-eq}
\forall f,g\in C(X) \quad \mathsf{Lip}(fg)\leq\mathsf{Lip}(f)\|g\|_{C(X)} + \|f\|_{C(X)}\mathsf{Lip}(g)\text{,}
\end{equation}
where we used the notations of Equation (\ref{Lip-eq}). Most examples of Lip-norms are also Leibniz seminorms, in an appropriate sense which we defined in \cite{Latremoliere13} (and often, in the same sense as Equation (\ref{Leibniz-eq}) if the Lip-norms are given on a dense subset of the whole C*-algebra rather than its self-adjoint part). This property has recently proven very desirable (e.g. \cite{Rieffel09,Rieffel10,Rieffel10c,Rieffel11,Rieffel12,Latremoliere13}) in order to study the behavior of C*-algebraic related structures, such as projective modules, under metric convergence. However, the quantum Gromov-Hausdorff distance does not provide a framework to work only with such Lip-norms: by its very construction, it forces one to consider many Lip-norms which may well not be connected to any underlying multiplicative structure. Our current research on the theory of Gromov-Hausdorff convergence for quantum locally compact metric spaces \cite{Latremoliere05b,Latremoliere12b} revealed other advantages in working with Leibniz seminorms, and in fact spawned our current work on finding a well-behaved metric for {\Lqcms s}. Yet, efforts in constructing a metric which only involve {\Lqcms s} have met with difficulties \cite{Rieffel09,Rieffel10c}. 

In particular, Rieffel introduced the quantum proximity in \cite{Rieffel10c} as a step toward the resolution of all the aforementioned issues by adding a strong form of the Leibniz property to the notion of admissibility of Lip-norms in the construction of the quantum Gromov-Hausdorff distance.  Yet the quantum proximity is unlikely to satisfy the triangle inequality, because the proof of the triangle inequality for the quantum Gromov-Hausdorff distance relies on taking the quotients of Lip-norms, and it is known \cite{Blackadar91} that the quotient of a Leibniz seminorm is not necessarily Leibniz. The matter of addressing this issue, as well as the coincidence property, while retaining the fundamental idea of Connes and Rieffel to use the {\mongekant} on the state spaces of quantum compact metric spaces as the fundamental tool of noncommutative metric geometry, remained an unsolved challenge for some time.

Recently \cite{Latremoliere13,Latremoliere13c}, we introduced a metric on the class of {\Lqcms s}, called the quantum Gromov-Hausdorff propinquity, to play the role intended by the quantum proximity from which it draws its name. Our new metric put very strong structural requirements on admissible Lip-norms, which are always built using bimodules which are themselves C*-algebras, and, as many approaches before, replaces the state space by the Lipschitz ball in some of the computations of the distance between {\Lqcms s}. This metric answers questions raised in \cite{Rieffel10} in addition to the ones we have discussed so far, with the goal to facilitate the study of convergence of modules and other such C*-algebraic related structures \cite{Rieffel10,Rieffel10c}. Yet our exploration of the question of completeness for the quantum propinquity revealed challenges caused by our restrictions on admissible Lip-norms. 

The result of our investigation into the matter of completeness for the quantum Gromov-Hausdorff propinquity is our definition of a dual version of the quantum propinquity, which is the new metric presented in this paper. Remarkably, this dual Gromov-Hausdorff propinquity enjoys the same desirable property as the quantum Gromov-Hausdorff propinquity: distance zero implies *-isomorphisms, all admissible Lip-norms are Leibniz, and one may even use our construction to work with subclasses of Leibniz Lip-norms if desired. Yet, in addition to providing an actual metric as a replacement for the quantum proximity, our dual Gromov-Hausdorff propinquity is also complete. This is the main reason for our interest in this metric: the completeness of the Gromov-Hausdorff distance is an essential property \cite{Gromov}, which the quantum Gromov-Hausdorff distance shares, and we feel that a C*-algebraic analogue for these distances ought to possess this property as well. Moreover, the dual Gromov-Hausdorff propinquity only uses the {\mongekant} on the state spaces of {\Lqcms s}, just as the quantum Gromov-Hausdorff distance does. It thus avoids the use of substitute structures to achieve the proper coincidence property and has a very natural physical interpretation in terms of states. The proof of the completeness of our metric occupies a large part of this paper, while we prove the coincidence property by establishing all the needed estimates to apply our work in \cite{Latremoliere13} to our new metric.

The main difference between the approach which we propose in this paper and earlier approaches to strengthening the quantum Gromov-Hausdorff distance is that we focus on additional properties of Lip-norms, rather than changing the underlying objects used to measure the distance between compact quantum metric spaces. As a result, our distance is very close in spirit to the original quantum Gromov-Hausdorff distance, yet is well-behaved with respect to Leibniz Lip-norms on C*-algebras, and thus offers a natural framework to study the extension of metric geometry to quantum spaces.

After a reminder about the class of {\Lqcms s}, we define our new metric. We then show that it has the desirable coincidence property. This section uses some results from \cite{Latremoliere13} to which we refer when appropriate. Afterward, we prove a comparison theorem relating our new metric with the quantum propinquity and the quantum Gromov-Hausdorff distance, as well as Rieffel's proximity. We conclude with the important result that the dual propinquity is complete.

\section{Leibniz Quantum Compact Metric Spaces}

This section provides the framework within which our current work takes place. We first set some basic notations. We refer, for instance, to \cite{Pedersen79} for the general theory of C*-algebras.

\begin{notation}
Let $\A$ be a unital C*-algebra. The unit of $\A$ is denoted by $\unit_\A$. The self-adjoint part $\{a\in\A: a=a^\ast\}$ of $\A$ is denoted by $\sa{\A}$. The state space of $\A$ is denoted by $\StateSpace(\A)$. The norm of $\A$ is denoted by $\|\cdot\|_\A$.
\end{notation}

At the root of our work is a pair $(\A,\Lip)$ of a unital C*-algebra $\A$ and a seminorm $\Lip$ which enjoys various properties. The following definition contains the minimal assumptions we will make on such a pair:

\begin{definition}\label{Lipschitz-pair-def}
A \emph{Lipschitz pair} $(\A,\Lip)$ is given by a unital C*-algebra $\A$ and a seminorm $\Lip$ defined on a norm-dense subspace $\dom{\Lip}$ of $\sa{\A}$ such that:
\begin{equation*}
\left\{ a \in \dom{\Lip} : \Lip(a) = 0 \right\} = \R \unit_\A \text{.}
\end{equation*}
\end{definition}

\begin{convention}
We adopt the usual convention that if $\Lip$ is a seminorm defined on a dense subspace $\dom{\Lip}$ of a topological vector space $V$, and if $a\in V$ is not in the domain of $\Lip$, then $\Lip(a) = \infty$. With this convention, we observe that:
\begin{equation*}
\dom{\Lip} = \left\{ a \in V : \Lip(a) < \infty \right\} \text{.}
\end{equation*}
Note that with this convention, we do not introduce any ambiguity when talking about lower semi-continuous seminorms by exchanging the original seminorm with its extension.
\end{convention}

The central object of noncommutative metric geometry is the generalization of the {\mongekant} \cite{Kantorovich40,Kantorovich58} to any Lipschitz pair:

\begin{definition}
The \emph{\mongekant} $\Kantorovich{\Lip}$ associated with a Lipschitz pair $(\A,\Lip)$ is the extended metric on the state space $\StateSpace(\A)$ of $\A$ defined by setting for all $\varphi,\psi \in \StateSpace(\A)$:
\begin{equation*}
\Kantorovich{\Lip}(\varphi,\psi) = \sup \left\{ |\varphi(a) - \psi(a)| : a\in\sa{\A} \text{ and }\Lip(a)\leq 1 \right\}\text{.}
\end{equation*}
\end{definition}

\begin{remark}
Definition (\ref{Lipschitz-pair-def}) is designed to ensure that the {\mongekant} associated with a Lipschitz pair is indeed an extended metric. Specifically, for a Lipschitz pair $(\A,\Lip)$, the {\mongekant} satisfies the axioms of a metric --- triangle inequality, coincidence property and symmetry --- though it may take the value $\infty$.
\end{remark}

Rieffel introduced the notion of a {\qcms} \cite{Rieffel98a,Rieffel99}, which actually can be defined in the more general context of order-unit spaces endowed with densely-defined seminorms; however, for our purpose, we focus on {\qcms s} whose underlying space is given by a C*-algebra --- in fact, our work is precisely aimed at solving various difficulties introduced by the restriction to the category of C*-algebras. The definition of a {\qcms} extends to the setting of Lipschitz pairs the fundamental property of the {\mongekant} from the classical setting:

\begin{definition}\label{qcms-def}
A \emph{quantum compact metric space} $(\A,\Lip)$ is a Lipschitz pair whose associated {\mongekant} metrizes the weak* topology of $\StateSpace(\A)$. The seminorm $\Lip$ is then called a \emph{Lip-norm}.
\end{definition}

\begin{remark}
Definition (\ref{qcms-def}) includes the requirement that the {\mongekant} of a {\qcms} $(\A,\Lip)$ is an actual metric, which gives the state space of $\A$ a finite diameter, since $\StateSpace(\A)$ is compact in the weak* topology.
\end{remark}

Our first examples of {\qcms s} are our original inspiration:

\begin{example}\label{classical-example}
The fundamental example of a quantum compact metric space is given by a pair $(C(X),\mathsf{Lip})$ where $(X,\mathsf{m})$ is a compact metric space and $\mathsf{Lip}$ is the Lipschitz seminorm associated to $\mathsf{m}$, as defined by Equation (\ref{Lip-eq}).
\end{example}

We note, using the notations in Example (\ref{classical-example}), that the set:
\begin{equation*}
\left\{ f: X\rightarrow \R : \|f\|_{C(X)}\leq 1 \text{ and }\mathsf{Lip}(f)\leq 1 \right\}
\end{equation*}
is compact for the uniform norm by Arz{\'e}la-Ascoli's theorem. In \cite[Theorem 1.9]{Rieffel98a}, Rieffel proposes a noncommutative analogue of the Arz{\'e}la-Ascoli theorem as a characterization of {\qcms s}. Later, an equivalent characterization of {\qcms s} was given in \cite{Ozawa05}. 

\begin{notation}
The diameter of a subset $A$ of a metric space $(E,\mathsf{d})$ is denoted by $\diam{A}{\mathsf{d}}$.
\end{notation}

\begin{theorem}[Theorem 1.9, \cite{Rieffel98a} and Proposition 1.3, \cite{Ozawa05}]\label{az-thm}
Let $(\A,\Lip)$ be a Lipschitz pair. The following assertions are equivalent:
\begin{enumerate}
\item $(\A,\Lip)$ is a {\qcms},
\item the diameter $\diam{\StateSpace(\A)}{\Kantorovich{\Lip}}$ is finite and:
\begin{equation*}
\left\{a\in\sa{\A} : \Lip(a)\leq 1\text{ and }\|a\|_\A \leq 1 \right\}
\end{equation*}
is norm precompact,
\item for some $\varphi\in\StateSpace(\A)$, the set:
\begin{equation*}
\left\{a\in\sa{\A} : \Lip(a)\leq 1\text{ and }\varphi(a) = 0 \right\}
\end{equation*}
is norm precompact,
\item for all $\varphi\in\StateSpace(\A)$, the set:
\begin{equation*}
\left\{a\in\sa{\A} : \Lip(a)\leq 1\text{ and }\varphi(a) = 0 \right\}
\end{equation*}
is norm precompact.
\end{enumerate}
\end{theorem}

We generalized the notion of {\qcms s} to the concept of quantum locally compact metric spaces in \cite{Latremoliere12b}, and proved a generalization of Rieffel's characterization for our {\qms s}, in the spirit of our earlier work \cite{Latremoliere05b} on the bounded-Lipschitz distance.

We note the following useful necessary condition on Lipschitz pairs to be {\qcms s}:
\begin{proposition}\label{qcms-separable-prop}
If $(\A,\Lip)$ is a {\qcms} then $\A$ is separable and, in particular, the domain of $\Lip$ contains a countable dense subset of $\sa{\A}$.
\end{proposition}

\begin{proof}
By Theorem (\ref{az-thm}), the set:
\begin{equation*}
\mathfrak{b} = \left\{ a\in\sa{\A} : \Lip(a)\leq 1 \text{ and } \|a\|_\A\leq 1 \right\}
\end{equation*}
is pre-compact for the norm topology of $\sa{\A}$, so it is totally bounded in norm. Consequently, it is separable in norm. Let $\mathfrak{a}$ be a countable and norm-dense subset of $\mathfrak{b}$.

We define:
\begin{equation*}
\mathfrak{l} = \left\{ qa : q\in\Q, a\in \mathfrak{a}  \right\} \subseteq \dom{\Lip}\text{.}
\end{equation*}

Now, $\mathfrak{l}$ is countable and dense in $\dom{\Lip}$: if $a\in\sa{\A}$ with $\Lip(a)<\infty$ and $a\not=0$, then let $r = \max\{\|a\|_\A,\Lip(a)\}$. Since $r^{-1}a\in\mathfrak{b}$, there exists a sequence $(a_n)_{n\in\N}$ in $\mathfrak{a}$ converging in norm to $r^{-1}a$. Then, for any sequence $(q_n)_{n\in\N}$ in $\Q$ converging to $r = \max\{\|a\|_\A,\Lip(a)\}$,  we have $q_n(r^{-1}a_n)\in\mathfrak{l}$ for all $n\in\N$ and $(q_nr^{-1}a_n)_{n\in\N}$ converges to $a$. (Of course, $0\in\mathfrak{l}$.)

Last, since $\dom{\Lip}$ is norm dense in $\sa{\A}$, we conclude that $\mathfrak{l}$ is dense in $\sa{\A}$. Thus $\sa{\A}$ is separable in norm. Since any element $a$ of $\A$ is of the form $a=\frac{a+a^\ast}{2} + i\frac{a-a^\ast}{2i}$, with $\frac{a+a^\ast}{2}, \frac{a-a^\ast}{2i} \in \sa{\A}$, we conclude that $\A$ is separable.
\end{proof}

\begin{remark}
In \cite[Proposition 1.1]{Rieffel02}, Rieffel proved that one may always construct one (and in fact, many) Lip-norms on a separable unital C*-algebra. Thus, a unital C*-algebra is separable if, and only if it can be equipped with a quantum compact metric structure.
\end{remark}

Using Rieffel's characterization, several examples of {\qcms s} have been established.

\begin{example}[\cite{Rieffel98a}]\label{action-example}
Let $\alpha$ be a strongly continuous action of a compact group $G$ on a unital C*-algebra $\A$ by *-automorphisms, such that:
\begin{equation*}
\left\{ a \in \A : \forall g \in G \quad \alpha^g(a) = a \right\} = \C\unit_\A\text{.}
\end{equation*}
Let $\mathsf{l}$ be a continuous length function on $G$, and denote the unit of $G$ by $e$. If, for all $a\in \sa{\A}$, we define:
\begin{equation*}
\Lip(a) = \sup\left\{ \frac{\|a-\alpha^g(a)\|_\A}{\mathsf{l}(g)} : g \in G\setminus\{e\} \right\}\text{,}
\end{equation*}
then $(\A,\Lip)$ is a {\qcms}.
\end{example}

\begin{example}\label{qt-example}
An important special case of Example (\ref{action-example}) is given by the twisted C*-algebras $C^\ast\left(\widehat{G},\sigma\right)$ of the Pontryagin dual $\widehat{G}$ of a compact metrizable group $G$, where $\sigma$ is a skew-bicharacter of $\widehat{G}$, with $\alpha$ the dual action of $G$ on $C^\ast\left(\widehat{G},\sigma\right)$, using any continuous length function $\mathsf{l}$ on $G$. In particular, for the $d$-torus $G=\T^d$, we see that the quantum tori are {\qcms s}.
\end{example}

\begin{example}[\cite{Connes89,Ozawa05}]\label{Hyperbolic-example}
Let $G$ be a finitely generated group and let $\mathsf{l}$ be the length function associated with some set of generators of $G$ such that $(G,\mathsf{l})$ is a Gromov Hyperbolic group. Let $\lambda$ be the left regular representation of $G$ on the Hilbert space $\ell^2(G)$, and for all  $\xi\in\ell^2(G)$, define $D\xi$ as $g\in G\mapsto \mathsf{l}(g)\xi(g)$ wherever $D\xi\in\ell^2(G)$. Last, we define, for all $a\in C_{\mathrm{red}}^\ast(G)$:
\begin{equation*}
\Lip(a) = \|[D,\lambda(a)]\|_{\B(\ell^2(G))}
\end{equation*}
where the norm $\|\cdot\|_{\B(\ell^2(G))}$ is the operator norm for linear operators on $\ell^2(G)$, which we allow to be infinity for unbounded operators. Note that $D$ is densely defined on $\ell^2(G)$.

This construction follows Connes' original approach to metric structures in noncommutative geometry laid out in \cite{Connes89}, where in particular Connes establishes that $(C^\ast_{\mathrm{red}}(G),\Lip)$ is a Lipschitz pair.
In \cite{Ozawa05}, Ozawa and Rieffel proved that $(C_{\mathrm{red}}^\ast(G),\Lip)$ is a {\qcms}.
\end{example}

\begin{example}[\cite{Rieffel02}]\label{Z-example}
Let $d\in \N,d>0$ and let $\sigma$ be a skew-bicharacter of $\Z^d$. Let us denote the left-regular representation of $C^\ast(\Z^d,\sigma)$ on $\ell^2(\Z^d)$ by $\lambda$. If $\mathsf{l}$ is the word metric on $\Z^d$ for some finite set of generators of $\Z^d$, and if we define $D$ as the multiplication operator by $\mathsf{l}$ on $\ell^2(\Z)$, then the seminorm:
\begin{equation*}
\Lip(a) = \|[D,\lambda(a)]\|_{\B(\ell^2(\Z))}\mathrm{,}
\end{equation*}
defined for all $a \in C^\ast(\Z^d,\sigma)$, is a Lip-norm. 
\end{example}

Examples (\ref{Hyperbolic-example}) and (\ref{Z-example}) are special cases of a general construction of metrics from spectral triples, pioneered by Connes in \cite{Connes89}. If $(\A,\mathscr{H},D)$ is a spectral triple over some *-algebra $\A$, then we can define a seminorm:
\begin{equation}\label{lip-spectral-eq}
\Lip : a\in\A\mapsto \|[D,\pi(a)]\|_{\B(\Hilbert)}
\end{equation}
where the norm $\|\cdot\|_{\B(\Hilbert)}$ is the operator norm on $\Hilbert$ and $\pi$ is the chosen representation of $\A$ on $\Hilbert$. It is not known under what conditions such a construction leads to a Lip-norm on $\A$, as it is usually difficult to check when a Lipschitz pair is a quantum compact metric space. Other examples of such a construction where a spectral triple actually leads to a Lip-norm, besides the one we have listed, include quantum compact metric spaces associated with spectral triples on C*-crossed-products \cite{Bellissard10}, spectral triples on quantum tori \cite{Rieffel98a}, and spectral triples on Connes-Dubois-Landi-Violette spheres \cite{Li05}, to name a few.

The classical Lipschitz seminorm enjoys an additional property related to the multiplication of functions, given by Inequality (\ref{Leibniz-eq}). Indeed, if $(\A,\Lip)$ is any of the quantum compact metric spaces given by Examples (\ref{classical-example}), (\ref{action-example}), (\ref{qt-example}), (\ref{Hyperbolic-example}), (\ref{Z-example}), then $\Lip$ is in fact densely defined on $\A$, and for all $a,b \in \A$ we have:
\begin{equation}\label{full-Leibniz-eq0}
\Lip(ab)\leq\Lip(a)\|b\|_\A + \|a\|_\A\Lip(b) \text{.}
\end{equation}
More generally, Inequality (\ref{full-Leibniz-eq0}) is satisfied by any seminorm defined by a spectral triple using Equation (\ref{lip-spectral-eq}).

This connection between noncommutative metric structure and C*-algebraic structure was absent from the original work of Rieffel \cite{Rieffel98a,Rieffel99,Rieffel00} but appeared to be very important in more recent developments in noncommutative metric geometry, where the consequences of convergence for various structures, such as projective modules \cite{Rieffel09}, is investigated. We shall therefore adopt the perspective that such an additional property should be incorporated in our theory, and as we shall prove, this new approach allows for the definition of a stronger metric than the quantum Gromov-Hausdorff distance which fits the C*-algebraic framework as described in our introduction.

As discussed in our introduction, the natural dual notion for isometric embeddings used in the definition of the Gromov-Hausdorff distance \cite{Gromov} is given by the notion of admissible Lip-norms, which reflects the fact that real-valued Lipschitz functions can be extended, without growing their Lipschitz seminorm, from closed subsets of a metric space to the whole space \cite{McShane34}. This property is not valid for complex-valued functions \cite{Rieffel06}, and this fact alone justifies that we focus on Lip-norms densely defined only on the self-adjoint part of C*-algebras (other interesting and relevant issues are discussed in \cite{Rieffel10c}).
We thus propose to express the Leibniz property in terms of the Jordan-Lie algebra structure of the self-adjoint part of C*-algebras, which will prove enough for our need. As we shall see in this paper, our constructions and main results are valid if we were to require Lip-norms to be densely defined on the whole C*-algebra and to satisfy Equation (\ref{full-Leibniz-eq0}), rather than our more general Definition (\ref{Leibniz-def}) below.

\begin{notation}
Let $\A$ be a C*-algebra. Then for all $a,b \in \A$, we use the following notation for the Jordan product of $a$ and $b$:
\begin{equation*}
\Jordan{a}{b} = \frac{1}{2}\left(ab+ba\right)\text{.}
\end{equation*}
We also use the following notation for the Lie product:
\begin{equation*}
\Lie{a}{b} = \frac{1}{2i}\left(ab-ba\right)\text{.}
\end{equation*}
Our choice of normalization for the Lie product means that $\sa{\A}$ is closed under both the Jordan and the Lie products, so that $(\sa{\A},\Jordan{\cdot}{\cdot},\Lie{\cdot}{\cdot})$ is a Jordan-Lie algebra \cite{Alfsen01,landsman98}.
\end{notation}

\begin{definition}\label{Leibniz-def}
A seminorm $\Lip$ defined on a dense Jordan-Lie sub-algebra $\dom{\Lip}$ of the self-adjoint part $\sa{\A}$ of a C*-algebra $\A$ is a \emph{Leibniz seminorm} when, for all $a,b \in \dom{\Lip}$:
\begin{equation*}
\Lip(\Jordan{a}{b}) \leq \|a\|_\A\Lip(b) + \|b\|_\A \Lip(a)\text{,}
\end{equation*}
and
\begin{equation*}
\Lip(\Lie{a}{b}) \leq \|a\|_\A\Lip(b) + \|b\|_\A \Lip(a)\text{.}
\end{equation*}
\end{definition}

We are now able to define the objects of our study.

\begin{definition}
A \emph{\Lqcms} $(\A,\Lip)$ is a quantum compact metric space whose Lip-norm $\Lip$ is also a Leibniz seminorm which is lower semi-continuous with respect to the norm of $\A$.
\end{definition}

Requiring the Leibniz property introduced in Definition (\ref{Leibniz-def}) is central to this paper: we shall see how this additional requirement, absent from the original construction of the quantum Gromov-Hausdorff distance, will allow us to obtain the strong coincidence property for our new metric, which shows the dual propinquity is sensitive to the C*-algebraic structure, rather than the order-structure, of the underlying space.

The requirement of lower semi-continuity is motivated by the fact that the notion of isometry between {\Lqcms s} can be expressed in terms of closed Lip-norms \cite{Rieffel00}. In the case when the underlying space is a Banach space (which is our situation, as we work with C*-algebras), a Lip norm is closed if and only if it is lower semi-continuous: indeed, the Lip-norm of a {\qcms} $(\A,\Lip)$ is closed when $\{a\in\sa{\A} : \Lip(a)\leq 1\}$ is norm-closed. We note moreover that in general, by \cite{Rieffel10c}, the closure of a Lip-norm (defined in \cite{Rieffel99}) with the Leibniz property will have the Leibniz property; hence the assumption of lower semi-continuity can be conveniently made without loss of generality. On the other hand, all the examples of Lip-norms given in Example (\ref{classical-example}) are lower semi-continuous.

The class of {\Lqcms s} can be endowed with a natural category structure, where a morphism $\pi$ from a {\Lqcms} $(\A,\Lip_\A)$ to a {\Lqcms} $(\B,\Lip_\B)$ is a unital *-morphism $\pi : \A\rightarrow\B$ such that the dual map $\pi^\ast:\StateSpace(\B)\rightarrow\StateSpace(\A)$ is a Lipschitz function from $\Kantorovich{\Lip_\B}$ to $\Kantorovich{\Lip_\A}$. In this case, an isomorphism of {\Lqcms s} is given by a *-isomorphism whose dual map is a bi-Lipschitz map. For our purpose, we shall focus on the stronger notion of isometric isomorphism:

\begin{definition}\label{iso-iso-def}
An \emph{isometric isomorphism} $\pi$ from a {\Lqcms} $(\A,\Lip_\A)$ to a {\Lqcms} $(\B,\Lip_\B)$ is a *-isomorphism from $\A$ to $\B$ such that for all $a\in\sa{\A}$ we have $\Lip_\B\circ \pi(a) = \Lip_\A(a)$.
\end{definition}

\begin{remark}
Our definition of isometry is compatible with \cite[Definition 6.3]{Rieffel00}, since the Lip-norm $\Lip$ of a {\Lqcms} is closed, i.e. the set:
\begin{equation*}
\left\{ a \in \sa{\A} : \Lip(a)\leq 1\right\}
\end{equation*}
is closed in $\A$. Indeed, as $\A$ is a C*-algebra and thus complete for its norm, the assumption of lower semi-continuity on $\Lip$ is equivalent to the notion of closed Lip-norm. Moreover, by \cite[Theorem 6.2]{Rieffel00}, a *-isomorphism between two {\Lqcms s} is an isometric isomorphism if and only if its dual map is an isometry between the state spaces endowed with their respective {\mongekant s}. Thus, isometric isomorphisms are isomorphisms in the category of {\Lqcms s}, as would be expected.
\end{remark}

We shall see that working with our dual Gromov-Hausdorff propinquity is somewhat akin to working in a different category with a weaker notion of morphism, analogous to the notion of a correspondence between metric spaces. As seen in \cite{Gromov}, the Gromov-Hausdorff distance can be defined in terms of correspondences between metric spaces, and their ``distortion'' --- our construction of the dual propinquity can be seen as an attempt to provide a noncommutative version of this approach as well.

\section{The Dual Gromov-Hausdorff Propinquity}

We now define our new metric, which is the subject of our paper. We begin with a very brief motivation for our definition.

For any two compact metric spaces $(X,\mathsf{d}_X)$, $(Y,\mathsf{d}_Y)$, the Gromov-Hausdorff distance $\dist_{GH}((X,\mathsf{d}_X),(Y,\mathsf{d}_Y))$ is the infimum of the Hausdorff distance between $f(X)$ and $g(Y)$, where $f$ and $g$ are isometries from $X$ and $Y$ to some metric space $(Z,\mathsf{d}_Z)$, respectively \cite{Gromov}. One may even choose $Z=X\coprod Y$.

Given two {\Lqcms s} $(\A,\Lip_\A)$ and $(\B,\Lip_\B)$, the natural dual picture for such an embedding would thus be given by two *-epi\-morphisms $\pi_\A:\D\rightarrow\A$ and $\pi_\B:\D\rightarrow\B$, where $(\D,\Lip_\D)$ is yet another {\Lqcms}, and such that the dual maps of $\pi_\A$ and $\pi_\B$ are isometries, respectively, from $(\StateSpace(\A),\Kantorovich{\Lip_\A})$ and from $(\StateSpace
(\B),\Kantorovich{\Lip_\B})$ to $(\StateSpace(\D),\Kantorovich{\Lip_\D})$. One may wish to then take the infimum of the Hausdorff distance between the images of $\StateSpace(\A)$ and $\StateSpace(\B)$ in $\StateSpace(\D)$ for $\Kantorovich{\Lip_\D}$ over all possible such embeddings. Up to metric equivalence, one may even be tempted to always take $\D$ of the form $\A\oplus\B$.

There is one major obstruction to this idea: because the quotient of Leibniz seminorms may not be Leibniz, it is not known how to show whether the above construction would give an object which satisfies the triangle inequality, as seen in \cite{Rieffel10c}. One solution is to forget the Leibniz property, which led to the original quantum Gromov-Hausdorff distance $\dist_q$. Yet, it is desirable to keep some connection between the quantum metric structure and the C*-algebraic structure, so that distance zero implies *-isomorphism (rather than Jordan isomorphism, as with $\dist_q$), and more generally to be able to work within the category of C*-algebras and associated structures. It is however very unlikely that we will obtain the triangle inequality if one simply restrict the Lip-norms in the construction of $\dist_q$ to be Leibniz. A new solution is needed.

Moreover, suppose that we are given two *-epimorphisms $\pi_\A:\D\rightarrow\A$ and $\pi_\B :\D\rightarrow\B$ from some {\Lqcms} $(\D,\Lip_\D)$. One can construct a seminorm $\Lip$ on $\A\oplus\B$ from $\Lip_\D$ (essentially by doubling $\D$ to $\D\oplus\D$, making both copies of $\StateSpace(\D)$ in $\StateSpace(\D\oplus\D)$ as close as we want), in such a manner that the Hausdorff distance between $\StateSpace(\A)$ and $\StateSpace(\B)$ in $(\StateSpace(\A\oplus\B),\Kantorovich{\Lip})$ is arbitrarily close to the Hausdorff distance between $\StateSpace(\A)$ and $\StateSpace(\B)$ in $(\StateSpace(\D),\Kantorovich{\Lip_\D})$. However this process, which we describe in detail in the proof of Lemma (\ref{dist-q-lemma}), involves taking a quotient of a Leibniz seminorm, which in general will fail to be a Leibniz seminorm. Thus, if we wish to remain within the class of {\Lqcms s} in the construction of our new metric, we can not restrict ourselves to always choosing $\D = \A\oplus\B$, as is done with the quantum Gromov-Hausdorff distance, without risking to make our distance larger than necessary. In particular, the quantum proximity \cite{Rieffel10c} may be larger than it needs to be since it only involves 
(strongly) Leibniz Lip-norms on $\A\oplus\B$ rather than the more general setting we will now present for our new metric.

The subject of this paper is to fix these issues. We start from the general dual picture proposed above, and offer a mean to construct a well-behaved metric out of this general intuition. We first implemented some of the techniques we shall use here when working with the quantum Gromov-Hausdorff propinquity \cite{Latremoliere13}, which was also motivated by the question of defining a metric on {\Lqcms s}, yet focused on seminorms defined using some very specific bimodules constructed from C*-algebras, as appear in the literature on the subject recently. In contrast, this paper offers a construction which is as close to the original quantum Gromov-Hausdorff distance (and the quantum proximity) as possible, and will be proven to lead to a \emph{complete} metric. It is, informally, a dual version of the quantum propinquity --- though it is in fact a weaker metric in general. Due to the origin of our method in our earlier work, we employ an analogue terminology in the present paper, yet we ``dualize'' some terms to help our intuition in understanding similarities and differences between our two families of propinquities. Also, we note that the numerical quantities proposed in this paper to build our metric are quite different from the ones introduced in \cite{Latremoliere13}.

We first formalize our ``dual isometric embeddings'' as follows:

\begin{definition}\label{tunnel-def}
Let $(\A_1,\Lip_1)$, $(\A_2,\Lip_2)$ be two {\Lqcms s}. A \emph{tunnel} $(\D,\Lip_\D,\pi_1,\pi_2)$ from $(\A_1,\Lip_1)$ to $(\A_2,\Lip_2)$ is a {\Lqcms} $(\D,\Lip_\D)$ such that, for all $j\in\{1,2\}$, the map $\pi_j : \D \twoheadrightarrow \A_j$ is a unital *-epimorphism such that for all $a\in \sa{\A_j}$:
\begin{equation*}
\Lip_j(a) = \inf \{ \Lip_\D(d) : d\in\sa{\D} \text{ and }\pi_j(d) = a \} \text{.}
\end{equation*}
\end{definition}

A tunnel is just a generalization of an admissible Lip-norm in the sense of \cite{Rieffel00}, where the {\Lqcms s} are quotients of arbitrary {\Lqcms s} rather than their coproduct only, and where the Leibniz property in the sense of Definition (\ref{Leibniz-def}), is imposed.

We now define two numerical quantities associated with tunnels, which will be the base for the construction of our distance. The first numerical quantity is rather natural, since we work with an extension of the Gromov-Hausdorff distance inspired by Rieffel's original construction for the quantum Gromov-Hausdorff distance. We call this quantity the reach of a tunnel.

\begin{notation}
The Hausdorff distance on the class of compact subsets of a metric space $(E,\mathsf{d})$ is denoted by $\Haus{\mathrm{d}}$.
\end{notation}

\begin{notation}\label{dual-map-notation}
If $\pi : \A \rightarrow \B$ is a continuous linear map from a topological vector space $\A$ to a topological vector space $\B$ then we denote by $\pi^\ast$ the dual map $T \in \B^\ast \mapsto T\circ\pi \in \A^\ast$.
\end{notation}

\begin{definition}\label{tunnel-reach-def}
Let $(\A,\Lip_\A)$ and $(\B,\Lip_\B)$ be two {\Lqcms s}. The \emph{reach} of a tunnel $\tau = (\D,\Lip_\D,\pi_\A,\pi_\B)$ from $(\A,\Lip_\A)$ to $(\B,\Lip_\B)$ is the non-negative real number:
\begin{equation*}
\tunnelreach{\tau}{\Lip_\A,\Lip_\B} = \Haus{\Kantorovich{\Lip_\D}}\left(\pi_\A^\ast\left(\StateSpace(\A)\right),\pi_\B^\ast\left(\StateSpace(\B)\right)\right) \text{.}
\end{equation*}
\end{definition}

\begin{remark}
If $(\D,\Lip_\D,\pi_\A,\pi_\B)$ is a quadruple where $(\D,\Lip_\D)$ is a {\Lqcms}, and $\pi_\A:\D\twoheadrightarrow\A$ and $\pi_\B:\D\twoheadrightarrow\B$ are *-epimorphism onto some unital C*-algebras $\A$ and $\B$, respectively, then the quotient seminorms $\Lip_\A$ and $\Lip_\B$ of $\Lip_\D$ by $\pi_\A$ and $\pi_\B$ are closed Lip-norms on $\A$ and $\B$, by \cite[Proposition 3.1, Proposition 3.3]{Rieffel00} --- although they may not be Leibniz. This observation however implies that a tunnel contains enough information to recover the quantum metric structure on its domain and codomain. Thus, there is no need to decorate our notation for the reach of a tunnel with the Lip-norms $\Lip_\A$ and $\Lip_\B$, unlike the situation with the notation for the reach of a bridge in \cite{Latremoliere13}. This remark will apply as well to the notations for depth and length of a tunnel.
\end{remark}

The second numerical quantity is new to our metric. The reason for its introduction, which will be formalized in Proposition (\ref{tunnel-lift-bound-prop}) and used with full strength in Proposition (\ref{tunnel-product-prop}) and our main theorems, Theorem (\ref{coincidence-thm}) and Theorem (\ref{completeness-thm}), is to help control the norm of the lift of Lipschitz elements. We call this quantity the depth of a tunnel.

\begin{notation}
If $A\subseteq E$ is a subset of a topological vector space $E$, then $\co{A}$ is the closure of the convex envelope of $A$, i.e. the smallest closed convex subset of $E$ containing $A$.
\end{notation}

\begin{definition}\label{tunnel-depth-def}
Let $(\A,\Lip_\A)$ and $(\B,\Lip_\B)$ be two {\Lqcms s}. The \emph{depth} of a tunnel $\tau = (\D,\Lip_\D,\pi_\A,\pi_\B)$ from $(\A,\Lip_\A)$ to $(\B,\Lip_\B)$ is the non-negative real number:
\begin{equation*}
\tunneldepth{\tau}{\Lip_\A,\Lip_\B} = \Haus{\Kantorovich{\Lip_\D}}\left(\StateSpace(\D),\co{\StateSpace(\A)\cup\StateSpace(\B)}\right) \text{.}
\end{equation*}
\end{definition}

One reason for the depth of tunnels not to appear in Rieffel's original construction may be the following observation:

\begin{remark}\label{null-depth-rmk}
If $\D=\A\oplus\B$, $\pi_\A : (a,b)\in\D\mapsto a$ and $\pi_\B:(a,b)\in\D\mapsto b$, and if $(\D,\Lip_\D,\pi_\A,\pi_\B)$ is a tunnel, then $\tunneldepth{\tau}{\Lip_\A,\Lip_\B} = 0$.
\end{remark}

Given our basic numerical quantities for tunnels, we then define:

\begin{definition}\label{tunnel-length-def}
Let $(\A,\Lip_\A)$ and $(\B,\Lip_\B)$ be two {\Lqcms s}. The \emph{length} of a tunnel $\tau = (\D,\Lip_\D,\pi_\A,\pi_\B)$ from $(\A,\Lip_\A)$ to $(\B,\Lip_\B)$ is the non-negative number:
\begin{equation*}
\tunnellength{\tau}{\Lip_\A,\Lip_\B} = \max\{ \tunnelreach{\tau}{\Lip_\A,\Lip_\B}, \tunneldepth{\tau}{\Lip_\A,\Lip_\B} \} \text{.}
\end{equation*}
\end{definition}

The informal idea to define the dual propinquity should be to take the infimum of the lengths of all tunnels between two given {\Lqcms s}. This idea, as we discussed, will likely fail to satisfy the triangle inequality. As with the quantum propinquity \cite{Latremoliere13}, we introduce a notion of ``paths'' made of tunnels, which we shall call journeys. Such an approach is used, for instance, in the definition of the quotient (pseudo)metric of a metric \cite{Gromov}. In general however, such a construction leads to a pseudo-metric, and one main result of this paper is that the dual Gromov-Hausdorff propinquity is in fact a metric.

Introducing journeys opens an interesting possibility: we may restrict our attention to journeys made of specific choices of tunnels, with additional properties which may be of use for a particular purpose. We thus introduce the following notion of a compatible class of tunnels, which is a class of tunnels rich enough to guarantee that the construction of our metric, when specialized to such a compatible class, is indeed a metric. We begin with the definition of the reversed tunnel of a tunnel, which is trivially checked to be tunnel in its own right:

\begin{definition}\label{reversed-tunnel-def}
Let $(\A,\Lip_\A)$ and $(\B,\Lip_\B)$ be two {\Lqcms s}. Let $\tau = (\D,\Lip_\D,\pi_\A,\pi_\B)$ be a tunnel from $(\A,\Lip_\A)$ to $(\B,\Lip_\B)$. The \emph{reversed tunnel} $\tau^{-1}$ of $\tau$ is the tunnel $(\D,\Lip_\D,\pi_\B,\pi_\A)$ from $(\B,\Lip_\B)$ to $(\A,\Lip_\A)$.
\end{definition}

\begin{definition}\label{compatible-tunnel-def}
Let $\mathcal{C}$ be a nonempty class of {\Lqcms s}. A class $\mathcal{T}$ of tunnels is \emph{compatible} with $\mathcal{C}$ (or \emph{$\mathcal{C}$-compatible}) when:
\begin{enumerate}
\item for all $\tau = (\D,\Lip_\D,\pi,\rho)\in\mathcal{T}$, we have $\tau^{-1} = (\D,\Lip_\D,\rho,\pi)\in\mathcal{T}$,
\item for all $(\A,\Lip_\A),(\B,\Lip_\B)\in\mathcal{C}$, there exists a finite family $(\A_j,\Lip_j)_{j\in\{1,\ldots,n+1\}}$ in $\mathcal{C}$ for some $n\in\N$ with $(\A_1,\Lip_1)=(\A,\Lip_\A)$, $(\A_{n+1},\Lip_{n+1})=(\B,\Lip_\B)$, as well as a finite family $(\tau_j)_{j\in\{1,\ldots,n\}}$ in $\mathcal{T}$ such that for all $j\in\{1,\ldots,n\}$ the tunnel $\tau_j$ goes from $(\A_j,\Lip_j)$ to $(\A_{j+1},\Lip_{j+1})$,
\item for all $(\A,\Lip_\A), (\B,\Lip_\B) \in\mathcal{C}$, if there exists a *-isomorphism $h : \A\rightarrow \B$ such that $\Lip_\B\circ h = \Lip_\A$, then we have $(\A,\Lip_\A,\mathrm{id}_\A,h)\in\mathcal{T}$ and $(\B,\Lip_\B,\mathrm{id}_\B,h^{-1})\in\mathcal{T}$, where $\mathrm{id}_\A : a\in\A\mapsto a$.
Note in particular that $(\A,\Lip_\A,\mathrm{id}_\A,\mathrm{id}_\A) \in \mathcal{T}$,
\item any tunnel in $\mathcal{T}$ is from an element in $\mathcal{C}$ and to an element of $\mathcal{C}$.
\end{enumerate}
\end{definition}

\begin{remark}\label{iso-tunnel-rmk}
Let $\mathcal{C}$ be a nonempty class of {\Lqcms s}. Let $(\A,\Lip_\A), (\B,\Lip_\B) \in\mathcal{C}$ such that there exists some *-isomorphism $h : \A\rightarrow \B$ such that $\Lip_\B\circ h = \Lip_\A$. It is then easy to check that, indeed, $(\A,\Lip_\A,\mathrm{id}_\A,h)$ and $(\B,\Lip_\B,h^{-1},\mathrm{id}_\B)$ are tunnels between $(\A,\Lip_\A)$ and $(\B,\Lip_\B)$.
\end{remark}

\begin{notation}
Let $\mathcal{C}$ be a nonempty class of {\Lqcms s} and let $\mathcal{T}$ be a class of tunnels compatible with $\mathcal{C}$. The set of all tunnels in $\mathcal{T}$ from $(\A,\Lip_\A)\in\mathcal{C}$ to $(\B,\Lip_\B)\in\mathcal{C}$ is denoted by:
\begin{equation*}
\tunnelset{\A,\Lip_\A}{\B,\Lip_\B}{\mathcal{T}}\text{.}
\end{equation*}
\end{notation}

\begin{example}\label{Lqcms-example}
The class $\mathcal{LT}$ of all possible tunnels over the class $\LQCMS$ of all {\Lqcms s} is $\LQCMS$-compatible. By Remark (\ref{iso-tunnel-rmk}) and Definition (\ref{reversed-tunnel-def}), it is easy to check that Assertions (1), (3) and (4) of Definition (\ref{compatible-tunnel-def}) are met. To show that Assertion (2) is met, it is sufficient to prove that given $(\A,\Lip_\A)$ and $(\B,\Lip_\B)$ in $\LQCMS$, there exists a tunnel $\tau$ from $(\A,\Lip_\A)$ and $(\B,\Lip_\B)$. 

In \cite[Proposition (4.6)]{Latremoliere13}, we constructed a bridge between $(\A,\Lip_\A)$ and $(\B,\Lip_\B)$, and from this (and in fact, any) bridge, Theorem (\ref{ncpropinquity-comparison-thm}), proven later in this paper, constructs a tunnel. While these two last results give all the necessary details, we can however give a brief explicit description of this tunnel here. 

Since $\A$ and $\B$ are separable as {\qcms s} by Proposition (\ref{qcms-separable-prop}), there exists two unital faithful representations $\omega_\A$ and $\omega_\B$ of $\A$ and $\B$, respectively, on some separable Hilbert space $\Hilbert$. Let:
\begin{equation*}
D = \max\left\{\diam{\StateSpace(\A)}{\Kantorovich{\Lip_\A}},\diam{\StateSpace(\B)}{\Kantorovich{\Lip_\B}}\right\}\text{.}
\end{equation*}
For all $(a,b)\in\sa{\A\oplus\B}$, we set:
\begin{equation*}
\Lip(a,b) = \sup\left\{\Lip_\A(a),\Lip_\B(b),\frac{1}{D}\left\|\omega_\A(a)-\omega_\B(b)\right\|_{\Hilbert\rightarrow\Hilbert}\right\}\text{,}
\end{equation*}
where $\|\cdot\|_{\Hilbert\rightarrow\Hilbert}$ is the operator norm for linear operators on $\Hilbert$.

Let $\pi_\A$ and $\pi_\B$ be the respective canonical surjections from $\A\oplus\B$ onto $\A$ and $\B$. Then one may check, using the techniques of \cite[Proposition (4.6)]{Latremoliere13}, that $(\A\oplus\B,\Lip,\pi_\A,\pi_\B)$ is a tunnel from $(\A,\Lip_\A)$ to $(\B,\Lip_\B)$. Thus Assertion (2) of Definition (\ref{compatible-tunnel-def}) is met as well. Note that moreover, Theorem (\ref{ncpropinquity-comparison-thm}) will show that the length of this tunnel is no more than $D$.
\end{example}

In view of the common use of Example (\ref{Lqcms-example}) in the rest of this paper, we introduce the following notation simplification:
\begin{notation}
The class of all tunnels from a {\Lqcms} $(\A,\Lip_\A)$ to a {\Lqcms} $(\B,\Lip_\B)$ is denoted simply as:
\begin{equation*}
\tunnelset{\A,\Lip_\A}{\B,\Lip_\B}{}\text{.}
\end{equation*}
\end{notation}

\begin{example}\label{CM-example}
A compact JLC*-metric space $(\A,\Lip)$ is a {\Lqcms} where the Lip-norm is JL-strongly Leibniz, in the sense that its domain is closed under the inverse map, and it also satisfies, for all $a\in\mathrm{GL}(\A)\cap\sa{\A}$:
\begin{equation*}
\Lip\left(a^{-1}\right)\leq \|a^{-1}\|_\A^2\Lip(a)\text{.}
\end{equation*}
This notion differs slightly from the notion of a compact C*-metric space introduced in \cite{Rieffel10c}, to which we shall return momentarily. The letters JL are meant for Jordan-Lie to emphasize the form of the Leibniz property as defined in Definition (\ref{Leibniz-def}).
Let $\mathcal{JLC}^\ast$ be the class of all compact JLC*-metric spaces \cite{Rieffel10c}. The class $\mathcal{JLCT}^\ast$ of all tunnels $(\D,\Lip_\D,\pi,\rho)$ from and to elements of $\mathcal{JLC}^\ast$, and such that $(\D,\Lip_\D)$ is a compact JLC*-metric space, is $\mathcal{JLC}^\ast$-compatible. Indeed, the construction of a tunnel between arbitrary {\Lqcms s} in Example (\ref{Lqcms-example}), when specialized to a pair of JLC*-metric spaces, lead to a tunnel in $\mathcal{JLCT}^\ast$, as can be easily checked.

Now, the class $\mathcal{CM}$ of compact C*-metric spaces is defined as a subclass of our class $\mathcal{JLC}$: a {\Lqcms} $(\A,\Lip) \in \mathcal{JLCM}$ is a compact C*-metric space if $\Lip$ is defined on a dense *-subalgebra of $\A$, closed under the inverse map, and such that for all $a,b \in \A$ we have:
\begin{equation}\label{full-Leibniz-eq}
\Lip(ab)\leq\|a\|_\A\Lip(b)+\Lip(a)\|b\|_\A\text{,}
\end{equation}
while for all $a\in\mathrm{GL}(\A)$ in the domain of $\Lip$, we have:
\begin{equation*}
\Lip\left(a^{-1}\right)\leq \|a^{-1}\|_\A^2\Lip(a)\text{.}
\end{equation*}
It is easy to check that Inequality (\ref{full-Leibniz-eq}) implies the Leibniz property of Definition (\ref{Leibniz-def}) --- see for instance \cite[Proposition 2.17]{Latremoliere13}. In \cite{Rieffel10c}, the notion of compact C*-metric space does not require the underlying algebra to be Banach --- only C*-normed --- but we focus in this paper on various subclasses of {\Lqcms s}.

With this in mind, the class $\mathcal{CMT}$ of all tunnels $(\D,\Lip_\D,\pi,\rho)$ from and to elements of $\mathcal{CM}$ and with $(\D,\Lip_\D)\in\mathcal{CM}$ is $\mathcal{CM}$-compatible. Indeed, Assertions (1), (3) and (4) of Definition (\ref{compatible-tunnel-def}) are automatic, while Assertion (2) follows again from the same construction as proposed in Example (\ref{Lqcms-example}), where it is easy to check that the constructed tunnel is in fact Strong Leibniz in that case.
\end{example}

\begin{example}\label{sum-example}
Another example is the class $\mathcal{SUM}_{\mathcal{C}}$ of tunnels of the form $(\A\oplus\B,\Lip,\pi_\A,\pi_\B)$ where $\pi_\A$ and $\pi_\B$ are the canonical surjections on each term, and $(\A,\Lip_\A)$,$(\B,\Lip_\B)$ ranges over all {\Lqcms s} in some nonempty subclass $\mathcal{C}$ of $\LQCMS$. The class $\mathcal{SUM}_{\mathcal{C}}$ is compatible with $\mathcal{C}$, again thanks to the construction of a tunnel between arbitrary elements of $\mathcal{C}$ given in Example (\ref{Lqcms-example}), which is an element of $\mathcal{SUM}_{\mathcal{C}}$ by construction.
\end{example}

We now have the tools to define journeys, i.e. path of tunnels in some appropriate class:

\begin{definition}\label{journey-def}
Let $\mathcal{C}$ be a nonempty class of {\Lqcms s} and $\mathcal{T}$ be a compatible class of tunnels for $\mathcal{C}$. A $\mathcal{T}$-\emph{journey} from $(\A,\Lip_\A)$ to $(\B,\Lip_\B)$, with $(\A,\Lip_\A),(\B,\Lip_\B)\in\mathcal{C}$, is a finite family:
\begin{equation*}
(\A_j,\Lip_j,\tau_j,\A_{j+1},\Lip_{j+1} : j=1,\ldots,n)
\end{equation*}
where:
\begin{enumerate}
\item $(\A_j,\Lip_j)$ is a {\Lqcms} in $\mathcal{C}$ for all $j\in\{1,\ldots,n+1\}$,
\item $(\A_1,\Lip_1) = (\A,\Lip_\A)$,
\item $(\A_{n+1},\Lip_{n+1}) = (\B,\Lip_\B)$,
\item $\tau_j$ is a tunnel in $\mathcal{T}$ from $(\A_j,\Lip_j)$ to $(\A_{j+1},\Lip_{j+1})$ for all $j\in\{1,\ldots,n\}$.
\end{enumerate}
The integer $n$ is called the size of the journey.
\end{definition}

\begin{notation}
Let $\mathcal{C}$ be a nonempty class of {\Lqcms s} and let $\mathcal{T}$ be a class of tunnels compatible with $\mathcal{C}$. The set of all $\mathcal{T}$-journeys from $(\A,\Lip_\A)\in\mathcal{C}$ to $(\B,\Lip_\B)\in\mathcal{C}$ is denoted by:
\begin{equation*}
\journeyset{\mathcal{T}}{\A,\Lip_\A}{\B,\Lip_\B}\text{.}
\end{equation*}
\end{notation}

When we do not specify a class of tunnels for a given journey, we mean that the journey is made of arbitrary tunnels between arbitrary {\Lqcms s}. However, we wish to emphasize the relative freedom which journeys afford us in the construction of the dual propinquity. Namely, the dual propinquity is really a collection of various metrics (up to isometric-isomorphism) on various subclasses of {\Lqcms s}. Indeed, journeys involve many intermediate {\Lqcms s}: using the notations of Definition (\ref{journey-def}), for a journey of size $n$, we have much liberty in choosing the {\Lqcms s} $(\A_2,\Lip_2),\ldots,(\A_{n},\Lip_n)$ and the tunnels $\tau_1,\ldots,\tau_n$. Depending on our application of the dual propinquity, we may have an interest in choosing all these intermediate {\Lqcms s} in a special class, such as the class of all compact C*-metric spaces, as in Example (\ref{CM-example}), for instance. Doing so would allow us to carry various computations within all these intermediate objects in the journey which may require some additional structure compared with what we assume a {\Lqcms} to be.

This freedom, which is shared with the quantum propinquity, is only constrained by the requirements listed as Definition (\ref{compatible-tunnel-def}) on the class of tunnels which we can choose from to build journeys. These requirements, as we shall see, are really the minimum needed to obtain a metric from our construction of the dual propinquity.

Once a class of compatible tunnels is chosen, we can thus work entirely within our favorite class of {\Lqcms s}. The proof of Theorem (\ref{coincidence-thm}) is valid for whatever choice of a compatible class of tunnel we pick: it should be noted that its conclusion will give an isometric isomorphism of {\Lqcms s} and nothing more; any preservation of additional structure would require careful inspection and possibly additional hypotheses. Our work on completeness is presented in the largest tunnel class (Example (\ref{Lqcms-example})), and thus whether a particular class of {\Lqcms s} is complete for some choice of a class of compatible tunnels would require that the constructions presented in the last part of our paper are compatible with the choice of tunnels as well. Last, the quantum propinquity of \cite{Latremoliere13} can be seen as a special case of our dual propinquity where tunnels are restricted to those which come from bridges in the manner we shall describe later in Theorem (\ref{ncpropinquity-comparison-thm}).

We now return to our main topic. A journey has a natural length:

\begin{definition}\label{length-def}
Let $(\A,\Lip_\A)$ and $(\B,\Lip_\B)$ be two {\Lqcms s}. The \emph{length} of a journey $\Upsilon = (\A_j,\Lip_j,\tau_j,\A_{j+1},\Lip_{j+1} : j=1,\ldots,n)$ from $(\A,\Lip_\A)$ to $(\B,\Lip_\B)$ is:
\begin{equation*}
\journeylength{\Upsilon} = \sum_{j=1}^n \tunnellength{\tau_j}{\Lip_j,\Lip_{j+1}} \text{.}
\end{equation*}
\end{definition}

We now can define our family of metrics, which we call the dual Gromov-Hausdorff propinquities.

\begin{definition}\label{dual-propinquity-def}
Let $\mathcal{C}$ be an admissible class of {\Lqcms s} and $\mathcal{T}$ be a class of tunnels compatible with $\mathcal{C}$. The \emph{dual Gromov-Hausdorff $\mathcal{T}$-propinquity} $\dpropinquity{\mathcal{T}}$ between two {\Lqcms s} $(\A,\Lip_\A)$ and $(\B,\Lip_\B)$ in $\mathcal{C}$ is defined as:
\begin{equation*}
\dpropinquity{\mathcal{T}}((\A,\Lip_\A),(\B,\Lip_\B)) = \inf \left\{ \journeylength{\Upsilon} : \Upsilon \in \journeyset{\mathcal{T}}{\A,\Lip_\A}{\B,\Lip_\B} \right\}\text{.}
\end{equation*}
\end{definition}

The dual propinquity associated with the category $\mathcal{CM}$ (Example (\ref{CM-example})) of compact C*-metric spaces is a direct extension of Rieffel's proximity. An even closer extension of the proximity is given by restricting one's attention to the class $\mathcal{SUM}_{\mathcal{CM}}$, as in Example (\ref{sum-example}). On the other hand, by default, we work with the dual propinquity associated with the class of all {\Lqcms s} (Example (\ref{Lqcms-example})). We introduce two notations to stress the role of the most important examples of dual propinquity:

\begin{notation}
The dual propinquity $\dpropinquity{}$ designates the dual propinquity $\dpropinquity{\mathcal{LT}}$.
\end{notation}

\begin{notation}
The dual propinquity $\dpropinquity{\ast}$ designates the dual propinquity $\dpropinquity{\mathcal{CMT}}$.
\end{notation}

At the level of generality which we have chosen, the dual propinquity is never infinite, though more interesting estimates depend on the choice of a class of tunnel. Moreover, the dual propinquity is zero between two isometrically isomorphic {\Lqcms s}. These properties partially illustrate the role of the notion of compatibility for tunnel.

\begin{proposition}\label{bounded-ncprop-prop}
For any class $\mathcal{C}$ of {\Lqcms s} and any $\mathcal{C}$-compatible class $\mathcal{T}$ of tunnels, and for any $(\A,\Lip_\A)$,$(\B,\Lip_\B)$ in $\mathcal{C}$, we have:
\begin{equation*}
\dpropinquity{\mathcal{T}}((\A,\Lip_\A),(\B,\Lip_\B)) < \infty\text{.}
\end{equation*}
Moreover, if there exists a *-isomorphism $h: \A\rightarrow \B$ such that $\Lip_\B\circ h=\Lip_\A$, then:
\begin{equation*}
\dpropinquity{\mathcal{T}}((\A,\Lip_\A),(\B,\Lip_\B)) = 0\text{.}
\end{equation*}
\end{proposition}

\begin{proof}
By Definition (\ref{compatible-tunnel-def}), the set $\journeyset{\mathcal{T}}{\A,\Lip_\A}{\B,\Lip_\B}$ is not empty, hence our first inequality. By the same definition, $\iota_\A = (\A,\Lip_\A,\mathrm{id}_\A,h)\in\mathcal{T}$, where $\mathrm{id}_\A$ is the identity map on $\A$, and the length of the tunnel $\iota_\A$ is trivially zero. This concludes our proposition.
\end{proof}

The most natural classes of tunnels proposed in Examples (\ref{Lqcms-example}),(\ref{CM-example}) and Example (\ref{sum-example}) give rise to noncommutative propinquities with a more explicit upper bound. We propose a terminology for such classes, as it may be useful for later use.

\begin{definition}
Let $\mathcal{C}$ be a class of {\Lqcms s}. A class $\mathcal{T}$ of tunnels which is compatible with $\mathcal{C}$ is \emph{$\mathcal{C}$-regular} when, for all $(\A,\Lip_\A)$, $(\B,\Lip_\B)$ in $\mathcal{C}$:
\begin{equation*}
\dpropinquity{\mathcal{T}}((\A,\Lip_\A),(\B,\Lip_\B))\leq\max\{\diam{\StateSpace(\A)}{\Kantorovich{\Lip_\A}},\diam{\StateSpace(\B)}{\Kantorovich{\Lip_\B}}\text{.}
\end{equation*}
\end{definition}

We shall see, using Theorem (\ref{ncpropinquity-comparison-thm}), that the class of tunnels in Examples (\ref{Lqcms-example}) and (\ref{CM-example}) are regular.



We conclude this section with the result that the dual Gromov-Hausdorff propinquity satisfies the triangle inequality and is symmetric. In the process of doing so, we describe a category structure on every nonempty class of {\Lqcms s}, whose Hom-sets are given by sets of reduced journeys, to be defined below. We will use this observation to provide some guidance to the intuition of the main theorem of this paper in the next section.

\begin{definition}
Let $\mathcal{C}$ be a nonempty class of {\Lqcms s} and $\mathcal{T}$ be a class of tunnels compatible with $\mathcal{C}$. Let $(\A,\Lip_\A)$, $(\B,\Lip_\B)$ and $(\D,\Lip_\D)$ in $\mathcal{C}$. Let:
\begin{equation*}
\Upsilon_1 = \left( \A_j,\Lip_j,\tau_j,\A_{j+1},\Lip_{j+1} : j=1,\ldots,n_1 \right)
\end{equation*}
be a journey from $(\A,\Lip_\A)$ to $(\B,\Lip_\B)$, and:
\begin{equation*}
\Upsilon_2 = \left( \B_j,\Lip^j,\tau^j,\B_{j+1},\Lip^{j+1} : j=1,\ldots,n_2 \right)
\end{equation*}
be a journey from $(\B,\Lip_\B)$ to $(\D,\Lip_\D)$.

If we set, for all $j\in\{1,\ldots,n_1+n_2\}$:
\begin{equation*}
\begin{split}
\D_j &=  \left\{\begin{array}{l} \A_j \text{ if $j\leq n_1$}\\ \B_{j-n_1} \text{ if $j>n_1$}\end{array}\right.\text{ and }\\
\Lip[j] &= \left\{\begin{array}{l} \Lip_j \text{ if $j\leq n_1$}\\ \Lip^{j-n_1} \text{ if $j>n_1$}\end{array}\right.\text{ and }\\
\rho_j  &= \left\{\begin{array}{l} \tau_j \text{ if $j\leq n_1$}\\ \tau^{j-n_1} \text{ if $j>n_1$}\end{array}\right.\text{,}
\end{split}
\end{equation*}
then $\Upsilon_1\star \Upsilon_2$ is the $\mathcal{T}$-journey $(\D_j,\Lip[j],\rho_j,\D_{j+1},\Lip[j+1] : j=1,\ldots,n_1+n_2)$ from $(\A,\Lip_\A)$ to $(\D,\Lip_\D)$. We call $\Upsilon_1\star\Upsilon_2$ the \emph{composition} of $\Upsilon_1$ by $\Upsilon_2$.
\end{definition}

We already used two of the four conditions for a class of tunnels to be compatible with some class $\mathcal{C}$ of {\Lqcms s} in Proposition (\ref{bounded-ncprop-prop}): the existence in the class of a tunnel of length zero for all isometric isomorphisms between elements of $\mathcal{C}$ and the existence of at least one journey between any two elements of $\mathcal{C}$. The condition regarding closure under the reverse operation allows us to define reverse journeys, which will give us that the dual Gromov-Hausdorff propinquity is symmetric.

\begin{definition}
Let $\mathcal{C}$ be a nonempty class of {\Lqcms s}. Let $\mathcal{T}$ be a $\mathcal{C}$-compatible class of tunnels and let $\Upsilon = (\A_j,\Lip_j,\tau_j,\A_{j+1},\Lip_{j+1} : j=1,\ldots,n)$ be a $\mathcal{T}$-journey. The \emph{reversed journey} $\Upsilon^{-1}$ is the $\mathcal{T}$-journey:
\begin{equation*}
(\A_{n-j+1},\Lip_{n-j+1},\tau_j^{-1},\A_{n-j},\Lip_{n-j} : j=1,\ldots,n)\text{.}
\end{equation*}
\end{definition}

The two operations of composition and reverse of journeys allow us to prove the following properties of the dual Gromov-Hausdorff propinquity, bringing us closer to proving that the dual Gromov-Hausdorff propinquity is a metric.

\begin{theorem}\label{triangle-thm}
Let $\mathcal{C}$ be a nonempty class of {\Lqcms s} and let $\mathcal{T}$ be a $\mathcal{C}$-compatible class of tunnels.
\begin{enumerate}
\item For all $(\A,\Lip_\A),(\B,\Lip_\B)\in\mathcal{C}$, we have:
\begin{equation*}
\dpropinquity{\mathcal{T}}((\A,\Lip_\A),(\B,\Lip_\B)) = \dpropinquity{\mathcal{T}}((\B,\Lip_\B),(\A,\Lip_\A)) \text{.}
\end{equation*}
\item For all $(\A,\Lip_\A),(\B,\Lip_\B),(\D,\Lip_\D)\in\mathcal{C}$, we have:
\begin{equation*}
\dpropinquity{\mathcal{T}}((\A,\Lip_\A),(\D,\Lip_\D)) \leq \dpropinquity{\mathcal{T}}((\A,\Lip_\A),(\B,\Lip_\B)) +  \dpropinquity{\mathcal{T}}((\B,\Lip_\B),(\D,\Lip_\D)) \text{.}
\end{equation*}
\end{enumerate}
\end{theorem}

\begin{proof}
The proof follows the scheme of \cite[Proposition 2.3.7]{Latremoliere13}. Let $\varepsilon > 0$. Let $\Upsilon_1 \in \journeyset{\mathcal{T}}{\A,\Lip_\A}{\B,\Lip_\B}$ and $\Upsilon_2 \in \journeyset{\mathcal{T}}{\B,\Lip_\B}{\D,\Lip_\D}$ such that:
\begin{equation*}
\journeylength{\Upsilon_1} \leq \dpropinquity{\mathcal{T}}((\A,\Lip_\A),(\B,\Lip_\B)) + \varepsilon
\end{equation*}
and
\begin{equation*}
\journeylength{\Upsilon_2} \leq \dpropinquity{\mathcal{T}}((\B,\Lip_\B),(\D,\Lip_\D)) + \varepsilon\text{.}
\end{equation*}
First, note that $\journeylength{\Upsilon_1^{-1}} = \journeylength{\Upsilon_1}$ and thus:
\begin{equation*}
\dpropinquity{\mathcal{T}}((\B,\Lip_\B),(\A,\Lip_\A))\leq \dpropinquity{\mathcal{T}}((\A,\Lip_\A),(\B,\Lip_\B)) + \varepsilon\text{.}
\end{equation*}
Since $\varepsilon>0$ is arbitrary, and by symmetry of the roles of $(\A,\Lip_\A)$ and $(\B,\Lip_\B)$, we conclude that $\dpropinquity{\mathcal{T}}((\A,\Lip_\A),(\B,\Lip_\B)) = \dpropinquity{\mathcal{T}}((\B,\Lip_\B),(\A,\Lip_\A)) \text{.}$

Similarly, we have $\Upsilon_1\star\Upsilon_2\in\journeyset{\mathcal{T}}{\A,\Lip_\A}{\D,\Lip_\D}$ and $\journeylength{\Upsilon_1\star\Upsilon_2} = \journeylength{\Upsilon_1} + \journeylength{\Upsilon_2}$, which implies:
\begin{equation*}
\dpropinquity{\mathcal{T}}((\A,\Lip_\A),(\D,\Lip_\D)) \leq \dpropinquity{\mathcal{T}}((\A,\Lip_\A),(\B,\Lip_\B)) + \dpropinquity{\mathcal{T}}((\B,\Lip_\B),(\D,\Lip_\D)) + 2\varepsilon\text{,}
\end{equation*}
which implies the triangle inequality since $\varepsilon > 0$ is arbitrary.
\end{proof}

\begin{remark}
Let $\mathcal{C}$ be an admissible class of {\Lqcms s} and $\mathcal{T}$ be a class of tunnels compatible with $\mathcal{C}$. The dual Gromov-Hausdorff $\mathcal{T}$-propinquity $\dpropinquity{\mathcal{T}}$ is by construction, the largest pseudo-distance $\mathsf{D}$ between elements of $\mathcal{C}$ with the property that, for all $(\A,\Lip_\A),(\B,\Lip_\B) \in \mathcal{C}$, and for any tunnel $\tau$ in $\mathcal{T}$, we have:
\begin{equation*}
\mathsf{D}((\A,\Lip_\A),(\B,\Lip_\B))\leq \tunnellength{\tau}{}\text{.}
\end{equation*}
As we have discussed at the beginning of this section, the quantity:
\begin{equation*}
\inf\left\{\tunnellength{\tau}{} : \tau \in\tunnelset{\A,\Lip_\A}{\B,\Lip_\B}{\mathcal{T}}\right\}
\end{equation*}
may not define a pseudo-metric as it may generally fail to satisfy the triangle inequality. We shall prove in Theorem (\ref{coincidence-thm}) that the dual propinquity is actually the largest \emph{metric} dominated by this quantity, up to isometric-isomorphism.
\end{remark}

We conclude this section with an observation. Let $\mathcal{C}$ be a nonempty class of {\Lqcms s} and $\mathcal{T}$ be a class of tunnels compatible with $\mathcal{T}$. We define a \emph{reduced $\mathcal{T}$-journey} $\Upsilon$ as a $\mathcal{T}$-journey which is either of the form $(\A,\Lip_\A,\mathrm{id}_\A,h)$ for some $(\A,\Lip_\A)\in\mathcal{C}$ and some isometric isomorphism $h$, or such that if $\Upsilon = (\A_j,\Lip_j,\tau_j,\A_{j+1},\Lip_{j+1} : j=1,\ldots,n)$, then for all $j,k\in \{1,\ldots,n+1\}$, if $j\not=k$ then $(\A_j,\Lip_j)$ is not isomorphic, as a {\Lqcms}, to $(\A_k,\Lip_k)$. The process of \emph{reducing} a $\mathcal{T}$-journey consists in removing any loop (i.e. sub-finite family starting and ending at isomorphic {\Lqcms}) to obtain a reduced journey, up to obvious changes to tunnels (by composing by isomorphisms of {\Lqcms s} the surjections of a tunnel).

It should be observed that reducing a journey reduces its length, and that the journeys consisting of a single tunnel defined by an isomorphism have length zero. Moreover, reducing a $\mathcal{T}$-journey leads to another $\mathcal{T}$-journey, as long as $\mathcal{T}$ has the additional property that if $\tau = (\D,\Lip_\D,\pi_\A,\pi_\B)\in\mathcal{T}$ is a tunnel from $(\A,\Lip_\A)$ to $(\B,\Lip_\B)$ and if there exists an isometric isomorphism $h:(\A,\Lip_\A)\rightarrow (\A',\Lip_\A')\in\mathcal{C}$, then $(\D,\Lip_\D,h\circ \pi_\A ,\pi_\B)$ is also in $\mathcal{T}$. We shall assume this property for the rest of this section, and note that any $\mathcal{C}$-compatible class of tunnel may be completed into a class with the above property while remaining compatible with $\mathcal{C}$.

Last, let us define the composition of two reduced journeys as the reduction of the star product of journeys. We then easily check that the composition of journeys is associative and that identity journeys are neutral (where the identity journeys consist of the single tunnel $(\A,\Lip_\A,\mathrm{id}_\A,\mathrm{id}_\A)$ for all $(\A,\Lip_\A)\in\mathcal{C}$). Moreover the reverse of a reduced journey $\Upsilon$ is the inverse of $\Upsilon$ for this composition.

We now define a category whose objects are given by the class $\mathcal{C}$. For any two objects $(\A,\Lip_\A)$ and $(\B,\Lip_\B)$, we denote the set of reduced $\mathcal{T}$-journeys by:
\begin{equation*}
\reversejourneyset{\mathcal{T}}{\A,\Lip_\A}{\B,\Lip_\B}
\end{equation*}
which is a nonempty subset of $\journeyset{\mathcal{T}}{\A,\Lip_\A}{\B,\Lip_\B}$. Now, it is easy to check that the sets $\reversejourneyset{\mathcal{T}}{\cdot}{\cdot}$ are the Hom-sets of a category over $\mathcal{C}$. Moreover, this category is a groupoid (more formally, if $\mathcal{C}$ is a set then our category is a groupoid; if $\mathcal{C}$ is not a set, then all arrows are invertible). Furthermore, our category is ``connected'', in the sense given by Assertion (2) of Definition (\ref{compatible-tunnel-def}): no Hom-set is ever empty.

Thus, up to reduction, journeys should be thought of as morphisms of a category, with a length associated to each morphism. This picture should be understood as a noncommutative analogue of the idea to define the Gromov-Hausdorff distance using correspondences as morphisms: to each correspondence, one associate a distortion, i.e. a numerical quantity, and the Gromov-Hausdorff distance is then taken as the infimum of all distortion over all correspondence between two given metric spaces. Our notion of a journey could be seen as our replacement for correspondences, and length as a measurement of distortion.

We also note that we could have done the same construction with the quantum propinquity \cite{Latremoliere13}, introducing a natural notion of reduced treks.

\section{Distance Zero}

One of the two main theorems of this paper is that the dual Gromov-Hausdorff propinquity is null between two {\Lqcms s} if and only if their underlying C*-algebras are *-isomorphic and their state spaces are isometric, thus completing the proof that $\dpropinquity{}$ is a metric on the isometry, *-isomorphism equivalence classes of {\Lqcms s}. The other main theorem will concern completeness and will be proven in the last section of this paper.

The proof of our first main theorem follows the scheme of \cite{Latremoliere13}. Seeing a journey as a morphism akin to a correspondence, we first introduce the notion of an ``image'' of an element of a {\Lqcms} by a journey, which is a subset of the co-domain of the journey. There is, in fact, a family of possible images indexed by a real number, and we call these sets \emph{target sets}.

We then proceed to give estimates on the norm-radius of target sets computed from the lengths of journeys.  We then note that journeys possess an analogue of the algebraic morphism properties between C*-algebras, expressed in terms of target sets. The estimates we obtain are analogue of the results about treks for the quantum propinquity, so once they are established, we progress toward our goal in a similar fashion as for the quantum propinquity in \cite{Latremoliere13}.

Results about journey are derived by inductions based upon results on tunnels. We therefore begin with the notion of the target set of a tunnel.

\begin{definition}\label{tunnel-targetset-def}
Let $(\A,\Lip_\A)$ and $(\B,\Lip_\B)$ be two {\Lqcms s}. Let:
\begin{equation*}
\tau = (\D,\Lip_\D,\pi_\A,\pi_\B) \in \tunnelset{\A,\Lip_\A}{\B,\Lip_\B}{}\text{.}
\end{equation*}

For any $a\in\sa{\A}$ and $r\geq \Lip_\A(a)$, the $r$-target set of $a$ for $\tau$ is defined as:
\begin{equation*}
\targetsettunnel{\tau}{a}{r} = \left\{ \pi_\B(d) \in \sa{\B} \left\vert \begin{array}{l}
d\in\sa{\D},
\Lip_\D(d) \leq r,\\
\pi_\A(d) = a
\end{array}
\right. \right\}\text{.}
\end{equation*}
\end{definition}

A first observation is that target sets are non-empty compact sets.

\begin{lemma}\label{tunnelset-not-empty-lemma}
Let $(\A,\Lip_\A)$ and $(\B,\Lip_\B)$ be two {\Lqcms s}, and let:
\begin{equation*}
\tau \in \tunnelset{\A,\Lip_\A}{\B,\Lip_\B}{}\text{.}
\end{equation*}

For any $a\in\sa{\A}$ with $\Lip_\A(a)<\infty$ and $r\geq \Lip_\A(a)$, the set $\targetsettunnel{\tau}{a}{r}$ is not empty and compact in $\B$.
\end{lemma}

\begin{proof}
Let $\varepsilon > 0$. By Definition (\ref{tunnel-def}), there exists $d_\varepsilon\in\sa{\D}$ such that $\pi_\A(d_\varepsilon) = a$ and:
\begin{equation*}
\Lip_\A(a)\leq \Lip_\D(d_\varepsilon) \leq\Lip_\A(a)+\varepsilon\text{.}
\end{equation*} 
Let $\varphi_\A\in\StateSpace(\A)$ and let $l = \max\{r,\Lip_\A(a) + 1\}$. The set:
\begin{equation*}
\mathfrak{d}=\{ d \in \sa{\D} : \Lip_\D(d)\leq l \text{ and }\varphi_\A\circ\pi_\A(d) = 0 \}
\end{equation*}
is norm compact by Theorem (\ref{az-thm}), since $\Lip_\D$ is lower semi-continuous. 

Now, for any $t\in\{r,\Lip_\A(a)+1\}$, the set:
\begin{equation*}
\mathfrak{l}_t=\left\{ d \in \sa{\A} : \pi_\A(d) = a \text{ and } \Lip_\D(d)\leq t \right\}
\end{equation*}
is closed since $\Lip_\D$ is lower semi-continuous and $\pi_\A$ is continuous. Moreover, for any $t\in\{r,\Lip_\A(a)+1\}$, we have:
\begin{equation*}
\mathfrak{l}_t \subseteq \left\{ d + \varphi_\A(a)\unit_\D : d\in \mathfrak{d} \right\}
\end{equation*}
by construction. Since $\mathfrak{d}$ is compact in $\D$ and $d\in\D\mapsto d+\varphi_\A(a)\unit_\D$ is continuous, the set $\mathfrak{l}_t$ is also compact in $\D$ for any $t\in\{r,\Lip_\A(a)+1\}$.

Therefore, in particular, the continuous image $\targetsettunnel{\tau}{a}{r}$ for $\pi_\B$ of $\mathfrak{l}_r$ is norm compact in $\B$.

Moreover, for all $n\in\N,n>0$, we have $d_{n^{-1}} \in \mathfrak{l}_{\Lip_\A(a)+1}\text{.}$ Thus, there exists a convergent subsequence $(d_{n_k^{-1}})_{k\in\N}$ which converges in norm to some $f \in \sa{\D}$. By lower-semi-continuity of $\Lip_\D$, we have $\Lip_\A(a)\leq \Lip_\D(f) \leq \Lip_\A(a)$, while by continuity of $\pi_\A$ we have $\pi_\A(f) = a$. Thus $\pi_\B(f)\in \targetsettunnel{\tau}{a}{r}$ for all $r\geq \Lip_\A(a)$ as desired.
\end{proof}

\begin{remark}
Using the notations of Definition (\ref{tunnel-targetset-def}), we note that if $a\in\sa{\A}$ is not in the domain of $\Lip_\A$, i.e. $\Lip_\A(a)=\infty$ with our convention, then $\targetsettunnel{\tau}{a}{\infty}$ is not empty since $\pi_\A$ is surjective, though it is not compact in general.
\end{remark}

The crucial property of the target sets for a tunnel $\tau$ is that their diameters are controlled by the length $\tau$. Consequently, when two {\Lqcms s} are close for the dual Gromov-Hausdorff propinquity, then one may expect that target sets for appropriately chosen tunnels have diameters of the order of the distance between our two {\Lqcms s}. Thus, if two {\Lqcms s} $(\A,\Lip_\A)$ and $(\B,\Lip_\B)$ are in fact at distance zero, one may find a sequence of target sets for any $a\in\sa{\A}$ associated to tunnels of ever smaller length, which converges to a singleton: the element in this singleton would then be our candidate as an image for $a$ by some prospective isometric isomorphism in the sense of Definition (\ref{iso-iso-def}). This general intuition will be the base for our proof of Theorem (\ref{coincidence-thm}). We now state the fundamental property of target sets upon which all our estimates rely.

\begin{proposition}\label{tunnel-fundamental-prop}
Let $(\A,\Lip_\A)$ and $(\B,\Lip_\B)$ be two {\Lqcms s} and:
\begin{equation*}
\tau \in \tunnelset{\A,\Lip_\A}{\B,\Lip_\B}{}\text{.}
\end{equation*}
If $a\in\sa{\A}$ with $\Lip_\A(a)<\infty$, if $r \geq \Lip_\A(a)$, and if $b\in \targetsettunnel{\tau}{a}{r}$ then:
\begin{equation*}
\|b\|_\B \leq \|a\|_\A + r\tunnelreach{\tau}{\Lip_\A,\Lip_\B}\text{,}
\end{equation*}
and $\Lip_\B(b)\leq r$.
\end{proposition}

\begin{proof}
Let $a\in\sa{\A}$, $r \geq \Lip_\A(a)$ and $b\in \targetsettunnel{\tau}{a}{r}$. Let $\psi\in\StateSpace(\B)$. By Definition (\ref{tunnel-length-def}), there exists $\varphi \in \StateSpace(\A)$ such that:
\begin{equation*}
\Kantorovich{\Lip_\D}(\varphi\circ\pi_\A,\psi\circ\pi_\B) \leq \tunnelreach{\tau}{\Lip_\A,\Lip_\B}\text{.}
\end{equation*}
By Definition (\ref{tunnel-targetset-def}), there exists $d \in \sa{\D}$ such that $\pi_\A(d)=a$, $\pi_\B(d) = b$ and $\Lip_\D(d)\leq r$. Then:
\begin{equation*}
\begin{split}
|\psi(b)| &= |\psi\circ\pi_\B(d)|\\
&\leq |\psi\circ\pi_\B(d) - \varphi\circ\pi_\A(d)| + |\varphi\circ\pi_\A(d)|\\
&\leq \Lip_\D(d)\Kantorovich{\Lip_\D}(\pi_\A^\ast(\varphi),\pi_\B^\ast(\psi)) + |\varphi(a)|\\
&\leq r\tunnelreach{\tau}{\Lip_\A,\Lip_\B} + \|a\|_\A \text{.}
\end{split}
\end{equation*}
Thus, by \cite[Theorem 4.3.4]{Kadison97}, since $b\in\sa{\B}$:
\begin{equation*}
\|b\|_\B = \sup \{ |\psi(b)| : \psi\in\StateSpace(\B) \} \leq r\tunnelreach{\tau}{\Lip_\A,\Lip_\B} + \|a\|_\A\text{.}
\end{equation*}
Since $\Lip_\B$ is the quotient seminorm of $\Lip_\D$ on $\B$ via $\pi_\B$, we have:
\begin{equation*}
\Lip_\B(b) = \Lip_\B(\pi_\B(d)) \leq \Lip_\D(d) \leq r\text{,}
\end{equation*}
as desired.
\end{proof}

We can now relate the linear structure of {\Lqcms s} and target sets of tunnels. An important consequence of this relation regards the diameter of target sets. Note that this result only involves the reach of tunnels (see Definition (\ref{tunnel-reach-def})).

\begin{corollary}\label{tunnel-linear-corollary}
Let $(\A,\Lip_\A)$ and $(\B,\Lip_\B)$ be two {\Lqcms s} and:
\begin{equation*}
\tau \in \tunnelset{\A,\Lip_\A}{\B,\Lip_\B}{}\text{.}
\end{equation*}

Let $a,a'\in \sa{\A}$ with $\Lip_\A(a)<\infty$ and $\Lip_\A(a')<\infty$. Let $r \geq \Lip_\A(a)$ and $r' \geq \Lip_\A(a')$. Then:
\begin{enumerate}
\item For all $b\in\targetsettunnel{\tau}{a}{r}$, $b'\in\targetsettunnel{\tau}{a'}{r'}$ and $t\in\R$, we have:
\begin{equation*}
b + tb' \in \targetsettunnel{\tau}{a+ta'}{r+|t|r'}\text{,}
\end{equation*}
\item If $r \geq \max\{\Lip_\A(a),\Lip_\A(a')\}$ then:
\begin{equation*}
\sup \left\{ \|b-b'\|_\B : b\in\targetsettunnel{\tau}{a}{r},b'\in\targetsettunnel{\tau}{a'}{r} \right\} \leq \|a-a'\|_\A + 2r\tunnelreach{\tau}{\Lip_\A,\Lip_\B}\text{.}
\end{equation*}
\item In particular:
\begin{equation*}
\diam{\targetsettunnel{\tau}{a}{r}}{\|\cdot\|_\B} \leq 2r\tunnelreach{\tau}{\Lip_\A,\Lip_\B}\text{.}
\end{equation*}
\end{enumerate}
\end{corollary}

\begin{proof}
Let $b\in\targetsettunnel{\tau}{r}{a}$, $b'\in\targetsettunnel{\tau}{r'}{a'}$ and $t\in\R$. By Definition (\ref{tunnel-targetset-def}), there exists $d\in\sa{\D}$ such that $\pi_\B(d) = b$, $\pi_\A(d) = a$, $\Lip_\D(d)\leq r$. Similarly, there exists $d'\in\sa{\D}$ such that $\pi_\A(d')=a'$, $\pi_\B(d')=b'$ and $\Lip_\D(d')\leq r'$. Then:
\begin{equation*}
\Lip_\D(d+td') \leq \Lip_\D(d) + |t|\Lip_\D(d')\leq r + |t|r' \text{,}
\end{equation*}
and $\pi_\A(d+td') = a+ta'$, so $b+tb'=\pi_\B(d+td') \in \targetsettunnel{\tau}{a+ta'}{r+|t|r'}$ by Definition (\ref{tunnel-targetset-def}). This completes the proof of (1).

Now, let $a,a' \in \sa{\A}$ and $r \geq \max\{\Lip_\A(a),\Lip_\A(a')\}$. Then if $b\in\targetsettunnel{\tau}{a}{r}$ and $b'\in\targetsettunnel{\tau}{a'}{r}$ then $b-b' \in \targetsettunnel{\tau}{a-a'}{2r}$ by the proof of (1). By Proposition (\ref{tunnel-fundamental-prop}), we have:
\begin{equation}\label{tunnel-algebra-eq1}
\|b-b'\|_\B \leq \|a-a'\|_\A + 2r\tunnelreach{\tau}{\Lip_\A,\Lip_\B}\text{.}
\end{equation}
This proves Assertion (2) of our proposition.

Assertion (3) is now obtained from Inequality (\ref{tunnel-algebra-eq1}) with $a=a'$. This completes our proof.
\end{proof}

Our next goal is to relate the multiplicative structure of {\Lqcms s} with target sets of tunnels. This requires us to obtain some bound on the norm of lifts of elements with finite Lip-norm since the Leibniz property given in Definition (\ref{Leibniz-def}) involves both Lip-norms and norms of elements. Such estimates are made possible thanks to our Definition (\ref{tunnel-depth-def}) of the depth of a tunnel.

\begin{proposition}\label{tunnel-lift-bound-prop}
Let $(\A,\Lip_\A)$ and $(\B,\Lip_\B)$ be two {\Lqcms s} and:
\begin{equation*}
\tau \in \tunnelset{\A,\Lip_\A}{\B,\Lip_\B}{}\text{.}
\end{equation*}

If $a\in\sa{\A}$ with $\Lip_\A(a)<\infty$, if $r \geq \Lip_\A(a)$ and if $d\in\sa{\D}$ with $\Lip_\D(d)\leq r$ and $\pi_\A(d) = a$, then:
\begin{equation*}
\|d\|_\D \leq \|a\|_\A + 2 r\tunnellength{\tau}{\Lip_\A,\Lip_\B}\text{.}
\end{equation*}
\end{proposition}

\begin{proof}
Let $d\in\sa{\D}$ such that $\pi_\A(d) = a$ and $\Lip_\D(d)\leq r$. Write $b=\pi_\B(d)$, so that $b\in\targetsettunnel{\tau}{a}{r}$ by Definition (\ref{tunnel-targetset-def}).

Let $\psi \in \StateSpace(\D)$ and let $\varepsilon > 0$. By Definition (\ref{tunnel-depth-def}), there exist $t_\varepsilon \in [0,1]$, $\varphi_{\A,\varepsilon} \in \StateSpace(\A)$ and $\varphi_{\B,\varepsilon} \in \StateSpace(\B)$ such that:
\begin{equation*}
\Kantorovich{\Lip_\D}\left(\psi,(t_\varepsilon\varphi_{\A,\varepsilon}\circ\pi_\A + (1-t_\varepsilon)\varphi_{\B,\varepsilon}\circ\pi_\B))\right) < \tunneldepth{\tau}{\Lip_\A,\Lip_\B} + \varepsilon\text{.}
\end{equation*}
We write $\mu_\varepsilon = t_\varepsilon\varphi_{\A,\varepsilon}\circ\pi_\A + (1-t_\varepsilon)\varphi_{\B,\varepsilon}\circ\pi_\B \in \StateSpace(\D)$.

Now:
\begin{equation*}
\begin{split}
|\psi(d)| &\leq |\psi(d) - \mu_\varepsilon(d)| + |\mu_\varepsilon(d)|\\
&\leq r \Kantorovich{\Lip_\D}\left(\psi,\mu_\varepsilon\right) + t_\varepsilon|\varphi_{\A,\varepsilon}(a)| + (1-t_\varepsilon)|\varphi_{\B,\varepsilon}(b)|\\
&\leq r(\tunneldepth{\tau}{\Lip_\A,\Lip_\B} + \varepsilon) + t_\varepsilon\|a\|_\A + (1-t_\varepsilon)\|b\|_\B\\
&\leq r(\tunneldepth{\tau}{\Lip_\A,\Lip_\B} + \varepsilon) + \max\{\|a\|_\A,\|b\|_\B\}\text{.}
\end{split}
\end{equation*}
As $\varepsilon > 0$ is arbitrary, we conclude $|\psi(d)|\leq r\tunneldepth{\tau}{\Lip_\A,\Lip_\B} + \max\{\|a\|_\A,\|b\|_\B\}$. Since $\psi\in\StateSpace(\D)$ is arbitrary, we conclude $\|d\|_\D \leq r\tunneldepth{\tau}{\Lip_\A,\Lip_\B} + \max\{\|a\|_\A,\|b\|_\B\}$.

Now, by Proposition (\ref{tunnel-fundamental-prop}), we have $\|b\|_\B \leq \|a\|_\A + r\tunnelreach{\tau}{\Lip_\A,\Lip_\B}$. Hence our proposition is proven.
\end{proof}

\begin{remark}\label{generalized-tunnel-rmk}
The proofs of Proposition (\ref{tunnel-lift-bound-prop}) and Proposition (\ref{tunnel-fundamental-prop}) do not involve the Leibniz property of the various Lip-norms involved. Thus, we can still apply these propositions to a generalized notion of tunnel between {\qcms s} $(\A,\Lip_\A)$ and $(\B,\Lip_\B)$, given as a quadruple $(\D,\Lip_\D,\pi_\A,\pi_\B)$ of a {\qcms} $(\D,\Lip_\D)$ and two *-epimorphisms $\pi_\A:\D\rightarrow\A$ and $\pi_\B:\D\rightarrow\B$ such that the respective quotient seminorms of $\Lip_\D$ for $\pi_\A$ and $\pi_\B$ are given by $\Lip_\A$ and $\Lip_\B$. The notions of reach, depth and length are defined identically for these generalized tunnels as in Definitions (\ref{tunnel-reach-def}), (\ref{tunnel-depth-def}) and (\ref{tunnel-length-def}). This remark will prove useful for Lemma (\ref{tunnel-length-lemma}), where we will construct quadruples which will eventually turn out to be genuine tunnels, though the proof of this fact will require dealing with them as generalized tunnels first.
\end{remark}

We thus can now establish the analogue of a multiplicative morphism property for target sets. This property makes explicit use of the Leibniz property of Definition (\ref{Leibniz-def}), and it is the argument which will allow us to prove that the maps constructed from target sets, when the distance between two {\Lqcms s} is zero, are indeed multiplicative.

\begin{proposition}\label{tunnel-product-prop}
Let $(\A,\Lip_\A)$ and $(\B,\Lip_\B)$ be two {\Lqcms s} and:
\begin{equation*}
\tau \in \tunnelset{\A,\Lip_\A}{\B,\Lip_\B}{}\text{.}
\end{equation*}

Let $a,a'\in\sa{\A}$ with $\Lip_\A(a) < \infty$ and $\Lip_\A(a') < \infty$, and let $r\geq\max\{\Lip_\A(a),\Lip_\A(a')\}$. Let $b\in\targetsettunnel{\tau}{a}{r}$, $b'\in\targetsettunnel{\gamma}{a'}{r}$. Then:
\begin{equation*}
\Jordan{b}{b'} \in \targetsettunnel{\tau}{\Jordan{a}{a'}}{(\|a\|+\|a'\|+4r \tunnellength{\tau}{\Lip_\A,\Lip_\B})r}\text{,}
\end{equation*}
while:
\begin{equation*}
\Lie{b}{b'} \in \targetsettunnel{\tau}{\Lie{a}{a'}}{(\|a\|+\|a'\|+4r \tunnellength{\tau}{\Lip_\A,\Lip_\B})r}\text{.}
\end{equation*}
\end{proposition}

\begin{proof}
Let $a,a'\in\sa{\A}$, $r\geq \max\{\Lip_\A(a),\Lip_\A(a')\}$. Let $b\in\targetsettunnel{\gamma}{a}{r}$, $b'\in\targetsettunnel{\gamma}{a'}{r}$. Let $d,d' \in \sa{\D}$ such that:
\begin{equation*}
\Lip_\D(d)\leq r\text{, } \quad \pi_\A(d) = a\text{, } \quad \pi_\B(d)=b
\end{equation*}
and
\begin{equation*}
\Lip_\D(d')\leq r \text{, }\quad \pi_\A(d') = a' \text{, }\quad \pi_\B(d')=b'\text{.}
\end{equation*}

By Proposition (\ref{tunnel-lift-bound-prop}), we have:
\begin{equation*}
\|d\|_\D \leq \|a\|_\A + 2r\tunnellength{\tau}{\Lip_\A,\Lip_\B} \text{ and }\|d'\|_\D \leq \|a'\|_\A + 2r\tunnellength{\tau}{\Lip_\A,\Lip_\B} \text{.}
\end{equation*}

Since $(\D,\Lip_\D)$ is a {\Lqcms}, we get:
\begin{equation*}
\begin{split}
\Lip_\D(\Jordan{d}{d'}) &\leq \|d\|_\D\Lip_\D(d') + \|d'\|_\D\Lip_\D(d)\\
&\leq \left(\|a\|_\A + 2r\tunnellength{\tau}{\Lip_\A,\Lip_\B}\right)r + \left(\|a'\|_\A + 2r\tunnellength{\tau}{\Lip_\A,\Lip_\B}\right)r\\
&\leq \left(\|a\|_\A + \|a'\|_\A + 4r \tunnellength{\tau}{\Lip_\A,\Lip_\B}\right)r \text{.}
\end{split}
\end{equation*}
Since $\pi_\A(\Jordan{d}{d'}) = \Jordan{a}{a'}$, we conclude our proposition holds true for the Jordan product.

The proof for the Lie product is similar.
\end{proof}

We now turn to extending Proposition (\ref{tunnel-fundamental-prop}), Corollary (\ref{tunnel-linear-corollary}) and Proposition (\ref{tunnel-product-prop}) from tunnels to journeys. To this end, we first introduce the notion of an itinerary along a journey, which will then enable us to define target sets of journeys, as well as provide a natural tool to derive results about journeys from their analogues for tunnels.

\begin{definition}\label{itineraries-journey-def}
Let $\mathcal{C}$ be a nonempty class of {\Lqcms s} and let $\mathcal{T}$ be a class of tunnels compatible with $\mathcal{C}$. Let $(\A,\Lip_\A)$ and $(\B,\Lip_\B)$ in $\mathcal{C}$. Let $\Upsilon \in \journeyset{\mathcal{T}}{\A,\Lip_\A}{\B,\Lip_\B}$ and write:
\begin{equation*}
\Upsilon = \left( \A_j,\Lip_j,\tau_j,\A_{j+1},\Lip_{j+1}: j = 1,\ldots,n \right)\text{.}
\end{equation*}
For any $a\in\sa{\A}$ and $b\in\sa{\B}$, and for any $r\geq \Lip_\A(a)$, we define the set of \emph{$r$-itineraries from $a$ to $b$}:
\begin{equation*}
\lambdaitineraries{\Upsilon}{a}{b}{r}
\end{equation*}
by:
\begin{equation*}
\left\{ (\eta_j)_{j\in\{1,\ldots,n+1\}} : \forall j \in \{1,\ldots,n\} \quad \eta_{j+1} \in \targetsettunnel{\tau_j}{\eta_j}{r} \text{ and }\eta_1=a,\eta_{n+1}=b \right\} \text{.}
\end{equation*}
\end{definition}

\begin{notation}
Let $(\A,\Lip_\A)$ and $(\B,\Lip_\B)$ be two {\Lqcms s} and:
\begin{equation*}
\tau \in \tunnelset{\A,\Lip_\A}{\B,\Lip_\B}{}\text{.}
\end{equation*}

For any $a\in\sa{\A}$ and $r\geq \Lip_\A(a)$, we set:
\begin{equation*}
\lambdaitineraries{\tau}{a}{\cdot}{r} = \bigcup_{b\in\sa{\B}}\lambdaitineraries{\tau}{a}{b}{r} \text{.}
\end{equation*}
Similarly, for any $b\in\sa{\B}$ and $r\geq\Lip_\B(b)$, we set:
\begin{equation*}
\lambdaitineraries{\tau}{\cdot}{b}{r} = \bigcup_{a\in\sa{\B}}\lambdaitineraries{\tau}{a}{b}{r} \text{.}
\end{equation*}
\end{notation}

\begin{lemma}\label{itineraries-not-empty-lemma}
Let $\mathcal{C}$ be a nonempty class of {\Lqcms s} and let $\mathcal{T}$ be a class of tunnels compatible with $\mathcal{C}$. Let $(\A,\Lip_\A)$ and $(\B,\Lip_\B)$ in $\mathcal{C}$. Let $\Upsilon \in \journeyset{\mathcal{T}}{\A,\Lip_\A}{\B,\Lip_\B}$. For all $a\in\sa{\A}$ and $r\geq \Lip_\A(a)$, the set $\lambdaitineraries{\Upsilon}{a}{\cdot}{r}$ is not empty.
\end{lemma}

\begin{proof}
This follows from an easy induction using Lemma (\ref{tunnelset-not-empty-lemma}) when $\Lip_\A(a)<\infty$, and is trivial when $\Lip_\A(a)=\infty$.
\end{proof}

Journeys are a noncommutative analogue of correspondences between {\Lqcms s}, and in particular, as the analogue of set-valued functions, one can make sense of the image of an element by a journey; this is the subject of the following definition.

\begin{definition}\label{journey-targetset-def}
Let $\mathcal{C}$ be a nonempty class of {\Lqcms s} and let $\mathcal{T}$ be a class of tunnels compatible with $\mathcal{C}$. Let $(\A,\Lip_\A)$ and $(\B,\Lip_\B)$ in $\mathcal{C}$. Let $\Upsilon \in \journeyset{\mathcal{T}}{\A,\Lip_\A}{\B,\Lip_\B}$. For any $a\in\sa{\A}$ and $r\geq \Lip_\A(a)$, we set:
\begin{equation*}
\targetsetjourney{\Upsilon}{a}{r} = \left\{ b \in \sa{\B} : \lambdaitineraries{\Upsilon}{a}{b}{r} \not= \emptyset \right\}\text{.}
\end{equation*}
\end{definition}

The first natural observation is:

\begin{lemma}\label{targetsetjourney-not-empty-lemma}
Let $\mathcal{C}$ be a nonempty class of {\Lqcms s} and let $\mathcal{T}$ be a class of tunnels compatible with $\mathcal{C}$. Let $(\A,\Lip_\A)$ and $(\B,\Lip_\B)$ in $\mathcal{C}$. Let $\Upsilon \in \journeyset{\mathcal{T}}{\A,\Lip_\A}{\B,\Lip_\B}$. For any $a\in\sa{\A}$ and $r\geq\Lip_\A(a)$, the set $\targetsetjourney{\Upsilon}{a}{r}$ is not empty.
\end{lemma}

\begin{proof}
This follows immediately from (\ref{itineraries-not-empty-lemma}).
\end{proof}

By induction, we thus derive the following result about journeys. Observe that Assertions (2) and (3) of Corollary (\ref{journey-fundamental-corollary}) could be interpreted by stating, somewhat informally, that (reduced) journeys are morphisms for the underlying Jordan-Lie algebra structure of the self-adjoint parts of {\Lqcms s}. As we explained when presenting our results about target sets for tunnels, the following result will allow us to build isometric isomorphisms between {\Lqcms s} at distance zero, as limits of journeys, properly defined.

\begin{corollary}\label{journey-fundamental-corollary}
Let $\mathcal{C}$ be a nonempty class of {\Lqcms s} and let $\mathcal{T}$ be a class of tunnels compatible with $\mathcal{C}$. Let $(\A,\Lip_\A)$ and $(\B,\Lip_\B)$ be two {\Lqcms s} in $\mathcal{C}$ and $\Upsilon \in \journeyset{\mathcal{T}}{\A,\Lip_\A}{\B,\Lip_\B}$ be a journey from $(\A,\Lip_\A)$ to $(\B,\Lip_\B)$. Let $\dom{\Lip_\A}$ be the domain of $\Lip_\A$, i.e.:
\begin{equation*}
\dom{\Lip_\A} = \left\{ a\in\sa{\A} : \Lip_\A(a)<\infty \right\}\text{.}
\end{equation*}

\begin{enumerate}
\item If $a\in\dom{\Lip_\A}$, $r \geq \Lip_\A(a)$ and $b\in \targetsetjourney{\Upsilon}{a}{r}$ then:
\begin{equation*}
\|b\|_\B \leq \|a\|_\A + r\journeylength{\Upsilon}\text{.}
\end{equation*}
\item If $a,a'\in\dom{\Lip_\A}$, $r\geq\Lip_\A(a)$, $r'\geq\Lip_\A(a')$, $t\in\R$, and if $b\in\targetsetjourney{\Upsilon}{a}{r}$, $b'\in\targetsetjourney{\Upsilon}{a'}{r'}$ then:
\begin{equation*}
b+tb' \in \targetsetjourney{\Upsilon}{a+ta'}{r+|t|r'}\text{.}
\end{equation*}
\item If $a,a'\in\dom{\Lip_\A}$, $r\geq\max\{\Lip_\A(a),\Lip_\A(a')\}$, and if $b\in\targetsetjourney{\Upsilon}{a}{r}$, $b'\in\targetsetjourney{\Upsilon}{a'}{r}$ then:
\begin{equation*}
\Jordan{b}{b'} \in \targetsetjourney{\Upsilon}{\Jordan{a}{a'}}{(\|a\|_\A+\|a'\|_\A+4r\journeylength{\tau})r}\text{,}
\end{equation*}
while
\begin{equation*}
\Lie{b}{b'} \in \targetsetjourney{\Upsilon}{\Lie{a}{a'}}{(\|a\|_\A+\|a'\|_\A+4r\journeylength{\tau})r}\text{.}
\end{equation*}
\item If $a,a'\in\dom{\Lip_\A}$ and $r\geq \max\{\Lip_\A(a),\Lip_\A(a')\}$ then:
\begin{equation*}
\sup \left\{ \|b-b'\|_\B : b\in\targetsetjourney{\Upsilon}{a}{r},b'\in\targetsetjourney{\Upsilon}{a'}{r} \right\} \leq \|a-a'\|_\A + 2r\journeylength{\Upsilon}\text{.}
\end{equation*}
\item In particular, if $a\in\dom{\Lip_\A}$ and $r\geq\Lip_\A(a)$ then:
\begin{equation*}
\diam{\targetsetjourney{\Upsilon}{a}{r}}{\|\cdot\|_\B} \leq 2r\journeylength{\Upsilon}\text{.}
\end{equation*}
\end{enumerate}
\end{corollary}

\begin{proof}
We proceed by induction on the size of the journey $\Upsilon$ and using Proposition (\ref{tunnel-fundamental-prop}), Corollary (\ref{tunnel-linear-corollary}) and Proposition (\ref{tunnel-product-prop}) as well as the notion of itinerary. We refer to \cite[Proposition 5.11, Proposition 5.12]{Latremoliere13} for a similar argument.
\end{proof}

In addition, Proposition (\ref{tunnel-lift-bound-prop}) allows us to establish that target sets of journeys are compact in norm.

\begin{proposition}\label{journey-targetset-compact-prop}
Let $\mathcal{C}$ be a nonempty class of {\Lqcms s} and let $\mathcal{T}$ be a class of tunnels compatible with $\mathcal{C}$. Let $(\A,\Lip_\A)$ and $(\B,\Lip_\B)$ be two {\Lqcms s} in $\mathcal{C}$ and:
\begin{equation*}
\Upsilon \in \journeyset{\mathcal{T}}{\A,\Lip_\A}{\B,\Lip_\B}
\end{equation*}
be a journey from $(\A,\Lip_\A)$ to $(\B,\Lip_\B)$. If $a\in\sa{\A}$ with $\Lip_\A(a) < \infty$ and if $r\geq\Lip_\A(a)$ then $\targetsetjourney{\Upsilon}{a}{r}$ is compact in $\sa{\B}$.
\end{proposition}

\begin{proof}
By Lemma (\ref{tunnelset-not-empty-lemma}), our corollary holds true for every journey $\Upsilon$ of length $1$.

Assume now that, for some $n\in\N$ and for every $\mathcal{T}$-journey $\Upsilon$ of length $n$ starting at $(\A,\Lip_\A)$, the set $\targetsetjourney{\Upsilon}{a}{r}$ is compact.

Let $\Upsilon$ be a journey from $(\A,\Lip_\A)$ to $(\B,\Lip_\B)$ of length $n+1$. Write:
\begin{equation*}
\Upsilon = \left(\A_j,\Lip_j,\tau_j,\A_{j+1},\Lip_{j+1} : j=1,\ldots,n+1 \right)
\end{equation*}
and define:
\begin{equation*}
\Upsilon' =\left(\A_j,\Lip_j,\tau_j,\A_{j+1},\Lip_{j+1} : j=1,\ldots,n \right)
\end{equation*}
which is a $\mathcal{T}$-journey from $(\A,\Lip_\A)$ to $(\A_{n+1},\Lip_{n+1})$ and has length $n$. 

Let us write $\tau_{n+1}$ as $\tau_{n+1} = (\D,\Lip_\D,\pi,\omega)$, and note that the range of $\pi$ is $\A_{n+1}$ while the range of $\omega$ is $\A_{n+2} = \B$ by Definition (\ref{journey-def}).

Let us define:
\begin{equation*}
\mathfrak{l} = \left\{ d\in\sa{\D} : \Lip_\D(d)\leq r\text{ and }\pi(d) \in \targetsetjourney{\Upsilon'}{a}{r} \right\}\text{.}
\end{equation*}

If $b\in \targetsetjourney{\Upsilon}{a}{r}$ then by Definition (\ref{journey-targetset-def}), we check that there exists $\eta\in\targetsetjourney{\Upsilon'}{a}{r}$ with $b\in\targetsettunnel{\tau_{n+1}}{\eta}{r}$, and thus there exists $d\in\mathfrak{l}$ such that $\pi(d) = \eta$ and $\omega(d) = b$ with $\Lip_\D(d)\leq r$. Thus, $\targetsetjourney{\Upsilon}{a}{r}$ is the image of $\mathfrak{l}$ by the continuous map $\omega$ (the other inclusion being true by definition), and thus it is sufficient to prove that $\mathfrak{l}$ is compact in $\D$ in order to complete our induction.

Let $d\in\mathfrak{l}$. By definition of $\mathfrak{l}$, we have $\pi(d) \in \targetsetjourney{\Upsilon'}{a}{r}$ which is assumed compact in $\A_{n+1}$. Hence in particular, there exists $M > 0$ such that:
\begin{equation*}
\forall \eta\in \targetsetjourney{\Upsilon'}{a}{r}\quad \|\eta\|_{\A_{n+1}}\leq M\text{.}
\end{equation*}

By Proposition (\ref{tunnel-lift-bound-prop}), we conclude that:
\begin{equation*}
\|d\|_\D \leq \|\eta\|_{\A_{n+1}} + 2r\tunnellength{\tau_{n+1}}{} \leq M + 2r\tunnellength{\tau_{n+1}}{}\text{.}
\end{equation*}
Thus:
\begin{equation*}
\mathfrak{l} \subseteq \{ d\in\sa{\D} : \Lip_\D(d)\leq r\text{ and }\|d\|_\D\leq  M + 2r\tunnellength{\tau_{n+1}}{} \}
\end{equation*}
and the set of the right-hand side is compact by Theorem (\ref{az-thm}) and the lower semi-continuity of $\Lip_\D$.

It is thus enough to show that $\mathfrak{l}$ is closed in order to complete our induction.

Let $(d_n)_{n\in\N}$ be a sequence in $\mathfrak{l}$ which converges to some $d\in\sa{\D}$. By lower semi-continuity of $\Lip_\D$, we have $\Lip_\D(d) \leq r$. By continuity of $\pi$ and since $\targetsetjourney{\Upsilon'}{a}{r}$ is closed (since assumed compact), we have $\pi(d) \in \targetsetjourney{\Upsilon'}{a}{r}$. Hence $d\in\mathfrak{l}$ as desired.

Thus $\mathfrak{l}$ is compact in $\D$ and thus $\targetsetjourney{\Upsilon}{a}{r}$ is compact in $\B$. Our induction is complete.
\end{proof}

We have now established the analogues of \cite[Proposition 5.11, Proposition 5.12]{Latremoliere13}, and thus we are in a position to prove, using similar methods as \cite{Latremoliere13}, the main theorem of this section.

\begin{theorem}\label{coincidence-thm}
Let $\mathcal{C}$ be a nonempty class of {\Lqcms s} and let $\mathcal{T}$ be a class of tunnels compatible with $\mathcal{C}$. Let $(\A,\Lip_\A)$ and $(\B,\Lip_\B)$ be two {\Lqcms s} in $\mathcal{C}$. If:
\begin{equation*}
\dpropinquity{\mathcal{T}}((\A,\Lip_\A),(\B,\Lip_\B)) = 0
\end{equation*}
then there exists a *-isomorphism $h : \A\rightarrow \B$ such that $\Lip_\B\circ h = \Lip_\A$.
\end{theorem}

\begin{proof}
The proof of this theorem is now essentially identical to the proof of our \cite[Theorem 5.13]{Latremoliere13}, thanks to the estimates given by Corollary (\ref{journey-fundamental-corollary}) in lieu of the similar estimates in \cite[Proposition 5.11, Proposition 5.12]{Latremoliere13}, with tunnels for bridges and journeys for treks.

We only sketch the argument for the convenience of our reader. We refer the reader to the proof of \cite[Theorem 5.13]{Latremoliere13} for a detailed demonstration, and only aim in the following exposition at giving a rough idea of the construction of $h$.

By Definition (\ref{dual-propinquity-def}), and by our assumption, for all $n\in\N$, there exists a $\mathcal{T}$-journey:
\begin{equation*}
\Upsilon_n\in\journeyset{\mathcal{T}}{\A,\Lip_\A}{\B,\Lip_\B}
\end{equation*}
from $(\A,\Lip_\A)$ to $(\B,\Lip_\B)$ such that:
\begin{equation*}
\journeylength{\Upsilon_n}{} \leq \frac{1}{n+1}\text{.}
\end{equation*}
Now, fix $a\in \sa{\A}$ such that $\Lip_\A(a)<\infty$. By Assertion (1) of Corollary (\ref{journey-fundamental-corollary}), for all $r \geq \Lip_\A(a)$, the sequence of target sets $\left(\targetsetjourney{\Upsilon_n}{a}{r}\right)_{n\in\N}$ is a sequence of subsets of the compact set:
\begin{equation*}
\mathfrak{C}(a,r) = \left\{b \in \sa{\B} : \Lip_\B(b)\leq r\text{ and }\|b\|_\B\leq \|a\|_\A + r \right\}
\end{equation*}
since the length of $\Upsilon_n$ is less than $1$ for all $n\in\N$. Note that the sequence $\left(\targetsetjourney{\Upsilon_n}{a}{r}\right)_{n\in\N}$ consists of compact sets by Proposition (\ref{journey-targetset-compact-prop}).

Consequently, since the space of closed subsets of a compact metric space for the Hausdorff distance is itself compact by Blaschke's Theorem \cite[Theorem 7.3.8]{burago01}, there exists a subsequence $\left(\targetsetjourney{\Upsilon_{f(n)}}{a}{r}\right)_{n\in\N}$ which converges to some compact set $\targetsetjourney{}{a}{r}$ in the Hausdorff distance associated to $\|\cdot\|_\B$  on the norm-compact set $\mathfrak{C}(a,r)$. At the moment, the function $f:\N\rightarrow\N$ depends on $a$ and $r \geq \Lip_\A(a)$, and thus so does $\targetsetjourney{}{a}{r}$.

Now, since the sequence $\left(\diam{\targetsetjourney{\Upsilon_{f(n)}}{a}{r}}{\|\cdot\|_\B}\right)_{n\in\N}$ converges to $0$ by Assertion (4) of Corollary (\ref{journey-fundamental-corollary}), the set $\targetsetjourney{}{a}{r}$ is a singleton. Moreover, by Assertion (4) of Corollary (\ref{journey-fundamental-corollary}), for any $a$ in the domain of $\Lip_\A$ and $r\geq \Lip_\A(a)$, and for all $n\in\N$ we choose $b_n \in \targetsetjourney{\Upsilon_{f(n)}}{a}{r}$, then the sequence $(b_n)_{n\in\N}$ converges in norm to the unique element of $\targetsetjourney{}{a}{r}$.

This is the crux of Claim (5.14) in the proof of \cite[Theorem 5.13]{Latremoliere13}.

We now construct $h$ on a dense subset of $\sa{\A}$, as follows. By Proposition (\ref{qcms-separable-prop}), there exists a countable subset $\mathfrak{a}$ of the domain of $\Lip_\A$ which is norm-dense in $\sa{\A}$.

By a diagonal argument, we can find a strictly increasing $f :\N\rightarrow\N$ such that for all $a\in\mathfrak{a}$, the sequence $\left(\targetsetjourney{\Upsilon_{f(n)}}{a}{\Lip_\A(a)}\right)_{n\in\N}$ converges to a singleton, denoted $\{h(a)\}$, in the Hausdorff distance associated with $\|\cdot\|_\B$. 

Now, using density of $\mathfrak{a}$ in the domain of $\Lip_\A$, we can extend $h$ to a map from the domain of $\Lip_\A$ to $\sa{\B}$ and then prove that the resulting map is a Jordan-Lie morphism of norm $1$. This extension possess the property that, once again, if, for any $a\in\mathfrak{a}$ and $r\geq \Lip_\A(a)$, and for all $n\in\N$ we choose $b_n \in \targetsetjourney{\Upsilon_{f(n)}}{a}{r}$, then the sequence $(b_n)_{n\in\N}$ converges in norm to $h(a)$. This property follows from Assertion (3) of Corollary (\ref{journey-fundamental-corollary}), since if $a\in\mathfrak{a}$ and $a'$ is in the domain of $\Lip_\A$, then the Hausdorff distance between $\targetsetjourney{\Upsilon_{f(n)}}{a}{r}$ and $\targetsetjourney{\Upsilon_n}{a'}{r}$ is no more than $\|a-a'\|_\A+2r\journeylength{\Upsilon_{f(n)}}$, which can be made arbitrarily small by taking $n$ large enough and $a$ close enough to $a'$. This is proven in Claim (5.15) of \cite[Theorem 5.13]{Latremoliere13}.

The crucial step in proving that $h$ is a Jordan-Lie morphism utilizes Assertions (2) of Corollary (\ref{journey-fundamental-corollary}): for instance, if $a,a'$ are in the domain of $\Lip_\A$, and if $r \geq \max\{\Lip_\A(a),\Lip_\A(a')\}$, and if, for all $n\in\N$ we choose $b_n\in\targetsetjourney{\Upsilon_{f(n)}}{a}{r}$ and $b'_n\in\targetsetjourney{\Upsilon_{f(n)}}{a'}{r}$, then by Assertion (2) of Corollary (\ref{journey-fundamental-corollary}), we have:
\begin{equation*}
\begin{split}
\Jordan{b_n}{b_n'}&\in\targetsetjourney{\Upsilon_{f(n)}}{\Jordan{a}{a'}}{r\left(\|a\|_\A+\|a'\|_\A+\frac{4r}{f(n)+1}\right)}\\
&\subseteq \targetsetjourney{\Upsilon_{f(n)}}{\Jordan{a}{a'}}{r\left(\|a\|_\A+\|a'\|_\A+2r\right)}
\end{split}
\end{equation*}
and thus, the sequence $\left(\Jordan{b_n}{b_n'}\right)_{n\in\N}$ must converge to $h(\Jordan{a}{a'})$ since:
\begin{equation*}
\begin{split}
\Lip_\A\left(\Jordan{a}{a'}\right) &\leq \Lip_\A(a)\|a\|_\A + \Lip_\A(a')\|a\|_\A\\
&\leq r\left(\|a\|_\A+\|a'\|_\A\right) \leq r\left(\|a\|_\A+\|a'\|_\A+2r\right)\text{.}
\end{split}
\end{equation*}
 On the other hand, $\left(\Jordan{b_n}{b_n'}\right)_{n\in\N}$ converges to $\Jordan{h(a)}{h(a')}$. Thus:
\begin{equation*}
h(\Jordan{a}{a'})=\Jordan{h(a)}{h(a')}\text{ (see Claim (5.17) of the proof of \cite[Theorem 5.13]{Latremoliere13}).}
\end{equation*}
The same reasoning applies to linearity (see Claim (5.16) of the proof of \cite[Theorem 5.13]{Latremoliere13}) and the Lie product.

Moreover, using lower semi-continuity of $\Lip_\B$, one proves that $\Lip_\B\circ h \leq \Lip_\A$ (see Claim (5.16) of the proof of \cite[Theorem 5.13]{Latremoliere13}).

We then extend $h$ to $\A$ by density and linearity and prove that it is a *-morphism, which is natural from linearity and the Jordan-Lie morphism property on the domain of $\Lip_\A$, as shown in Claim (5.18) of the proof of \cite[Theorem 5.13]{Latremoliere13}.

The construction of the inverse of $h$ proceeds by observing that the construction of $h$ can be applied once again to the sequence of journeys $\left(\Upsilon_{f(n)}^{-1}\right)_{n\in\N}$, as shown in Claim (5.19) of the proof of \cite[Theorem 5.13]{Latremoliere13}. The key observation is that an itinerary from $a\in\sa{\A}$ to $b\in\sa{\B}$ for some journey $\Upsilon$ from $(\A,\Lip_\A)$ to $(\B,\Lip_\B)$ can be reversed into an itinerary for $\Upsilon^{-1}$ from $b$ to $a$. Once an inverse is constructed, it is easy to check that $h$ is an isometric isomorphism.

We point out that there is no unique isometric isomorphism $h$, as one may compose $h$ with isometric automorphisms on either end. This lack of uniqueness is reflected in the proof above by the use of compactness to extract a subsequence of journeys. Different subsequences may lead to different isometric isomorphisms. However, we show in Claim (5.19) of \cite[Theorem 5.13]{Latremoliere13} that once such a choice is made (after the diagonal process step),  the inverse which we construct is indeed unique --- it turns out that the sequences $\left(\targetsetjourney{\Upsilon^{-1}_{f(n)}}{a}{r}\right)_{n\in\N}$ converge to $\{h^{-1}(a)\}$ for all $a\in\sa{\A}$ and $r\geq \Lip_\A(a)$. 
\end{proof}

We conclude with a summary of the two previous sections:
\begin{theorem}
Let $\mathcal{C}$ be a nonempty class of {\Lqcms s} and let $\mathcal{T}$ be a class of tunnels compatible with $\mathcal{C}$. The dual Gro\-mov-Haus\-dorff \, $\mathcal{T}$-propinquity $\dpropinquity{\mathcal{T}}$ is a metric on the isometric isomorphic classes of {\Lqcms s} in $\mathcal{C}$, in the following sense. For any three {\Lqcms s} $(\A_1,\Lip_1)$, $(\A_2,\Lip_2)$ and $(\A_3,\Lip_3)$ in $\mathcal{C}$, we have:
\begin{enumerate}
\item $\dpropinquity{\mathcal{T}}((\A_1,\Lip_1),(\A_2,\Lip_2)) \in [0,\infty)$,
\item $\dpropinquity{\mathcal{T}}((\A_1,\Lip_1),(\A_2,\Lip_2)) = \dpropinquity{\mathcal{T}}((\A_2,\Lip_2),(\A_1,\Lip_1))$,
\item $\dpropinquity{\mathcal{T}}((\A_1,\Lip_1),(\A_2,\Lip_2)) \leq \dpropinquity{\mathcal{T}}((\A_1,\Lip_1),(\A_3,\Lip_3)) + \dpropinquity{\mathcal{T}}((\A_3,\Lip_3),(\A_2,\Lip_2))$,
\item $\dpropinquity{\mathcal{T}}((\A_1,\Lip_1),(\A_2,\Lip_2)) = 0$ if and only if there exists a isometric isomorphism $\varphi : (\A_1,\Lip_1) \rightarrow (\A_2,\Lip_2)$.
\end{enumerate}
\end{theorem}

\begin{proof}
By Proposition (\ref{bounded-ncprop-prop}) and Theorem (\ref{coincidence-thm}), the dual $\mathcal{T}$-propinquity is null between two {\Lqcms s} if and only if they are isometrically isomorphic.

By Theorem (\ref{triangle-thm}), the dual $\mathcal{T}$-propinquity is symmetric and satisfies the triangle inequality. This completes our proof.
\end{proof}

\section{Examples of Convergences}

We prove a comparison theorem between our dual Gromov-Hausdorff propinquity and other (pseudo-)metrics on the class of {\Lqcms s}. From these comparisons, we then derive various convergence results valid for our distance. This section focuses on the dual Gromov-Hausdorff propinquity $\dpropinquity{} = \dpropinquity{\LQCMS}$, though it also applies to $\dpropinquity{\ast}$, as well as any $\mathcal{T}$-propinquity with enough tunnels to allow for the proofs below.

As it would require much exposition to detail the construction of each metric which appears in this section, we refer to \cite{Latremoliere13} for the definitions of treks, bridges as they relate to the quantum propinquity, to \cite{Rieffel00} for the construction of the quantum Gromov-Hausdorff distance, and to \cite{Rieffel10c} for the construction of the quantum proximity. A brief summary of the idea of the construction of the Gromov-Hausdorff distance, the quantum Gromov-Hausdorff distance and the quantum proximity is given in the introduction of this paper. We shall use the following notations:

\begin{notation}
The quantum Gromov-Hausdorff distance on the class of compact quantum metric spaces \cite{Rieffel00} is denoted by $\dist_q$. The quantum Gromov-Hausdorff distance restricts to a pseudo-metric on the class of {\Lqcms}.
\end{notation}

\begin{notation}
The quantum proximity, defined on the class of compact C*-metric spaces $\mathcal{CM}$ in \cite{Rieffel10c} (see Example (\ref{CM-example})), which is not known to be a pseudo-distance, is denoted by $\prox$. The dual Gromov-Hausdorff propinquity restricts to a distance on the class $\mathcal{CM}$ and our work allows for the construction of a specialized version of the dual Gromov-Hausdorff propinquity $\dpropinquity{\ast}$ whose journeys all involve only compact C*-metric spaces.
\end{notation}

\begin{notation}
Last, the quantum Gromov-Hausdorff propinquity of \cite{Latremoliere13} is denoted by $\propinquity$. The quantum propinquity can also be specialized, as its dual version in this paper, to various subclasses of {\Lqcms s}.
\end{notation}

We begin with a lemma which will be useful in relating our dual Gromov-Hausdorff propinquity with the quantum Gromov-Hausdorff distance.

\begin{lemma}\label{dist-q-lemma}
Let $\tau = (\D,\Lip_\D,\pi_\A,\pi_\B)$ be a tunnel from a {\Lqcms} $(\A,\Lip_\A)$ to a {\Lqcms} $(\B,\Lip_\B)$. Then:
\begin{equation*}
\dist_q((\A,\Lip_\A),(\B,\Lip_\B)) \leq \tunnelreach{\tau}{\Lip_\A,\Lip_\B}\text{.}
\end{equation*}
\end{lemma}

\begin{proof}
Let $\gamma > 0$ be given. For all $(d_1,d_2)\in\sa{\D\oplus\D}$, we set:
\begin{equation*}
\Lip_\gamma(d_1,d_2) = \max\left\{\Lip_\D(d_1),\Lip_\D(d_2),\frac{1}{\gamma}\|d_1-d_2\|_\D \right\}\text{.}
\end{equation*}
Now, for any $d\in \D$, we note that $\Lip_\gamma(d,d) = \Lip_\D(d)$, and thus $\Lip_\D$ is an admissible Lip-norm, in the sense of \cite{Rieffel00}, for the pair of Lip-norms $(\Lip_\D,\Lip_\D)$. 

Let $\pi_j : (d_1,d_2)\in \D\oplus\D\mapsto d_j$ for $j\in\{1,2\}$. A simple computation shows that if $\varphi \in \StateSpace(\D)$ then:
\begin{equation}\label{dist-q-lemma-eq-0}
\Kantorovich{\Lip_\gamma}(\varphi\circ\pi_1,\varphi\circ\pi_2)\leq \gamma\text{.}
\end{equation}
We note that this implies:
\begin{equation*}
\Haus{\Kantorovich{\Lip_\gamma}}\left(\pi_1^\ast\left(\StateSpace(\D)\right),\pi_2^\ast\left(\StateSpace(\D)\right)\right) \leq\gamma\text{.}
\end{equation*}
For our purpose, however, the following computation is what is needed. Let $\varphi\in\StateSpace(\A)$. Then there exists $\psi\in\StateSpace(\B)$ such that $\Kantorovich{\Lip_\D}(\varphi\circ\pi_\A,\psi\circ\pi_\B)\leq\tunnelreach{\tau}{\Lip_\A,\Lip_\B}$ by Definition (\ref{tunnel-reach-def}). Now, since $\pi_2^\ast$ is an isometry:
\begin{equation*}
\begin{split}
\Kantorovich{\Lip_\gamma}(\varphi\circ\pi_\A\circ\pi_1,\psi\circ\pi_\B\circ\pi_2)&\leq \Kantorovich{\Lip_\gamma}(\varphi\circ\pi_\A\circ\pi_1,\varphi\circ\pi_\A\circ\pi_2) \\
&\quad + \Kantorovich{\Lip_\gamma}(\varphi\circ\pi_\A\circ\pi_2,\psi\circ\pi_\B\circ\pi_2)\\
&\leq \gamma + \Kantorovich{\Lip_\D}(\varphi\circ\pi_\A,\psi\circ\pi_\B)\text{ by Equation (\ref{dist-q-lemma-eq-0}),}\\
&\leq\gamma + \tunnelreach{\tau}{\Lip_\A,\Lip_\B} \text{.}
\end{split}
\end{equation*}
By symmetry in the roles of $\A$ and $\B$, as well as $\pi_1$ and $\pi_2$, we obtain:
\begin{equation}\label{dist-q-lemma-eq-1}
\Haus{\Kantorovich{\Lip_\gamma}}\left( \pi_1^\ast\circ\pi_\A^\ast\left(\StateSpace(\A)\right), \pi_2^\ast\circ\pi_\B^\ast\left(\StateSpace(\B)\right)\right) \leq\gamma + \tunnelreach{\tau}{\Lip_\A,\Lip_\b} \text{.}
\end{equation}

Now, let $\pi = (\pi_\A\circ\pi_1,\pi_\B\circ\pi_2) : \D\oplus\D \rightarrow \A\oplus\B$. By construction, $\pi$ is a *-epimorphism. Let $\Lip'_\gamma$ be the quotient seminorm from $\Lip_\gamma$ for $\pi$. By \cite[Proposition 3.1]{Rieffel00}, the Lipschitz pair $(\A\oplus\B,\Lip'_\gamma)$ is a compact quantum metric space and $\pi^\ast$ is an isometry from $\StateSpace(\A\oplus\B)$ into $\StateSpace(\D\oplus\D)$.

As a remark, we note that $\Lip'_\gamma$ may not possess the Leibniz property, though it is closed by \cite[Proposition 3.3]{Rieffel00}. We shall not need this remark in what follows; however this observation justifies that we do not solely work with tunnels in the class $\mathcal{SUM}$.

Since $\pi^\ast$ is an isometry, the Hausdorff distance between $\StateSpace(\A)$ and $\StateSpace(\B)$ for $\Kantorovich{\Lip'_\gamma}$ is no more than $\gamma + \tunnelreach{\tau}{\Lip,\Lip}$ by Equation (\ref{dist-q-lemma-eq-1}). Moreover, by \cite[Proposition 3.7]{Rieffel00}, the Lip-norm $\Lip'_\gamma$ is admissible for $(\Lip_\A,\Lip_\B)$. Hence by the definition of the quantum Gromov-Hausdorff distance $\dist_q$, we have:
\begin{equation*}
\dist_q(\Lip_\A,\Lip_\B)\leq \tunnelreach{\tau}{\Lip_\A,\Lip_\B}+\gamma\text{.}
\end{equation*}
As $\gamma>0$ is arbitrary, our lemma is proven.
\end{proof}

\begin{theorem}\label{ncpropinquity-comparison-thm}
Let $(\A,\Lip_\A)$ and $(\B,\Lip_\B)$ be two {\Lqcms s}. Then:
\begin{equation}\label{comparison-eq-1}
\mathrm{dist}_q((\A,\Lip_\A),(\B,\Lip_\B))\leq \dpropinquity{}((\A,\Lip_\A),(\B,\Lip_\B)) \leq 2\propinquity((\A,\Lip_\A),(\B,\Lip_\B)) \text{.}
\end{equation}
If $(\A,\Lip_\A)$ and $(\B,\Lip_\B)$ are in some class $\mathcal{C}$ of {\Lqcms s} and $\mathcal{T}\subseteq\mathcal{G}$ are two $\mathcal{C}$-compatible classes of tunnels, then:
\begin{equation}\label{comparison-eq-2}
\dpropinquity{\mathcal{G}}((\A,\Lip_\A),(\B,\Lip_\B))\leq\dpropinquity{\mathcal{T}}((\A,\Lip_\A),(\B,\Lip_\B)) \leq \propinquity_{\mathcal{C}}((\A,\Lip_\A),(\B,\Lip_\B)) \text{.}
\end{equation}

Moreover, if $(\A,\Lip_\A)$ and $(\B,\Lip_\B)$ are both compact C*-metric spaces, then:
\begin{equation}\label{comparison-eq-3}
\dpropinquity{}((\A,\Lip_\A),(\B,\Lip_\B)) \leq \dpropinquity{\ast}((\A,\Lip_\A),(\B,\Lip_\B)) \leq \prox((\A,\Lip),(\B,\Lip_\B)) \text{.}
\end{equation}
\end{theorem}

\begin{proof}
By Definition (\ref{dual-propinquity-def}), if $\mathcal{T}\subseteq\mathcal{G}$ are $\mathcal{C}$-compatible classes of tunnels, and $(\A,\Lip_\A)$ and $(\B,\Lip_\B)$ lie in $\mathcal{C}$, then:
\begin{equation*}
\dpropinquity{\mathcal{G}}((\A,\Lip_\A),(\B,\Lip_\B))\leq\dpropinquity{\mathcal{T}}((\A,\Lip_\A),(\B,\Lip_\B))\text{,}
\end{equation*}
i.e. the first halves of Inequality (\ref{comparison-eq-2}) and Inequality (\ref{comparison-eq-3}) hold.

By construction, $\dpropinquity{\ast}$ is dominated by Rieffel's proximity, and thus Inequality (\ref{comparison-eq-3}) is proven.

Let us now prove that the quantum Gromov-Hausdorff propinquity dominates the dual Gromov-Hausdorff propinquity.
 
Let $\varepsilon > 0$. We recall from \cite{Latremoliere13} that the quantum propinquity is computed as the infimum of the lengths of all treks between two given {\Lqcms s}. Let $\Gamma \in \trekset{\A,\Lip_\A}{\B,\Lip_\B}$ be a trek from $(\A,\Lip_\A)$ to $(\B,\Lip_\B)$ such that the length $\treklength{\Gamma}$ of $\Gamma$ satisfies $\treklength{\Gamma} \leq \propinquity((\A,\Lip_\A),(\B,\Lip_\B)) +\varepsilon$. 

Write $\Gamma = (\A_j,\Lip_j,\gamma_j,\A_{j+1},\Lip_{j+1} : j=1,\ldots,n)$, where $\gamma_j$ is a bridge from $(\A_j,\Lip_j)$ to $(\A_{j+1},\Lip_{j+1})$ for $j\in\{1,\ldots,n\}$ with $(\A,\Lip_\A)=(\A_1,\Lip_1)$ and $(\B,\Lip_\B)=(\A_{n+1},\Lip_{n+1})$.

For each $j\in\{1,\ldots,n\}$, $a_j\in\sa{\A_j}$ and $a_{j+1}\in\sa{\A_{j+1}}$, we set:
\begin{equation*}
\Lip^j(a_j,a_{j+1}) = \max\left\{ \Lip_j(a_j),\Lip_{j+1}(a_{j+1}),\frac{1}{\bridgelength{\gamma_j}{\Lip_j,\Lip_{j+1}}}\bridgenorm{\gamma_j}{a_j,a_{j+1}} \right\}\text{,}
\end{equation*}
where, for all $j\in\{1,\ldots,n\}$, the real number $\bridgelength{\gamma_j}{\Lip_j,\Lip_{j+1}}$ is the length of the bridge $\gamma_j$ as defined in \cite[Definition 3.17]{Latremoliere13}, and $\bridgenorm{\gamma_j}{\cdot,\cdot}$ is the bridge seminorm of $\gamma_j$ as defined in \cite[Definition 3.10]{Latremoliere13}.

By \cite[Theorem 6.3]{Latremoliere13}, $(\A_j\oplus\A_{j+1},\Lip^j)$ is a {\Lqcms} and $\tau_j = (\A_j\oplus\A_{j+1},\Lip^j,\pi^l_j,\pi^r_j)$ is a tunnel from $(\A_j,\Lip_j)$ to $(\A_{j+1},\Lip_{j+1})$, with $\pi^r_l$ and $\pi^l_r$ being the canonical surjections from $\A_j\oplus\A_{j+1}$ onto $\A_j$ and onto $\A_{j+1}$ respectively. Thus $\Upsilon = (\A_j,\Lip_j,\tau_j,\A_{j+1},\Lip_{j+1} : j=1,\ldots,n)$ is a journey from $(\A,\Lip_\A)$ to $(\B,\Lip_\B)$.

The depth of $\tau_j$ is $0$ for all $j\in\{1,\ldots,n\}$ by Remark (\ref{null-depth-rmk}). Moreover, by \cite[Theorem 6.3]{Latremoliere13}, the reach of the tunnel $\tau_j$ is bounded above by twice the length of the bridge $\gamma_j$ for all $j\in\{1,\ldots,n\}$. In conclusion, $\tunnellength{\tau_j}{\Lip_j,\Lip_{j+1}}\leq 2\bridgelength{\gamma_j}{\Lip_j,\Lip_{j+1}}$ for all $j\in\{1,\ldots,n\}$. Thus $\journeylength{\Upsilon}\leq 2\treklength{\Gamma}$ and thus by Definition (\ref{dual-propinquity-def}), we have:
\begin{equation*}
\dpropinquity{}((\A,\Lip_\A),(\B,\Lip_\B)) \leq 2\propinquity((\A,\Lip_\A),(\B,\Lip_\B)) + \varepsilon\text{.}
\end{equation*}
As $\varepsilon > 0$ is arbitrary, we have established the desired upper bound on $\dpropinquity{}$. The proof carries if we restrict the class of {\Lqcms s} and tunnels, so the upper bound in Inequality (\ref{comparison-eq-1}) is established.

The claimed lower bound in Inequality (\ref{comparison-eq-1}) proceeds as follows. Let $\varepsilon > 0$ and let:
\begin{equation*}
\Upsilon = (\A_j,\Lip_j,\tau_j,\A_{j+1},\Lip_{j+1}:j=1\,\ldots,n) \in \journeyset{\mathcal{T}}{\A,\Lip_\A}{\B,\Lip_\B}
\end{equation*}
such that:
\begin{equation*}
\journeylength{\Upsilon} = \sum_{j=1}^n \tunnellength{\tau_j}{\Lip_j,\Lip_{j+1}} \leq \dpropinquity{}((\A,\Lip_\A),(\B,\Lip_\B)) + \varepsilon \text{.}
\end{equation*}

Thus, since the quantum Gromov-Hausdorff distance satisfies the triangle inequality:
\begin{equation*}
\begin{split}
\mathrm{dist}_q((\A,\Lip_\A),(\B,\Lip_\B)) &\leq \sum_{j=1}^n \mathrm{dist}_q((\A_j,\Lip_j),(\A_{j+1},\Lip_{j+1}))\\
&\leq \sum_{j=1}^n \tunnelreach{\tau_j}{\Lip_j,\Lip_{j+1}} \text{ by Lemma (\ref{dist-q-lemma}),}\\
&\leq \sum_{j=1}^n \tunnellength{\tau_j}{\Lip_j,\Lip_{j+1}}\\
&\leq \dpropinquity{}((\A,\Lip_\A),(\B,\Lip_\B)) + \varepsilon \text{.}
\end{split}
\end{equation*}
Our theorem is thus proven as $\varepsilon > 0$ is arbitrary.
\end{proof}

We first obtain a more precise bound for the dual Gromov-Hausdorff propinquity using the known bound on the quantum propinquity:

\begin{corollary}
For any two {\Lqcms s} $(\A,\Lip_\A)$ and $(\B,\Lip_\B)$ we have:
\begin{equation*}
\dpropinquity{}((\A,\Lip_\A),(\B,\Lip_\B)) \leq \max\left\{\diam{\StateSpace(\A)}{\Kantorovich{\Lip_\A}},\diam{\StateSpace(\B)}{\Kantorovich{\Lip_\B}}\right\}\text{.}
\end{equation*}
\end{corollary}

\begin{proof}
Apply Theorem (\ref{ncpropinquity-comparison-thm}) to \cite[Proposition 4.6]{Latremoliere13}.
\end{proof}

We also can compare the dual Gromov-Hausdorff propinquity with the Gromov-Hausdorff distance:

\begin{corollary}
Let $(X,\mathsf{d}_X)$ and $(Y,\mathsf{d}_Y)$ be two compact metric spaces, and let $\mathsf{GH}$ be the Gro\-mov-Haus\-dorff distance \cite{Gromov}. Then:
\begin{equation*}
\dpropinquity{}((C(X),\Lip_X),(C(Y),\Lip_Y)) \leq \mathsf{GH}((X,\mathsf{d}_X),(Y,\mathsf{d}_Y))\text{,}
\end{equation*}
where $\Lip_X$ and $\Lip_Y$ are, respectively, the Lipschitz seminorms associated to $\mathsf{d}_X$ and $\mathsf{d}_Y$.
\end{corollary}

\begin{proof}
This follows from Theorem (\ref{ncpropinquity-comparison-thm}) applied to \cite[Theorem 6.6]{Latremoliere13}.
\end{proof}

We pause to comment on the relationship between the dual Gromov-Hausdorff propinquity and some other metrics in noncommutative metric geometry. The unital nuclear distance of Kerr and Li \cite{kerr09} dominates the quantum propinquity \cite{Latremoliere13}, hence it also dominates the dual Gromov-Hausdorff propinquity. It is very possible to define a matricial version of the dual Gromov-Hausdorff propinquity in the spirit of the work of Kerr in \cite{kerr02}, and such a metric would dominates Kerr's distance for the same reasons as our current version dominates Rieffel's metric. Our dual Gromov-Hausdorff propinquity solves the matter of the coincidence property of the quantum Gromov-Hausdorff distance in a different manner from Kerr's metric, which employs completely positive order-isomorphisms for this purpose; it may however prove useful to work with a matricial version of the dual Gromov-Hausdorff propinquity in the future. It is not clear how our dual Gromov-Hausdorff propinquity compares to Kerr's matricial distance. Last, another possible avenue for generalization of our current metric is to define a quantized dual Gromov-Hausdorff propinquity in analogy with the quantized Gromov-Hausdorff distance of \cite{Wu06b}, and again, the quantized dual Gromov-Hausdorff propinquity would dominate Wu's distance (on the class on which it would be defined, which presumably would be the class consisting of pairs of a C*-algebra $\A$ endowed with a sequence of Leibniz Lip-norms, one for each matrix algebra with entries in $\A$).

We now propose two examples of convergences for the dual Gromov-Hausdorff propinquity, both derived from earlier results about stronger metrics. We start with families of {\Lqcms s} given by Examples (\ref{qt-example}).

\begin{theorem}\label{qt-converge-thm}
Let $H_\infty$ be a compact Abelian group endowed with a continuous length function $\ell$. Let $(H_n)_{n\in\N}$ be a sequence of closed subgroups of $H_\infty$ converging to $H_\infty$ for the Hausdorff distance induced by $\ell$ on the class of closed subsets of $H_\infty$. Let $\sigma_\infty$ be a skew bicharacter of the Pontryagin dual $\widehat{H_\infty}$. For each $n\in\N$, we let $\sigma_n$ be a skew bicharacter of $\widehat{H_n}$, which we implicitly identity with its unique lift as a skew bicharacter of $\widehat{H_\infty}$. If $(\sigma_n)_{n\in\N}$ converges pointwise to $\sigma_\infty$, then:
\begin{equation*}
\lim_{n\rightarrow\infty} \dpropinquity{}\left(\left(C^\ast\left(\widehat{H_n},\sigma_n\right),\Lip_n\right),\left(C^\ast\left(\widehat{H_\infty},\sigma_\infty\right),\Lip_\infty\right)\right) = 0\text{,}
\end{equation*}
where for all $n\in\N\cup\{\infty\}$ and $a\in C^\ast\left(\widehat{H_n},\sigma_n\right)$ we set:
\begin{equation*}
\Lip_n (a) = \sup\left\{\frac{\|a-\alpha_n^g(a)\|_{C^\ast\left(\widehat{H_n},\sigma_n\right)}}{\ell(g)} : g\in H_n\setminus\{1_{H_\infty}\} \right\}
\end{equation*}
with $1_{H_\infty}$ is the unit of $H_\infty$ and $\alpha_n$ is the dual action of $H_n$ on $C^\ast\left(\widehat{H_n},\sigma_n\right)$.
\end{theorem}

\begin{proof}
Apply \cite[Theorem 6.8]{Latremoliere13}, which itself derives from the fact that the unital nuclear distance \cite{kerr09} dominates the quantum Gromov-Hausdorff distance, and the methods of \cite{kerr09} and \cite{Latremoliere05} combined.
\end{proof}

In particular, we obtain the following corollary about quantum tori and their finite dimensional approximations:

\begin{corollary}
Let $d\in \N\setminus\{0,1\}$. Let $\Nbar = \N\cup\{\infty\}$ be the one point compactification of $\N$ and we endow the space $B$ of skew-bicharacters of $\Z^d$ with the topology of pointwise convergence. For any $k = (k_1,\ldots,k_d)\in \Nbar^d$, we denote the quotient group $\bigslant{\Z^d}{\prod_{j=1}^d k_j\Z}$ by $\Z^d_k$, where by convention $\infty\Z = \{0\}$. Every skew-bicharacter of $\Z^d_k$ is identified with its unique lift to $B$. We endow $\Nbar^d\times B$ with the product topology.

Let $\sigma$ a skew-bicharacter of $\Z^d$. Write $\infty^d = (\infty,\ldots,\infty) \in \Nbar^d$. Let $l$ be a continuous length function on $\T^d$ where $\T = \{z\in\C : |z|=1\}$. For all $c\in\N^d$, $\theta\in B$, let $\Lip_{l,c,\theta}$ be
the Lip-norm on the twisted C*-algebra $C^\ast(\Z^d_c,\theta)$ defined, for all $a\in\C^\ast(\Z^d_k,\theta)$ by:
\begin{equation*}
\Lip_{l,c,\theta}(a) = \sup \left\{ \frac{\|\alpha^g(a)-a\|_{C^\ast(\Z^d_k,\theta)}}{l(g)} : g \not= (1,\ldots,1) \right\}
\end{equation*}
where $\alpha$ is the dual action of the dual of $\Z^d_k$, seen as a subgroup of $\T^d$, on $C^\ast(\Z^d,\theta)$.

Then:
\begin{equation*}
\lim_{(c,\theta)\rightarrow (\infty^d,\sigma)} \dpropinquity{}\left(\left(C^\ast\left(\Z^d_c,\theta\right),\Lip_{l,c,\theta}\right),\left(C^\ast\left(\Z^d,\sigma\right),\Lip_{l,\infty^d,\sigma}\right)\right) = 0 \text{.}
\end{equation*}

\end{corollary}

\begin{proof}
This result may be derived from Theorem (\ref{qt-converge-thm}), or from Theorem (\ref{ncpropinquity-comparison-thm}) applied to the work in \cite{Latremoliere13c}, which provides a more explicit and quantitative construction of bridges than the implicit approach of \cite{kerr09}.
\end{proof}

We can also recast Rieffel's work on convergences of matrix algebras to the sphere, as he develops in \cite{Rieffel10c} the necessary estimates to obtain the desired convergence for the proximity, and hence our dual Gromov-Hausdorff propinquity.

\begin{corollary}\label{rieffel-corollary}
Let $G$ be a compact connected Lie semisimple group and $\ell$ a continuous length function on $G$. Let $\pi$ be an irreducible unitary representation of $G$ on some Hilbert space $\mathscr{H}$. Let $\xi$ be a normalized highest weight vector for $\pi$. For any $n\in\N$, we denote by $\pi_n$ the restriction of $\pi^{\otimes n}$ to the irreducible space $\mathscr{H}_n$ for $\pi^{\otimes n}$ containing $\xi^{\otimes n}$. Note that $\pi_n$ is an irreducible representation for $G$ of highest weight vector $\xi^{\otimes n}$.

Let $\B_n = \B(\mathscr{H}_n)$ be the matrix algebras of linear endomorphisms of $\mathscr{H}_n$ and let $\alpha$ be the action of $G$ on $\B_n$ by conjugation.Let $H$ be the stability subgroup of the projection on $\C\xi^{\otimes n}$ and note that $H$ is independent of $n\in\N$.

Now, $G$ acts continuously and ergodically on both $A = C\left(\bigslant{G}{H}\right)$ and $\B_n$, thus defining Leibniz closed Lip-norms $\Lip_\A$ and $\Lip_n$ on $\A$ and $\B_n$, respectively. We then have:
\begin{equation*}
\lim_{n\rightarrow\infty} \dpropinquity{}((\A,\Lip_\A),(\B_n,\Lip_n)) = 0 \text{.}
\end{equation*}
\end{corollary}

\begin{proof}
Apply Theorem (\ref{ncpropinquity-comparison-thm}) to \cite[Theorem 9.1]{Rieffel10c}.
\end{proof}

We shall observe, at the conclusion of this paper, that Rieffel proved in \cite{Rieffel10} that the finite dimensional {\Lqcms s} which appear in Corollary (\ref{rieffel-corollary}) form a Cauchy sequence for the quantum proximity. As the latter does not satisfy the triangle inequality, this observation does not derive from \cite[Theorem 9.1]{Rieffel10c}, but instead uses new estimates based upon techniques which, in turn, inspired our work in \cite{Latremoliere13}. Since we are now going to prove that the dual Gromov-Hausdorff propinquity is complete, Rieffel's result in \cite{Rieffel10} provides an alternative approach to Corollary (\ref{rieffel-corollary}).

\section{Completeness}

In this last section, we prove that the dual Gromov-Hausdorff propinquity is complete. This property is a strong motivation for our introduction of this metric. Indeed, the Gromov-Hausdorff distance is a complete metric on the class of compact metric spaces \cite{Gromov}, which is a very important property in metric geometry. Rieffel's quantum Gromov-Hausdorff is complete as well \cite{Rieffel00}. However, our quantum propinquity is not known to be complete, and our investigation in this matter reveals difficulties tied to the restrictions we placed on the type of Lip-norms involved in \cite{Latremoliere13}. The dual Gromov-Hausdorff propinquity, which we have so far proven satisfies all desirable properties of the quantum propinquity, is the result of our work on completeness for our new metrics, a property we very strongly feel is necessary for any analogue to the Gromov-Hausdorff distance in noncommutative geometry.

 Our work focuses on the dual Gromov-Hausdorff propinquity $\dpropinquity{}$, though it is possible to extend it to various specialized noncommutative propinquities, with some care.

In order to lighten our exposition, we group some common hypothesis to many results in this section. We begin with a special type of sequences of {\Lqcms s} and some associated tunnels, which will prove sufficient to establish the completeness of the dual Gromov-Hausdorff propinquity, as we will see at the end of this section: indeed, every Cauchy sequence contains a subsequence with the properties listed in Hypothesis (\ref{completeness-hypothesis-2}). The reason to choose sequences with the properties listed in Hypothesis (\ref{completeness-hypothesis-2}) is that it allows for the construction of {\Lqcms s} which will be used as tunnels. The content of Hypothesis (\ref{completeness-hypothesis-1}) defines all the basic objects which will be used repeatedly in this section to construct the limit of a Cauchy sequence of {\Lqcms s} for $\dpropinquity{}$.

\newcommand{\MLip}{{\mathsf{H}}}
\newcommand{\SLip}{{\mathsf{S}}}

\begin{notation}
Let $N \in \N$ and set $\N_N = \{ n\in \N : n\geq N \}$. If $(\A_n)_{n\in\N_N}$ is a sequence of C*-algebras, then $\prod_{n\in\N_N} \A_n$ is the C*-algebra of bounded sequences $(a_n)_{n\in\N_N}$ with $a_n\in\A_n$ for all $n\in\N_N$, with norm given by:
\begin{equation*}
\|(a_n)_{n\in\N_N}\|_\infty = \sup \{ \|a_n\|_{\A_n} : n\in \N_N \} < \infty\text{.}
\end{equation*}
\end{notation}

We begin by defining some structures associated with sequences of {\Lqcms s} and associated tunnels. We will see at the end of this section that we can work with tunnels, instead of journeys, toward our goal. For now, we shall assume in this section:

\begin{hypothesis}\label{completeness-hypothesis-1}
Let $(\A_n,\Lip_n)_{n\in\N}$ be a sequence of {\Lqcms s} and, for each $n\in\N$, let:
\begin{equation*}
\tau_n\in\tunnelset{\A_n,\Lip_n}{\A_{n+1},\Lip_{n+1}}{}\text{.}
\end{equation*}
For each $n\in\N$ we write $\tau_n = (\D_n,\Lip^n,\pi_n,\omega_n)$, where $(\D_n,\Lip^n)$ is a {\Lqcms} and $\pi_n : \D_n\twoheadrightarrow\A_n$ and $\omega_n : \D_n\twoheadrightarrow\A_{n+1}$ are *-epimorphisms.

Let $N \in \N$. Define:
\begin{equation*}
\SLip_N : (d_n)_{n\in\N_N} \in \sa{\prod_{n\in\N_N} \D_n} \longmapsto \sup \left\{\Lip^n(d_n) : n\in \N_N \right\}\text{.}
\end{equation*}

Let:
\begin{equation*}
\mathfrak{L}_N = \left\{ d=(d_n)_{n\in\N_N} \in \prod_{n\in\N_N}\D_n \middle\vert \begin{array}{l} \forall n\in \N_N\quad
d_n\in\sa{\D_n},\\
\SLip_N(d)<\infty \text{ and }\\
\forall n \in \N_N \quad \pi_{n+1}(d_{n+1})=\omega_n(d_n)
\end{array} \right\}\text{.}
\end{equation*}
For all $d\in \prod_{n\in\N_N}\D_n$, we denote $\frac{1}{2}(d+d^\ast)$ by $\Re(d)$ and $\frac{1}{2i}(d-d^\ast)$ by $\Im(d)$. We then define:
\begin{equation*}
\mathfrak{S}_N = \left\{ d \in \prod_{n\in\N_N}\D_n : \Re(d),\Im(d) \in \mathfrak{L}_N \right\}\text{.}
\end{equation*}
We also let $\mathfrak{G}_N$ be the norm closure of $\mathfrak{S}_N$ in $\prod_{n\in\N_N}\D_n$.

For all $n\in\N_N$, we denote the projection $(d_n)_{n\in\N_N} \in \mathfrak{G}_N \mapsto d_n\in\D_n$ by $\Xi_n$, and we denote $\pi_n\circ\Xi_n : \mathfrak{G} \rightarrow \A_n$ by $\Pi_n$. Note that we do not decorate either $\Xi_n$ or $\Pi_n$ with $N$, as there is no risk of ambiguity.
\end{hypothesis}

Outside of the results and proofs in this section, where we will rigorously refer to the proper set of assumptions, the text in the rest of our paper will implicitly use the notations in Hypothesis (\ref{completeness-hypothesis-1}).

The spaces $\alg{G}_N$ introduced in Hypothesis (\ref{completeness-hypothesis-1}) will be used to construct new {\Lqcms s} and tunnels. The first step toward this goal is to establish the algebraic structures of these spaces.

\begin{lemma}\label{jordan-lie-lemma}
We assume Hypothesis (\ref{completeness-hypothesis-1}).

Then $\mathfrak{L}_N$ is a Jordan-Lie subalgebra of $\sa{\prod_{n\in\N_N}\D_n}$, while $\mathfrak{S}_N$ is the linear span of $\mathfrak{L}_N\cup i\mathfrak{L}_N$ in $\prod_{n\in\N_N}\D_n$ and is a *-subalgebra of $\prod_{n\in\N_N}\D_n$. Thus $\mathfrak{G}_N$ is a C*-subalgebra of $\prod_{n\in\N_N}\D_n$.
\end{lemma}

\begin{proof}
The unit $(\unit_{\D_n})_{n\in\N_N}$ lies in $\alg{L}_N$. Let $(d_n)_{n\in\N_N}$, $(d'_n)_{n\in\N_N}$ be elements in $\mathfrak{L}_N$ and let $t\in\R$. Then:
\begin{equation*}
\Lip^n(d_n+td'_n)\leq \SLip_N((d_n)_{n\in\N}) + |t|\SLip_N((d'_n)_{n\in\N_N})
\end{equation*}
for all $n\in\N_N$, so $\SLip_N((d_n+td'_n)_{n\in\N}) < \infty$. Moreover:
\begin{equation*}
\begin{split}
\pi_{n+1}(d_{n+1}+td'_{n+1}) &= \pi_{n+1}(d_{n+1}) + t\pi_{n+1}(d'_{n+1}) \\
&= \omega_n(d_n)+t\omega_n(d'_n) = \omega_n(d_n+td'_n)
\end{split}
\end{equation*}
for all $n\in\N_N$. Hence $(d_n+td'_n)_{n\in\N_N} \in \mathfrak{L}_N$ as desired.

Moreover, for all $n\in\N_N$:
\begin{equation*}
\begin{split}
\Lip^n(\Jordan{d_n}{d'_n}) &\leq \|d_n\|_{\D_n}\Lip^n(d'_n) + \|d'_n\|_{\D_n}\Lip^n(d_n)\\
&\leq \|(d_n)_{n\in\N_N}\|_\infty \SLip_N((d_n')_{n\in\N_N}) + \|(d'_n)_{n\in\N_N}\|_\infty\SLip_N((d_n)_{n\in\N_N}) < \infty
\end{split}
\end{equation*}
and, $\pi_{n+1}(\Jordan{d_{n+1}}{d'_{n+1}}) = \omega_n(\Jordan{d_n}{d'_n})$. The same computation holds for the Lie product, so $\mathfrak{L}_N$ is closed under the Jordan and Lie product, as desired.

It is now easy to check that $\mathfrak{S}_N$ is a *-subalgebra of $\prod_{n\in\N_N}\D_n$, and that it must be the smallest such *-algebra (in fact, the smallest linear space) containing $\mathfrak{L}_N\cup i\mathfrak{L}_N$. This concludes our lemma.
\end{proof}

We continue progressing toward the construction of new {\Lqcms s} based on the spaces $\alg{G}_N$ with the following result:

\begin{lemma}\label{lipschitz-pair-lemma}
Assume Hypothesis (\ref{completeness-hypothesis-1}). The pair $(\alg{G}_N,\SLip_N)$ is a Lipschitz pair.
\end{lemma}

\begin{proof}
Let $d=(d_n)_{n\in\N_N}\in\sa{\alg{G}_N}$ with $\SLip_N(d) = 0$. Then for all $n\geq N$, we have $\Lip^n(d_n) = 0$, thus $d_n=t_n \unit_{\D_n}$ for some $t_n \in \R$. Now, for all $n\in\N_N$ we have $t_{n+1}\unit_{\A_{n+1}} = \pi_{n+1}(d_{n+1}) = \omega_n(d_n) = t_n\unit_{\A_{n}}$ and thus $d = t_0 \unit_{\alg{G}_N}$ as expected.

The domain of $\SLip_N$ is dense in $\sa{\alg{G}_N}$ by definition of $\alg{G}_N$, thus our proof is completed.
\end{proof}

By Lemma (\ref{lipschitz-pair-lemma}), $\Kantorovich{\SLip_N}$ is a metric on $\StateSpace(\alg{G}_N)$ whose topology is finer than the weak* topology \cite{Rieffel98a}, though it may not metrize the weak* topology, or even be bounded. We shall prove in the next few lemmas that $(\alg{G}_N,\SLip_N)$ is indeed a {\Lqcms} and can be used to build tunnels, under the following additional assumption:

\begin{hypothesis}\label{completeness-hypothesis-2}
Assume Hypothesis (\ref{completeness-hypothesis-1}), as well as $\sum_{j=0}^\infty \tunnellength{\tau_j}{\Lip_j,\Lip_{j+1}} < \infty$, and we denote $\sum_{j=N}^\infty \tunnellength{\tau_j}{\Lip_j,\Lip_{j+1}}$ by $M_N$.
\end{hypothesis}

This assumption is natural in our context, as in particular it would seem difficult to prove that $\Kantorovich{\SLip_N}$ is bounded without it. Note that in particular, the sequence $(\A_n,\Lip_n)_{n\in\N}$ given by Hypothesis (\ref{completeness-hypothesis-2}) is Cauchy for the dual Gromov-Hausdorff propinquity. We thus hope to show, in the next few pages, that this sequence indeed has a limit for the dual Gromov-Hausdorff propinquity.

As a first step, we show that we can lift elements from $\D_n$ to $\mathfrak{G}_N$ with $n\geq N$, while keeping both the Lip-norm and norm close to the ones of the lifted element:

\begin{lemma}\label{surjection-lemma}
Assume Hypothesis (\ref{completeness-hypothesis-2}). Let $n\in\N_N$, $d_n \in \sa{\D_n}$ and $\varepsilon > 0$. There exists $d \in \mathfrak{L}_N$ such that $\Xi_n(d) = d_n$, $\Lip^n(d_n)\leq\SLip_N(d)\leq \Lip^n(d_n)+\varepsilon$ and:
\begin{equation*}
\|d\|_{\mathfrak{G}_N} \leq \|d_n\|_{\D_n} + 2(\Lip^n(d_n)+\varepsilon)M_N\text{.}
\end{equation*}
Moreover, if $d \in \alg{G}_N$ and $k \in \{0,\ldots,N\}$ then there exists $f=(f_n)_{n\in\N_k}\in \mathfrak{G}_k$ such that $(f_n)_{n\in\N_N} = (d_n)_{n\in\N_N}$ and:
\begin{equation*}
\SLip_N(d)\leq\SLip_k(f)\leq\SLip_N(d)+\varepsilon\text{ and }\|f\|_{\alg{G}_k}\leq \|d\|_{\alg{G}_N} + 2(\SLip_N(d)+\varepsilon)\sum_{j=k}^{N-1}\tunnellength{\tau_j}{\Lip_j,\Lip_{j+1}}\text{.}
\end{equation*}
\end{lemma}

\begin{proof}
We proceed by induction. To simplify notations, let $s_n = \sum_{j=1}^n\frac{1}{2^j}$ and, by convention, $s_{-1} = s_0 = 0$. 

Assume that for some $K\geq n$, we have $(d_n,\ldots,d_K)\in\oplus_{j=n}^K\sa{\D_j}$ such that $\omega_j(d_j)=\pi_{j+1}(d_{j+1})$ for all $j$ such that $n\leq j\leq K-1$, and:
\begin{equation*}
\Lip^j(d_j)\leq \Lip^n(d_n)+s_{j-n}\varepsilon
\end{equation*}
while:
\begin{equation*}
\|d_j\|_{\D_j}\leq \|d_n\|_{\D_n} + 2(\Lip^n(d_n)+s_{j-n-1}\varepsilon)\sum_{j=n+1}^j\tunnellength{\tau_j}{\Lip_j,\Lip_{j+1}}
\end{equation*}
for all $j\in\{n,\ldots,K\}$. Note that this assumption is met for $K = n$ trivially.

Let $b = \omega_K(d_K) \in \sa{\A_{K+1}}$. By Definition (\ref{tunnel-def}), there exists $d_{K+1} \in \sa{\D_{K+1}}$ such that $\pi_{K+1}(d_{K+1})=b$ and $\Lip^{K+1}(d_{K+1})\leq \Lip_{K+1}(b)+\frac{1}{2^{K+1-n}}\varepsilon$. Since $\Lip_{K+1}(b)\leq\Lip^K(d_K)$, we conclude that $\Lip^{K+1}(d_{K+1})\leq \Lip^n(d_n)+s_{K+1-n}\varepsilon$.

Moreover, by Proposition (\ref{tunnel-lift-bound-prop}), we have:
\begin{equation*}
\begin{split}
\|d_{K+1}\|_{\D_{K+1}} &\leq \|b\|_{\D_K} + 2\Lip_{K+1}(b)\tunnelreach{\tau_{K+1}}{\Lip_K,\Lip_{K+1}}\\
&\leq \|d_K\|_{\D_K} + 2\Lip^K(d_K)\tunnelreach{\tau_{K+1}}{\Lip_K,\Lip_{K+1}}\\
&\leq \|d_n\|_{\D_n} + 2(\Lip^n(d_n)+s_{K-n}\varepsilon)\sum_{j=n+1}^{K+1} \tunnellength{\tau_j}{\Lip_j,\Lip_{j+1}}\text{.}
\end{split}
\end{equation*}

Thus, our induction is complete and we have obtained a sequence $(d_j)_{j\geq K} \in \prod_{j\in\N_n}\D_j$ with the desired property. 

We can complete this sequence to a sequence of $\mathfrak{G}_N$ by a similar (finite) induction for $j\in\{N,\ldots,n-1\}$, which will prove the second assertion of our lemma as well. We thus assume given a sequence $(d_j)_{j\in\N_n}\in\alg{L}_n$ and we wish to find an extension to $\alg{L}_N$ for some $N\leq n$. Assume that for some $j \in \{N+1,\ldots,n\}$, we have constructed $d_j\in\sa{\D_j}$ such that:
\begin{align*}
\Lip^n(d_n) &\leq \Lip^j(d_j) \leq \Lip^n(d_n)+s_{n-j+1}\varepsilon\\
\intertext{ and }
\|d_j\|_{\D_j} &\leq \|d_n\|_{\alg{G}_n} + 2(\Lip^n(d_n)+s_{n-j+1}\varepsilon)\sum_{k=j}^{n-1}\tunnellength{\tau_k}{\Lip_k,\Lip_{k+1}}\text{.}
\end{align*}
These assumptions are met for $j=n$ trivially. Now, let $a = \pi_j(d_j)$. Then there exists $d_{j-1} \in \sa{\D_{j-1}}$ such that $\omega(d_{j-1}) = a$ and:
\begin{equation*}
\Lip^{j-1}(d_{j-1})\leq \Lip_j(a)+\frac{1}{2^j}\varepsilon\leq\Lip^{j-1}(d_{j-1})+\frac{1}{2^{n-j}}\varepsilon\text{.}
\end{equation*}
Now $\Lip^{j-1}(d_{j-1})\leq \Lip^n(d_n) + s_{n-j}\varepsilon$. Moreover, we conclude from Proposition (\ref{tunnel-lift-bound-prop}) that:
\begin{equation*}
\|d_{j-1}\|_{\D_{j-1}} \leq \|a\|_{\A_j} + 2\Lip_j(a)\tunnellength{\tau_{j-1}}{\Lip_{j-1},\Lip_j}\text{.}
\end{equation*}

Hence:
\begin{equation*}
\|d_{j-1}\|_{\D_{j-1}} \leq \|d_n\|_{\D_n} + 2(\Lip^n(d_n)+s_{n-j}\varepsilon)\sum_{k=j-1}^{n-1} \tunnellength{\tau_k}{\Lip_k,\Lip_{k+1}}\text{.}
\end{equation*}
Thus our induction is complete.
\end{proof}

We thus establish that the quotient seminorms of $\SLip_N$ for the canonical projections $\Xi_n$ and $\Pi_n$ ($n\in\N_N$) are given by the Lip-norms on $\D_n$ and $\A_n$, respectively: thus, $(\alg{G}_N,\SLip_N)$ can be used to form tunnels, except that we have yet to prove that $\SLip_N$ is a Lip-norm --- a fact which will in fact rely on the next corollaries as well.

We start with:
\begin{corollary}\label{D-isometry-corollary}
Assume Hypothesis (\ref{completeness-hypothesis-2}). Let $n\in\N_N$. The map:
\begin{equation*}
\Xi_n : (d_j)_{j\in\N_N}\in\mathfrak{G}_N \rightarrow d_n \in \D_n
\end{equation*}
is a *-epimorphism. Moreover:
\begin{equation*}
\forall d_n\in\sa{\D_n}\quad \Lip^n(d_n) =\inf\left\{\SLip_N(d) : d\in\sa{\mathfrak{G}_N}, \Xi_n(d)=d_n\right\}\text{.}
\end{equation*}
Consequently, if $\varphi,\psi\in\StateSpace(\D_n)$ then:
\begin{equation*}
\Kantorovich{\Lip^n}(\varphi,\psi) = \Kantorovich{\SLip_N}(\varphi\circ\Xi_N,\psi\circ\Xi_N)\text{.}
\end{equation*}
\end{corollary}

\begin{proof}
Immediate from Lemma (\ref{surjection-lemma}).
\end{proof}

We also observe that we can describe the full state space of $\alg{G}_N$:
\begin{corollary}\label{bipolar-corollary}
The set $\co{\bigcup_{n\in\N_N}\Xi_N^\ast(\StateSpace(\D_n))}$ is $\StateSpace(\alg{G}_N)$.
\end{corollary}

\begin{proof}
The polar of $\bigcup_{n\in\N_N}\Xi^\ast(\StateSpace(\D_n))$ is:
\begin{equation*}
\left\{ (d_n)_{n\in\N_N} \in \alg{G}_N : \forall n \in \N_N \quad \|d_n\|_{\D_n} = 0 \right\} = \{ 0 \}
\end{equation*}
which gives our conclusion by an application of the bipolar theorem as in \cite[Proposition 5.4]{Fell88}.
\end{proof}


We continue to prove that the quotient seminorms of $\SLip_N$ by the projections $\Pi_n$ are given as expected:

\begin{corollary}\label{isometry-corollary}
Assume Hypothesis (\ref{completeness-hypothesis-2}). For all $n\in\N_N$ and $a\in\sa{\A_n}$, and for $\Pi_n = \pi_n\circ\Xi_n$, we have:
\begin{equation*}
\Lip_n(a) = \inf \{ \SLip_N(d) : d\in\mathfrak{L}_N, \Pi_n(d) = a \}\text{.}
\end{equation*} 
\end{corollary}

\begin{proof}
Let $\varepsilon > 0$. Let $d\in\D_n$ such that $\Lip_n(a)\leq \Lip^n(d)\leq\Lip_n(a)+\varepsilon$. By Lemma (\ref{surjection-lemma}), there exists $g\in\mathfrak{G}_N$ such that $\Xi_n(g)=d$ and $\Lip^n(d)\leq\SLip_N(g)\leq\Lip^n(d) +\varepsilon$. The Lemma is thus proven since $\Pi_N(g) = \pi_n\circ\Xi_n(g)$ while $\Lip_n(a)\leq\SLip_N(g)\leq\Lip_n(a)+2\varepsilon$ and $\varepsilon > 0$ is arbitrary.
\end{proof}

Last, we also have another isometry between {\Lqcms s}:
\begin{corollary}\label{tail-isometry-corollary}
Assume Hypothesis (\ref{completeness-hypothesis-2}). Let $K \leq N \in \N$ and let:
\begin{equation*}
T_N^K : (d_n)_{n\in\N_K} \in \alg{G}_K \longmapsto (d_n)_{n\in\N_N} \in \alg{G}_N\text{.}
\end{equation*}
The map $T_N^K$ is a *-epimorphism and the quotient seminorm of $\SLip_K$ for $T_N^K$ is $\SLip_N$. In particular, the dual map $(T_N^K)^\ast$ induces an isometry from $(\StateSpace(\alg{G}_N),\Kantorovich{\SLip_N})$ into $(\StateSpace(\alg{G}_K),\Kantorovich{\SLip_K})$.
\end{corollary}

\begin{proof}
By Lemma (\ref{surjection-lemma}), if $d\in\mathfrak{G}_N$ and for all $\varepsilon > 0$, there exists $f \in \mathfrak{G}_K$ such that $T_N^K(f) = d$, $\SLip_N(d)\leq\SLip_K(f)\leq \SLip_N(f)+\varepsilon$. This completes our proof.
\end{proof}

Now, although $\SLip_N$ is not yet proven to be a Lip-norm on $\alg{G}_N$, the metric $\Kantorovich{\SLip_N}$ on $\StateSpace(\alg{G}_N)$ is well-defined, and we have constructed various isometries for this metric in Corollary (\ref{D-isometry-corollary}), Corollary (\ref{isometry-corollary}) and Corollary (\ref{tail-isometry-corollary}). We thus can estimate the Hausdorff distance between the state spaces of the {\Lqcms s} $(\D_n,\Lip^n)$, for all $n\geq N$, in $(\StateSpace(\alg{G}_N),\Kantorovich{\SLip_N})$.

\begin{lemma}\label{Dn-distance-lemma}
We assume Hypothesis (\ref{completeness-hypothesis-2}). Then for all $n\geq N$:
\begin{equation*}
\Haus{\Kantorovich{\SLip_N}}(\Xi_n^\ast(\StateSpace(\D_n)),\Pi_n^\ast(\StateSpace(\A_{n+1}))) \leq 2\tunnellength{\tau_n}{\Lip_{n},\Lip_{n+1}}
\end{equation*}
and
\begin{equation*}
\Haus{\Kantorovich{\SLip_N}}(\Xi_n^\ast(\StateSpace(\D_n)),\Xi_{n+1}^\ast(\StateSpace(\D_{n+1}))) \leq 2\max\{\tunnellength{\tau_n}{\Lip_{n},\Lip_{n+1}},\tunnellength{\tau_{n+1}}{\Lip_{n+1},\Lip_{n+2}}\}\text{.}
\end{equation*}
\end{lemma}

\begin{proof}
Let $\varphi\in\StateSpace(\D_{n+1})$ and $\varepsilon > 0$. By Definition (\ref{tunnel-depth-def}) of the depth of a tunnel, and by Corollary (\ref{D-isometry-corollary}), there exists $\psi\in\StateSpace(\A_{n+1})$, $\eta\in\StateSpace(\A_{n+2})$ and $t\in[0,1]$ such that, if we set:
\begin{equation*}
\mu=t\psi\circ\pi_{n+1}+(1-t)\eta\circ\omega_{n+1} \in\co{\pi_{n+1}^\ast(\StateSpace(\A_{n+1}))\cup\omega_{n+1}^\ast (\StateSpace(\A_{n+2}))}
\end{equation*}
then:
\begin{equation*}
\begin{split}
\Kantorovich{\SLip_N}(\Xi_{n+1}^\ast(\varphi),\Xi_{n+1}^\ast(\mu)) &= \Kantorovich{\Lip^{n+1}}(\varphi,\mu)\\
& \leq \varepsilon + \tunneldepth{\tau_{n+1}}{\Lip_{n+1},\Lip_{n+2}}\text{.}
\end{split}
\end{equation*}

Now, by Definition (\ref{tunnel-reach-def}) of the reach of a tunnel, we have:
\begin{equation*}
\Haus{\Kantorovich{\Lip^{n+1}}}\left(\StateSpace(\A_{n+1}),\StateSpace(\A_{n+2})\right) \leq \tunnelreach{\tau_{n+1}}{}\text{.}
\end{equation*}

So there exists $\theta \in \StateSpace(\A_{n+1})$ such that:
\begin{equation*}
\Kantorovich{\Lip^{n+1}}(\theta\circ\pi_{n+1},\eta\circ\omega_{n+1}) \leq \tunnellength{\tau_{n+1}}{\Lip_{n+1},\Lip_{n+2}}\text{.}
\end{equation*}

We set $\nu = t\psi+(1-t)\theta\in\StateSpace(\A_{n+1})$, and we then observe that:
\begin{equation*}
\begin{split}
\Kantorovich{\SLip_N}&(\Xi_{n+1}^\ast(\varphi), \Pi^\ast_{n+1}(\nu))\\
&= \Kantorovich{\SLip_N}(\Xi_{n+1}^\ast(\varphi), \Pi^\ast_{n+1}(t\psi+(1-t)\theta)) \\
&= \Kantorovich{\Lip^{n+1}}(\varphi, \pi^\ast_{n+1}(t\psi+(1-t)\theta)) \text{ by Corollary (\ref{D-isometry-corollary})}\\
&\leq \Kantorovich{\Lip^{n+1}}(\varphi,t\pi_{n+1}^\ast(\psi)+(1-t)\omega_{n+1}^\ast(\eta)) \\
&\quad + \Kantorovich{\Lip^{n+1}}(t\pi_{n+1}^\ast(\psi)+(1-t)\omega_{n+1}^\ast(\eta),t\pi_{n+1}^\ast(\psi)+(1-t)\pi_{n+1}^\ast(\theta))\\
&\leq  \tunnellength{\tau_{n+1}}{} + \\
&\quad \sup\left\{\left|t\psi(\pi_{n+1}(d))+(1-t)\eta(\omega_{n+1}(d))\right.\right.\\
&\quad \left.\left. - t\psi(\pi_{n+1}(d))-(1-t)\theta(\pi_{n+1}(d))\right| : d\in\sa{\D_{n+1}}, \Lip^{n+1}(d)\leq 1\right\}\\
&=  \tunnellength{\tau_{n+1}}{} + \sup\left\{(1-t)\left|\eta(\omega_{n+1}(d))-\theta(\pi_{n+1}(d))\right| : d\in\sa{\D_{n+1}}, \Lip^{n+1}(d)\leq 1\right\}\\
&= \tunnellength{\tau_{n+1}}{} + (1-t)\Kantorovich{\Lip^{n+1}}(\omega_{n+1}^\ast(\eta),\pi_{n+1}^\ast(\theta))\\
&\leq 2\tunnellength{\tau_{n+1}}{\Lip_{n+1},\Lip_{n+2}}\text{.}
\end{split}
\end{equation*}
Thus we have proven:
\begin{equation*}
\Kantorovich{\SLip_N}\left(\Xi_{n+1}^\ast(\varphi),\Pi_{n+1}^\ast(\StateSpace(\A_{n+1}))\right) \leq 2\tunnellength{\tau_{n+1}}{}\text{.}
\end{equation*}

Last, by definition of $\mathfrak{G}_N$, we have:
\begin{equation*}
\Pi_{n+1}^\ast(\nu)=\Xi_{n+1}^\ast\circ\pi^\ast_{n+1}(\nu) = \Xi_n^\ast\circ\omega_n^\ast(\nu) \in \Xi_n^\ast(\StateSpace(\D_n))\text{,}
\end{equation*}
since $\pi_{n+1}\circ\Xi_{n+1} = \omega_n\circ\Xi_n$ by Hypothesis (\ref{completeness-hypothesis-1}). In conclusion:
\begin{equation*}
\Kantorovich{\SLip_N}\left(\Xi_{n+1}^\ast(\varphi),\Xi_{n}^\ast(\StateSpace(\D_{n}))\right) \leq 2\tunnellength{\tau_{n+1}}{}\text{.}
\end{equation*}
The proof is similar when the roles of $\D_n$ and $\D_{n+1}$ are exchanged. This concludes our lemma.
\end{proof}

In particular, under Hypothesis (\ref{completeness-hypothesis-2}), the sequence $(\Xi_n^\ast(\StateSpace(\D_n)))_{n\in\N_N}$ is Cauchy for the Hausdorff distance associated with $\Kantorovich{\SLip_N}$. However, without additional information on $\Kantorovich{\SLip_N}$, we can not conclude about the convergence of this sequence.

We now prove, with the next few lemmas, that $\SLip_N$ are Lip-norms for all $N\in\N$. The first matter to address is to prove that $\Kantorovich{\SLip_N}$ is bounded for all $N\in\N$. To this end, we work first with:

\begin{notation}
For all $n\in\N$, let:
\begin{equation*}
\alg{H}_n = \left\{ (d_j)_{j=0}^n \in \oplus_{j=0}^n \D_j : \forall j\in\{0,\ldots,n-1\} \quad \omega_j(d_j) = \pi_{j+1}(d_{j+1}) \right\}\text{,}
\end{equation*}
and let:
\begin{equation*}
\mathsf{H}_n((d_j)_{j=0}^n) = \max\{ \Lip^j(d_j) : j\in\{0,\ldots,n\} \}\text{.}
\end{equation*}
\end{notation}

The key observation regarding the {\Lqcms s} $(\alg{H}_n,\MLip_n)$ $(n\in\N)$ is given by:

\begin{lemma}\label{uniform-bound-lemma}
Assume Hypothesis (\ref{completeness-hypothesis-2}). Let $T_N : (d_n)_{n\in\N}\in\mathfrak{G}_0 \mapsto (d_n)_{n\in \{0,\ldots,N\}}\in\alg{H}_N$. Then $T_N$ is a *-epimorphism such that the quotient of $\SLip_0$ by $T_N$ is $\MLip_N$. 

In particular, $T_N^\ast$ induces an isometry from $(\StateSpace(\alg{H}_N),\Kantorovich{\MLip_N})$ into $(\StateSpace(\alg{G}_0),\Kantorovich{\SLip_0})$.

There exists $k \geq 0$ such that for all $n\in\N$, we have:
\begin{equation*}
\diam{\StateSpace(\alg{H}_n)}{\Kantorovich{\mathsf{H}_n}} \leq k \text{.}
\end{equation*}

Consequently, $(\alg{H}_n,\MLip_n)$ is a {\Lqcms} for all $n\in\N$.
\end{lemma}

\begin{proof}
We apply Lemma (\ref{surjection-lemma}) to prove the properties of $T_N$.

Fix $n\in\N$. Now, by using the same method as in \cite[Lemma 12.4,Corollary 12.5]{Rieffel00}, we have:
\begin{equation*}
\diam{\StateSpace(\alg{H}_n)}{\Kantorovich{\mathsf{H}_n}} \leq \diam{\StateSpace
(\D_0)}{\Kantorovich{\Lip_0}} + 2\sum_{j=0}^{n-1} \Haus{\Kantorovich{\Lip}}(\StateSpace(\D_j),\StateSpace(\D_{j+1})) \text{.}
\end{equation*}
By Lemma (\ref{Dn-distance-lemma}), we conclude:
\begin{equation*}
\begin{split}
\diam{\StateSpace(\alg{H}_n)}{\Kantorovich{\mathsf{H}_n}} &\leq \diam{\StateSpace(\D_0)}{\Kantorovich{\Lip_0}} + 4 \sum_{j=0}^{n-1} \tunnellength{\tau_j}{\Lip_j,\Lip_{j+1}}\\
&\leq \diam{\StateSpace(\D_0)}{\Kantorovich{\Lip_0}} + 4 M_0 \text{,}
\end{split}
\end{equation*}
where $M_0 = \sum_{j=0}^\infty \tunnellength{\tau_j}{}$ was introduced in Hypothesis (\ref{completeness-hypothesis-2}). 
We conclude by setting $k = \diam{\StateSpace(\D_0)}{\Kantorovich{\Lip_0}} +4 M_0$.

Last, we note that $\MLip_n$ is the supremum of lower semi-continuous Leibniz seminorms and thus is, itself, a lower semi-continuous Leibniz seminorm. The domain of $\MLip_n$ is trivially dense in $\sa{\alg{H}_n}$ and its kernel is reduced to the scalar multiple of the identity for the same reason as established in Lemma (\ref{lipschitz-pair-lemma}) for $\SLip_N$.

Last, $\diam{\StateSpace(\alg{H}_n)}{\Kantorovich{\MLip_n}}\leq k$ and:
\begin{equation*}
\left\{ a \in \sa{\alg{H}_n} : \MLip_n(a)\leq 1, \|a\|_{\alg{H}_n}\leq 1\right\} \subseteq \prod_{j=0}^n \left\{ d \in \sa{\D_j} : \Lip^j(d)\leq 1, \|d\|_{\D_j}\leq 1\right\}
\end{equation*}
and by Theorem (\ref{az-thm}), the lower semi-continuity of $\Lip^j$ for all $j\in\N$, and Tychonoff's theorem, the set on the right hand side is compact. Hence by Theorem (\ref{az-thm}) again, $\MLip_n$ is a Lip-norm on $\alg{H}_n$.
\end{proof}

We thus get, in a manner similar to \cite[Lemma 12.6]{Rieffel00}:
\begin{lemma}\label{norm-bound-lemma}
Assume Hypothesis (\ref{completeness-hypothesis-2}). Let $k$ be given by Lemma (\ref{uniform-bound-lemma}). For any $d\in\mathfrak{G}_0$ with $\SLip_0(d)\leq 1$, there exists $t\in\R$ such that $\|d-t\unit_{\mathfrak{G}}\|_\infty \leq k$.
\end{lemma}

\begin{proof}
Let $d\in\mathfrak{G}_0$ with $\SLip_0(d)\leq 1$. Let $n\in\N$. Let:
\begin{equation*}
\Omega_n : (d_j)_{j\in\N}\in \mathfrak{G}_0 \mapsto (d_j)_{j\leq n}\in\alg{H}_n
\end{equation*}
for all $n\in\N$. 

For any $n\in\N$, set $b_n = \Omega_n(d)$, and note that there exists $t_n\in\R$ such that:
\begin{equation*}
\|d_n-t_n\unit_{B_n}\|_{\alg{H}_n} \leq \diam{\StateSpace(\alg{H}_n)}{\MLip_n} \leq k
\end{equation*}
by \cite[Proposition 2.2]{Rieffel99} and Lemma (\ref{uniform-bound-lemma}).

 Let $B_n = \{ t \in \R : \|b_n - t\unit_n\|_{\alg{H}_n} \leq k \}$. Then $B_n$ is nonempty, closed and bounded in $\R$. Moreover, if $t\in B_{n+1}$ then $\|b_n-t\unit_{\alg{H}_n}\|_{\alg{H}_n} \leq \|b_{n+1}-t\unit_{\alg{H}_{n+1}}\|_{\alg{H}_{n+1}} \leq k$. Thus $B_{n+1}\subseteq B_n$ for all $n\in\N$. Consequently, $\bigcap_{n\in\N}B_n\not=\emptyset$. Let $t\in \bigcap_{n\in\N}B_n$. By construction, we have $\|b_n - t\unit_{\alg{H}_n}\|_{\D_n}\leq k$ for all $n\in\N$, and thus $\|d-t\unit_{\mathfrak{G}_0}\|_{\mathfrak{G}_0} \leq k$ as desired.
\end{proof}

\begin{lemma}\label{finite-diameter-lemma}
Assume Hypothesis (\ref{completeness-hypothesis-2}). $(\StateSpace(\mathfrak{G}_N),\Kantorovich{\SLip_N})$ has finite diameter.
\end{lemma}

\begin{proof}
Since $(\StateSpace(\alg{G}_N),\Kantorovich{\SLip_N})$ is isometric to a subset of $(\StateSpace(\alg{G}_0,\Kantorovich{\SLip_0})$ by Corollary (\ref{tail-isometry-corollary}), it is sufficient to show our lemma for $\alg{G}_0$. Let $k$ be given by Lemma (\ref{uniform-bound-lemma}).

Let $d\in\alg{G}_0$ with $\SLip_0(d)\leq 1$, and let $\varphi,\psi \in \StateSpace(\alg{G}_0)$. Let $t\in\R$ be given by Lemma (\ref{norm-bound-lemma}):
\begin{equation*}
|\varphi(d) - \psi(d)| = |\varphi(d-t\unit_\D) - \psi(d- t \unit_\D)| \leq 2k\text{.}
\end{equation*}
This concludes our lemma.
\end{proof}

To conclude that $\SLip_N$ is a Lip-norm on $\alg{G}_N$, we wish to apply Theorem (\ref{az-thm}). Since we have shown in Lemma (\ref{finite-diameter-lemma}) that $\Kantorovich{\SLip_N}$ is bounded, the following Lemma is sufficient to apply Theorem (\ref{az-thm}):

\begin{lemma}\label{compact-lemma}
Assume Hypothesis (\ref{completeness-hypothesis-2}). Let $D = \diam{\StateSpace(\mathfrak{G}_0)}{\Kantorovich{\SLip_0}}$. The set:
\begin{equation*}
\mathfrak{l} = \left\{ d \in \mathfrak{L}_0 : \SLip_0(d) \leq 1 \text{ and } \|d\|_\infty\leq D \right\}
\end{equation*}
is norm-compact in $\mathfrak{G}_0$.
\end{lemma}

\begin{proof}
Let $d\in\mathfrak{l}$. Then by definition:
\begin{equation*}
d \in \prod_{n\in\N} \{ d_n \in \sa{\D_n} : \Lip^n(d_n) \leq 1 \text{ and }\|d_n\|_{\D_n} \leq D \}
\end{equation*}
and the latter set is compact by Tychonoff's theorem, since each factor is compact as $(\D_n,\Lip^n)$ is an {\Lqcms} for all $n\in\N$.
\end{proof}

We can now conclude:

\begin{proposition}\label{Lqcms-GN-proposition}
Assume Hypothesis (\ref{completeness-hypothesis-2}). $(\mathfrak{G}_N,\SLip_N)$ is a {\Lqcms}.
\end{proposition}

\begin{proof}
By Lemma (\ref{finite-diameter-lemma}), Lemma (\ref{compact-lemma}) and \cite[Theorem 1.9]{Rieffel98a}, we thus conclude that $(\mathfrak{G}_0,\SLip_0)$ is a compact quantum metric space. 
Moreover, $\mathfrak{G}_0$ is a unital C*-algebra, $\SLip_0$ is Leibniz as the supremum of Leibniz seminorms, and $\SLip_0$ is lower-semi-continuous as the supremum of lower semi-continuous seminorms. 

By Corollary (\ref{tail-isometry-corollary}), $(\mathfrak{G}_N,\SLip_N)$ is a quotient of $(\mathfrak{G}_0,\SLip_0)$ and thus it is also a compact quantum metric space by \cite[Proposition 3.1]{Rieffel00}. $\SLip_N$ is a lower semi-continuous Leibniz seminorm for the same reasons as $\SLip_0$ (note that $\SLip_N$ is closed by \cite[Proposition 3.3]{Rieffel00} as well).

Our proposition is thus proven.
\end{proof}

We now have established that our spaces $(\alg{G}_N,\SLip_N)$ are {\Lqcms s}, and thanks to Corollary (\ref{isometry-corollary}), we can use these spaces as part of tunnels from $(\A_n,\Lip_n)$, for $n\geq N$.

We now turn to the construction the prospective limit of the sequence $(\A_n,\Lip_n)_{n\in\N}$ for the dual Gromov-Hausdorff propinquity. The first step is to identify the state space of the prospective limit.

\begin{proposition}\label{Z-prop}
Assume Hypothesis (\ref{completeness-hypothesis-2}). The sequences $(\Pi_n^\ast(\StateSpace(\A_n)))_{n\in\N}$ and $(\Xi_n^\ast(\StateSpace(\D_n)))_{n\in\N_N}$ converges in $(\StateSpace(\mathfrak{G}_0),\Kantorovich{\SLip_0})$ to the same limit. We shall denote this common limit by $Z$.
\end{proposition}

\begin{proof}
By assumption and Corollary (\ref{isometry-corollary}):
\begin{equation*}
\Haus{\Kantorovich{\SLip_0}}\left(\Pi^\ast_n (\StateSpace(\A_n)),\Pi^\ast_{n+1}(\StateSpace(\A_{n+1}))\right)\leq \tunnellength{\tau_n}{\Lip_n,\Lip_{n+1}}
\end{equation*}
and by Hypothesis (\ref{completeness-hypothesis-2}), $(\sum \tau(n))_{n\in\N}$ is summable. Thus $(\Pi_n^\ast(\StateSpace(\A_n)))_{n\in\N}$ is a Cauchy sequence in the complete set of all weak* compact subsets of $\StateSpace(\alg{G}_0)$ for the Hausdorff distance associated with $\Kantorovich{\SLip_N}$ (which metrizes the weak* topology of the weak* compact $\StateSpace(\alg{G}_N)$ by Proposition (\ref{Lqcms-GN-proposition})).

Thus the sequence converges to some weak* compact subset $Z$ in $\StateSpace(\alg{G}_N)$. Now, for all $n\in\N_N$, we have:
\begin{equation*}
\begin{split}
\Haus{\Kantorovich{\SLip_N}}(\Xi_n^\ast(\StateSpace(\D_n)),Z) &\leq \Haus{\Kantorovich{\Lip^n}}(\StateSpace(\D_n)),\pi_n^\ast(\StateSpace(\A_n)))\\
&\quad +\Haus{\Kantorovich{\SLip_n}}(\Pi_n^\ast(\StateSpace(\A_n)),Z)\\
&\leq 2\tunnellength{\tau_n}{\Lip_n,\Lip_{n+1}} + \Haus{\Kantorovich{\SLip_N}}(\Pi_n^\ast(\StateSpace(\A_n)),Z) \text{ by Lemma (\ref{Dn-distance-lemma}),}
\end{split}
\end{equation*}
so our proposition is proven since $(\tunnellength{\tau_n}{\Lip_n,\Lip_{n+1}})_{n\in\N}$ converges to $0$.
\end{proof}

Using Kadison functional calculus \cite{Kadisson51}, we associate to every element $a\in\alg{G}_0$ a continuous affine function $\widehat{a}$ defined on the convex set $\StateSpace(\alg{G}_0)$, endowed with the weak* topology, by setting for all $\varphi\in\StateSpace(\alg{G}_0)$:
\begin{equation*}
\widehat{a}(\varphi) = \varphi(a)\text{.}
\end{equation*}
The quantum Gromov-Hausdorff limit for the sequence $(\A_n,\Lip_n)_{n\in\N}$ is the set:
\begin{equation*}
\alg{P}_0 = \left\{ \left.\widehat{a}\right|_Z : a\in\alg{G}_0 \right\}
\end{equation*} 
where $\left.\widehat{a}\right|_Z$ is the restriction of $\widehat{a}$ to $Z\subseteq\StateSpace(\alg{G}_0)$ for all $a\in\alg{G}_0$. Yet it is not clear, from such a description, that we can use this order-unit space to build our candidate for a unital C*-algebra, and that the quotient of $\SLip_0$ on this object is a closed Leibniz Lip-norm.

To construct a C*-algebra candidate for our limit, we begin with:

\begin{lemma}\label{F-lemma}
Assume Hypothesis (\ref{completeness-hypothesis-2}). Then:
\begin{equation*}
\{ d \in \mathfrak{G}_N : \forall\varphi\in Z\quad \varphi(d) = 0\} = \{ (d_n)_{n\in\N_N} \in \mathfrak{G}_N : \lim_{n\rightarrow\infty} \|d_n\|_{\D_n} = 0\}\text{.}
\end{equation*}

We note that the set:
\begin{equation*}
\alg{I}_N = \left\{ (d_n)_{n\in\N} \in \alg{G}_N : \lim_{n\rightarrow\infty} \|d_n\|_{\D_n} = 0 \right\}
\end{equation*}
is a closed two-sided ideal of $\alg{G}_N$.
\end{lemma}

\begin{proof}
Let $\varepsilon > 0$ and $d = (d_n)_{n\in\N} \in \alg{G}_0$, which we assume not zero without loss of generality, and such that $\varphi(d) = 0$ for all $\varphi\in Z$. By density, let $w=(w_n)_{n\in\N}\in\mathfrak{G}_0$ with $\SLip_0(w)<\infty$ and $\|d-w\|_\infty<\frac{1}{3}\varepsilon$. Note that $d\not\in\R\unit_{\alg{G}_0}$ and we can choose $w$ with $\SLip_0(w)>0$. By Proposition (\ref{Z-prop}), There exists $N\in\N$ such that for all $n\geq N$, we have:
\begin{equation*}
\Haus{\Kantorovich{\SLip_0}}(\Xi_n^\ast(\StateSpace(\D_n)),Z) \leq \frac{1}{3\SLip_0(w)}\varepsilon\text{.}
\end{equation*}
Let $\varphi\in\StateSpace(\D_n)$ for some $n\geq N$. There exists $\psi\in Z$ such that $\Kantorovich{\SLip_0}(\varphi\circ\Xi_n,\psi)\leq \frac{1}{3\SLip_0(w)}\varepsilon$. Now:
\begin{equation*}
\begin{split}
|\varphi(d_n)|&\leq|\varphi(d_n)-\varphi(w_n)|+|\varphi(w_n)|\\
&\leq \frac{1}{3}\varepsilon + |\varphi(w_n)-\psi(w)| +|\psi(w)|\\
&= \frac{1}{3}\varepsilon + |\varphi\circ\Xi_n(w)-\psi(w)| + |\psi(w)|\\
&\leq \frac{2}{3}\varepsilon + |\psi(w)-\psi(d)| \text{ as $\psi(d)=0$,}\\
&\leq \varepsilon\text{.}
\end{split}
\end{equation*}

Hence, for all $n\geq N$, we have shown $\|d_n\|_{\D_n}\leq \varepsilon$. Thus $\lim_{n\rightarrow\infty} \|d_n\|_{\D_n} = 0$.

Conversely, assume $d=(d_n)_{n\in\N}\in\alg{G}_N$ is chosen so that $\lim_{n\rightarrow\infty} \|d_n\|_{\D_n} = 0$.  Again without loss of generality, we assume $d\not=0$. Let $\varepsilon > 0$. 

There exists $w\in\alg{G}_N$ with $\SLip_N(w)<\infty$ and $\|d-w\|_{\alg{G}_N} < \frac{1}{4}\varepsilon$ since $(\alg{G}_N,\SLip_N)$ is a Lipschitz pair by Lemma (\ref{lipschitz-pair-lemma}). Since $d\not=0$ and $\lim_{n\rightarrow\infty}\|d_n\|_{\D_n}=0$, we have $d\not\in\R\unit_{\alg{G_N}}$ and thus we can pick $w\not\in\R\unit_{\alg{G_N}}$ since $\R\unit_{\alg{G}_N}$ is closed. Thus $\SLip(w)>0$.

There also exists $P\in\N$ such that for all $n\geq N$, we have $\|d_n\|_{\D_n}\leq\frac{\varepsilon}{4}$. Note that, writing $w=(w_n)_{n\in\N_N}$, we have, for all $n\geq P$, that $\|d_n-w_n\|_{\D_n}\leq\frac{1}{4}\varepsilon$, and thus $\|w_n\|_{\D_n}\leq \frac{1}{2}\varepsilon$.

Let $\psi\in Z$.  There exists $M\in\N$ such that for all $n\geq M$, there exists $\varphi\in\StateSpace(\D_n)$ such that:
\begin{equation}\label{F-lemma-eq0}
\Kantorovich{\SLip_0}(\varphi\circ\Xi_n,\psi)<\frac{1}{4\SLip_N(w)}\varepsilon
\end{equation}
by Proposition (\ref{Z-prop}). Thus, for $n=\max\{M,P\}$ and choosing $\varphi\in\StateSpace(\D_n)$ such that Inequality (\ref{F-lemma-eq0}) holds, we have:
\begin{equation*}
\begin{split}
|\psi(d)| &\leq |\psi(d)-\psi(w)|+|\psi(w)|\\
&\leq \frac{1}{4}\varepsilon + |\psi(w) -\varphi(w_n)| + |\varphi(w_n)|\\
&\leq \frac{1}{4}\varepsilon + \frac{1}{4 \SLip(w)} \varepsilon\SLip(w) + \|w_n\|_{\D_n}\\
& = \varepsilon \text{.}
\end{split}
\end{equation*}
As $\varepsilon > 0$ is arbitrary, we conclude that $\psi(d) = 0$ as desired.
\end{proof}

\newcommand{\QLip}{{\mathsf{Q}}}

We now have a candidate for our limit, as the quotient C*-algebra $\alg{F} = \bigslant{\alg{G}_0}{\alg{I}_0}$, where $\alg{I}_0$ is the ideal of $\alg{G}_0$ given in Lemma (\ref{F-lemma}).  We now must focus our attention on the Lip-norm we wish to endow it with. We start with:

\begin{lemma}\label{QLip-lemma}
Let $\QLip$ be the quotient seminorm of $\SLip_0$ on $\mathfrak{F} = \bigslant{\alg{G}_0}{\alg{I}_0}$, using the assumptions and notations of Lemma (\ref{F-lemma}). Let $q : \alg{G}_0 \twoheadrightarrow\alg{F}$ be the canonical surjection. We define:
\begin{equation*}
\forall a \in \sa{\alg{F}}\quad \QLip(a) = \inf\left\{ \SLip_0(d) : d\in\sa{\alg{G}_0}\text{ and }q(d) = a \right\}\text{.}
\end{equation*}
The seminorm $\QLip$ is a lower semi-continuous Lip-norm on $\alg{F}$.
\end{lemma}

\begin{proof}
Since $\SLip_0$ is a closed Lip-norm, \cite[Proposition 3.1, Proposition 3.3]{Rieffel00} shows that $\QLip$ is a closed Lip-norm on $\alg{F}$. 
\end{proof}

We have yet to prove that $\QLip$, as defined in Lemma (\ref{QLip-lemma}), has the Leibniz property. While in general, the quotient of a Leibniz Lip-norm is not Leibniz, there is a natural condition to assure that the Leibniz property is inherited by quotient seminorms, which Rieffel called compatibility \cite[Definition 5.1]{Rieffel10c}. However, Rieffel's notion is too strong here. Instead, we are going to use a form of ``asymptotic'' compatibility, by using Proposition (\ref{tunnel-lift-bound-prop}) together with all the tunnels we have built so far --- one for each $N\in\N$, in the notations of Hypothesis (\ref{completeness-hypothesis-1}). We begin by observing that our quotient Lip-norm does not change, of course, if we work with any truncated subsequence:

\begin{lemma}\label{multi-quotient-lemma}
Let $T_N^0 : (d_n)_{n\in\N} \in \mathfrak{G}_0 \longmapsto (d_n)_{n\in \N_N} \in \mathfrak{G}_N$. Let:
\begin{equation*}
\alg{I}_N = \left\{ (d_n)_{n\in\N_N} \in \alg{G}_N : \lim_{n\rightarrow\infty} \|d_n\|_{\D_n} = 0 \right\}\text{.}
\end{equation*}
Then the induced map $Q_N$ defined by:
\begin{equation*}
\begin{array}{ccc}
\mathfrak{G} & \stackrel{T^0_N}{\longrightarrow} & \mathfrak{G}_N \\
\downarrow & & \downarrow \\
\mathfrak{F} & \stackrel{Q_N}{\longrightarrow} & \bigslant{\mathfrak{G}_N}{\mathfrak{I}_N}
\end{array}
\end{equation*}
is a *-isomorphism such that, for all $f\in\alg{F}$:
\begin{equation*}
\QLip(f) = \inf \left\{ \SLip_N(d) : q_N(d) = Q_N(f) \right\}\text{,}
\end{equation*}
where $q_N : \alg{G}_N\twoheadrightarrow\bigslant{\alg{G}_N}{\alg{I}_N}$ is the canonical surjection.
\end{lemma}

\begin{proof}
We note that $\alg{I}_N$ is a closed ideal of $\alg{G}_N$, so $\alg{F}_N = \bigslant{\alg{G}_N}{\alg{F}_N}$ is a unital C*-algebra. Now, $T_N^0$ maps $\mathfrak{I}$ onto $\mathfrak{I}_N$, and moreover if $f\in\mathfrak{I}_N$ and $d\in\alg{G}_0$ is chosen so that $T_N^0(d) =f$ then $d\in \alg{I}_0$ by definition of these ideals. Note that such a $d$ exists by Lemma (\ref{surjection-lemma}).

Hence, the map $Q_N$ is a well-defined *-isomorphism.

Let $f \in \mathfrak{F}$ with $\QLip(f)<\infty$ and write:
\begin{equation*}
\QLip_N(f) = \inf\left\{\SLip_N(d) : d\in\alg{G}_N,q_N(d) = f \right\}\text{.}
\end{equation*}

Let $\varepsilon > 0$. By definition of $\QLip$ in Lemma (\ref{QLip-lemma}), there exists $d\in\alg{G}_0$ such that that $q(d) = f$ and $\SLip_0(d)\leq\QLip(f)+\varepsilon$, hence $\SLip_N(T_N^0(d))\leq\QLip(f)+\varepsilon$. Hence $\QLip_N(f)\leq\QLip(f)$ as $\varepsilon > 0$ is arbitrary.

Now, let $d\in\alg{G}_N$ such that $q_N(d) = f$ and $\SLip_N(d)\leq \QLip_N(f) +\frac{1}{2}\varepsilon$. Then by Lemma (\ref{surjection-lemma}), there exists $g\in\alg{G}_0$ such that $T_N^0(g)=d$ and $\SLip_0(g)\leq\SLip_N(d)+\frac{1}{2}\varepsilon$. Thus:
\begin{equation*}
\QLip(f)\leq\SLip_0(f)\leq \QLip_N(f)+\varepsilon\text{.}
\end{equation*}
As $\varepsilon > 0$ is arbitrary, we conclude that $\QLip(f)=\QLip_N(f)$.
\end{proof}

\begin{remark}
Let $Z$ be the limit of $(\Pi_n^\ast(\StateSpace(\A_n)))_{n\in\N}$ for the Hausdorff distance in $(\StateSpace(\alg{G}_0),\Kantorovich{\SLip_0})$ as defined in Proposition (\ref{Z-prop}). Let $q:\alg{G}_0\twoheadrightarrow\alg{F}$ be the canonical surjection. By Lemma (\ref{F-lemma}), we have $q_0^\ast(\StateSpace(\alg{F})) = Z$, with the map $q_0^\ast$ being an isometry by construction. 

Now, by Lemma (\ref{multi-quotient-lemma}), we also have an isometry $q_N^\ast$ from $\StateSpace(\alg{F})$ onto some subset $Z_N$ of $\StateSpace(\alg{G}_N)$. With the same Lemma, we conclude that $Z$ is the isometric image of $Z_N$ by $T_N^{0\ast}$, and moreover since $T_N^{0\ast}$ is an isometry, we obtain that $Z_N$ is the limit of $(\StateSpace(\A_n))_{n\in\N_N}$ for the Hausdorff distance associated with $\Kantorovich{\SLip_N}$ in $\StateSpace(\alg{G}_N)$.
\end{remark}

Now, in order to apply Proposition (\ref{tunnel-lift-bound-prop}) --- and as a step toward proving our candidate for a limit is indeed the limit of our sequence --- we compute the length of the natural tunnels we have constructed. We note that $(\alg{F},\QLip)$ is not proven to be a {\Lqcms} --- a fact we shall indeed resolve using the following result --- so we can not call the quadruples $(\alg{G}_n,\SLip_n,\Pi_n,q_n)$ tunnels quite yet. However, in the sense of Remark ({\ref{generalized-tunnel-rmk}), we shall now see that such quadruples are generalized tunnels, whose length can be made arbitrarily small for $n$ large enough.

\begin{lemma}\label{tunnel-length-lemma}
Assume Hypothesis (\ref{completeness-hypothesis-2}) and let $\QLip$ be defined as in Lemma (\ref{QLip-lemma}). Let $\varepsilon > 0$. There exists $N\in\N$ such that for all $n\geq N$, and denoting by $q_n : \alg{G}_N\twoheadrightarrow\alg{F}$ the composition of the canonical surjection with the inverse of $Q_N$ defined in Lemma (\ref{multi-quotient-lemma}):
\begin{enumerate}
\item the quotient seminorms of $\SLip_n$ for $\Pi_n$ and $q_n$ are given, respectively, by $\Lip_n$ and $\QLip$,
\item we have:
\begin{equation*}
\Haus{\Kantorovich{\SLip_n}}\left(\Pi_n^\ast(\StateSpace(\A_n)),q_n^\ast(\StateSpace(\alg{F}))\right) = \Haus{\Kantorovich{\SLip_0}}\left(\Pi_n^\ast(\StateSpace(\A_n)),Z)\right) \leq \varepsilon\text{,}
\end{equation*}
\item we have:
\begin{equation*}
\Haus{\Kantorovich{\SLip_n}}\left(\StateSpace(\alg{G}_n),Z\right) \leq\varepsilon\text{,}
\end{equation*}
\item for all $f\in\sa{\alg{F}}$, there exists $d\in\sa{\alg{G}_n}$ such that:
\begin{equation*}
q_n(d) = f \text{ and }\QLip(f)\leq\SLip_n(d)\leq\QLip(f)+\varepsilon\text{ and }\|d\|_{\alg{G}_n}\leq \|f\|_{\alg{F}}+\varepsilon\text{.}
\end{equation*}
\end{enumerate}

Thus, in the terminology of Remark (\ref{generalized-tunnel-rmk}), for all $\varepsilon > 0$, there exists $N\in\N$ such that for all $n\geq N$, the quadruple $(\alg{G}_n,\SLip_n,\Pi_n,q_n)$ is a generalized tunnel from $(\A_n,\Lip_n)$ to $(\alg{F},\QLip)$ of length at most $\varepsilon$.
\end{lemma}

\begin{proof}
Let $\varepsilon > 0$. By Proposition (\ref{Z-prop}), there exists $N\in\N$ so that for all $n\geq N$:
\begin{equation*}
\Haus{\SLip_0}\left(\Xi_N^\ast(\StateSpace(\A_n)),Z\right)<\varepsilon\text{ and }\Haus{\SLip_0}\left(\Xi_N^\ast(\StateSpace(\D_n)),Z\right)<\varepsilon\text{.}
\end{equation*}

Let $n\geq N$. By Corollary (\ref{isometry-corollary}), Lemma (\ref{QLip-lemma}) and Lemma (\ref{multi-quotient-lemma}), the quadruple $\omega_n = (\alg{G}_n,\SLip_n,\Pi_n,q_n)$ is a generalized tunnel from $(\A_n,\Lip_n)$ to $(\alg{F},\QLip)$ --- in the sense that Assertion (1) of our Lemma holds.

Now, $T_N^{0\ast}$ is an isometry from $(\StateSpace(\alg{G}_N),\Kantorovich{\SLip_N})$ into $\StateSpace(\alg{G}_0),\SLip_0)$, and we conclude that:
\begin{equation*}
\Haus{\Kantorovich{\SLip_n}}\left(\Pi_n^\ast(\StateSpace(\A_n)),q_n^\ast(\StateSpace(\alg{F}))\right) = \Haus{\Kantorovich{\SLip_0}}\left(\Pi_n^\ast(\StateSpace(\A_n)),Z)\right) \leq \varepsilon\text{.}
\end{equation*}
Thus, the reach of $\omega_n$ is bounded above by $\varepsilon$.

Let $\varphi\in\StateSpace(\D_k)$ for some $k\geq n$. Then there exists $\psi \in Z = q_n^\ast(\StateSpace(\alg{F}))$ such that $\Kantorovich{\SLip_n}(\Pi_n^\ast(\varphi),\psi)\leq\varepsilon$ by Proposition (\ref{Z-prop}). Since $\Kantorovich{\SLip_n}$ is a convex metric, we conclude by Lemma (\ref{bipolar-corollary}) that:
\begin{equation*}
\Haus{\Kantorovich{\SLip_n}}\left(\StateSpace(\alg{G}_n),Z\right) \leq\varepsilon\text{.}
\end{equation*}
This implies, in turn, that the generalized tunnel $\omega_n$ has depth at most $\varepsilon$.

The rest of the Lemma now follows from the same argument as for Proposition (\ref{tunnel-lift-bound-prop}) --- as this proposition did not use the Leibniz property of tunnels, yet used all other estimates which we have now established for  $(\alg{G}_n,\SLip_n,\Pi_n,q_n)$.


%
\end{proof}

We can now conclude that our candidate for a limit is a {\Lqcms}.

\begin{lemma}\label{F-Lqcms-lemma}
Assume Hypothesis (\ref{completeness-hypothesis-2}) and let $\QLip$ be defined as in Lemma (\ref{QLip-lemma}). The seminorm $\QLip$ is a Leibniz seminorm on $\sa{\alg{F}}$. Thus $(\alg{F},\QLip)$ is a {\Lqcms}.
\end{lemma}

\begin{proof}
Let $f,g \in \alg{F}$. Let $\varepsilon > 0$. Let $N_f\in\N, x \in \sa{\alg{G}_{N_f}}$ be given by Lemma (\ref{tunnel-length-lemma}) for $f$, and $N_g\in\N,y\in\sa{\alg{G}_{N_g}}$ be given by the same Lemma for $g$. Up to replacing $x$ and $y$ by $T_{\max\{N_f,N_g\}}^{N_f}(f)$ and $T_{\max\{N_f,N_g\}}^{N_g}(g)$, we may assume $f,g \in \sa{\alg{G}_N}$ for $N=\max\{N_f,N_g\}$. 

Now:
\begin{equation*}
\begin{split}
\QLip(\Jordan{f}{g}) &\leq \SLip_N(\Jordan{x}{y}) \\
&\leq \SLip_N(x)\|y\|_{\alg{G}_N} + \|x\|_{\alg{G}_N}\SLip_N(x)\\
&\leq (\QLip(f)+\varepsilon)(\|g\|_{\alg{F}}+\varepsilon) + (\|f\|_{\alg{F}}+\varepsilon)(\QLip(g)+\varepsilon)\text{.}
\end{split}
\end{equation*}
As $\varepsilon >0$ is arbitrary, we conclude that:
\begin{equation*}
\QLip(\Jordan{f}{g})\leq \|f\|_{\alg{F}}\QLip(g) + \|g\|_{\alg{F}}\QLip(f)\text{.}
\end{equation*}
The proof for the Lie product is identical.

The conclusion of this lemma follows from Lemma (\ref{QLip-lemma}).
\end{proof}

\begin{corollary}
Assume Hypothesis (\ref{completeness-hypothesis-2}) and let $\QLip$ be defined as in Lemma (\ref{QLip-lemma}). Let $\varepsilon > 0$. There exists $N\in\N$ such that for all $n\geq N$, and denoting by $q_n : \alg{G}_N\twoheadrightarrow\alg{F}$ the composition of the canonical surjection with the inverse of $Q_N$ defined in Lemma (\ref{multi-quotient-lemma}), the quadruple $(\alg{G}_n,\SLip_n,\Pi_n,q_n)$ is a tunnel from $(\A_n,\Lip_n)$ to $(\alg{F},\QLip)$ of length at most $\varepsilon$.
\end{corollary}

\begin{proof}
This is simply restating Lemma (\ref{tunnel-length-lemma}), where we now can call $(\alg{G}_n,\SLip_n,\Pi_n,q_n)$ a tunnel since Lemma (\ref{F-Lqcms-lemma}) proves that $(\alg{F},\QLip)$ is a {\Lqcms}.
\end{proof}

We have now reached the conclusion of our efforts, rewarded with the desired result as follows. All the required work has been done to construct our limit and prove that it was indeed a {\Lqcms}, so we now have:

\begin{proposition}\label{convergence-prop}
The sequence $(\A_n,\Lip_n)_{n\in\N}$ converges to $(\mathfrak{F},\Lip_{\mathfrak{F}})$ for the dual Gromov-Hausdorff propinquity.
\end{proposition}

\begin{proof}
Let $\varepsilon > 0$. By Lemma (\ref{tunnel-length-lemma}), let $N\in \N$ such that for all $n\geq N$, we have that the tunnel $(\alg{G}_n,\SLip_n,\Pi_n,q_n)$ has length at most $\varepsilon$. Thus:
\begin{equation*}
\dpropinquity{}((\A_n,\Lip_n),(\alg{F},\QLip)) \leq \varepsilon\text{.}
\end{equation*}
This completes our proof since $(\alg{F},\QLip)$ is a {\Lqcms} by Lemma (\ref{F-Lqcms-lemma}).
\end{proof}

We now are able to prove that the dual Gromov-Hausdorff propinquity is complete. This result is the core feature of our new metric.

\begin{theorem}\label{completeness-thm}
The dual Gromov-Hausdorff propinquity $\dpropinquity{}$ is complete.
\end{theorem}

\begin{proof}
Let $(\A_n,\Lip_n)_{n\in\N}$ be a Cauchy sequence for $\dpropinquity{}$. Let $(\A_{k_n},\Lip_{k_n})_{n\in\N}$ be a subsequence of $(\A_n,\Lip_n)_{n\in\N}$ such that $\sum_{n\in\N} \dpropinquity{}(\A_{k_n},\A_{k_{n+1}}) < \infty$. 

For each $n\in\N$, let $\Upsilon_n$ be a journey from $(\A_{k_n},\Lip_{k_n})$ to $(\A_{k_{n+1}},\Lip_{k_{n+1}})$ whose length is no more that:
\begin{equation*}
\dpropinquity{}((\A_{k_n},\Lip_{k_n}),(\A_{k_{n+1}},\Lip_{k_{n+1}})) + 2^{-n-1}\text{.}
\end{equation*}
By construction, we note that the concatenation $\Upsilon_0\star\cdots\star\Upsilon_N$ is a subfamily of the concatenation $\Upsilon_0\star\cdots\star\Upsilon_K$ for all $K\geq N$. Hence, we can define, without ambiguity, the infinite concatenation $\star_{n\in\N}\Upsilon_n$ (which is not a journey itself).

Now, let us write:
\begin{equation*}
\star_{n\in\N}\Upsilon_n = \left( \B_n,\Lip'_n,\tau_n,\B_{n+1},\Lip'_{n+1} : n\in\N\right)\text{.}
\end{equation*}

It is easy to check that $(\B_n,\Lip_n')$ is itself a Cauchy sequence for $\dpropinquity{}$. Once again, we pick $f:\N\rightarrow\N$ strictly increasing such that such that $(\B_{f(n)},\Lip'_{f(n)})_{n\in\N}$ and $(\tau_{f(n)})_{n\in\N}$ satisfy Hypothesis (\ref{completeness-hypothesis-2}). 

Thus the sequence $(\B_{f(n)},\Lip'_{f(n)})_{n\in\N}$ converges for $\dpropinquity{}$ by Proposition (\ref{convergence-prop}). Consequently, $(\B_n,\Lip'_n)_{n\in\N}$ converges for $\dpropinquity{}$ as a Cauchy sequence with a convergent subsequence. 

Now, the sequence $(\A_{k_n},\Lip_{k_n})_{n\in\N}$ is a subsequence of $(\B_n,\Lip'_n)_{n\in\N}$ and thus, it converges as well. Thus, the original sequence $(\A_n,\Lip_n)_{n\in\N}$ is a Cauchy sequence with a convergent subsequence, and thus it itself converges for $\dpropinquity{}$. This completes our proof.
\end{proof}

We conclude with an interesting remark. In \cite{Rieffel09}, a new set of Lip-norms is computed to prove that the sequence $(\B_n,\Lip_n)_{n\in\N}$ described in Corollary (\ref{rieffel-corollary}) is Cauchy. These new Lip-norms are in fact of the type used to define the quantum propinquity in \cite{Latremoliere13}. Now, since our dual Gromov-Hausdorff propinquity is complete, the estimates in \cite{Rieffel09} actually prove Corollary (\ref{rieffel-corollary}) again, thanks to Theorem (\ref{completeness-thm}).

\bibliography{../thesis}

\end{document}